\documentclass[11pt]{article}

\usepackage[margin=1in]{geometry}
\usepackage{times}
\usepackage{microtype}

\usepackage{amsmath,amssymb,amsthm,amsfonts,mathtools}

\usepackage{graphicx}
\usepackage{xcolor}
\usepackage{booktabs}
\usepackage{enumitem}
\usepackage{url}

\usepackage[round,authoryear]{natbib}

\usepackage{algorithm}
\usepackage{algorithmic}
\usepackage{float}
\floatstyle{ruled}
\restylefloat{algorithm}

\usepackage[
  colorlinks=true,
  citecolor=blue,
  linkcolor=blue,
  urlcolor=blue
]{hyperref}
\usepackage{aliascnt}
\usepackage[capitalize,noabbrev]{cleveref}

\newenvironment{keywords}{\par\smallskip\noindent\textbf{Keywords:} }{\par\smallskip}

\DeclareMathOperator*{\argmin}{arg\,min}
\DeclareMathOperator*{\argmax}{arg\,max}
\DeclareMathOperator{\spanop}{span}
\DeclareMathOperator{\relint}{relint}
\DeclareMathOperator{\dir}{dir}
\DeclareMathOperator{\cone}{cone}

\DeclareMathOperator{\conv}{conv}

\DeclareMathOperator{\rank}{rank}
\DeclareMathOperator{\cl}{cl}

\newcommand{\RR}{\mathbb{R}}
\newcommand{\PP}{\mathbb{P}}
\newcommand{\EE}{\mathbb{E}}
\newcommand{\1}{\mathbf{1}}
\newcommand{\X}{\mathcal{X}}
\newcommand{\C}{\mathcal{C}}

\newcommand{\Dset}{\mathcal{D}}

\newcommand{\Wstar}{W_\star}
\newcommand{\dstar}{d^\star}
\newcommand{\disc}{\mathrm{disc}}

\newcommand{\dist}{\mathrm{dist}}

\theoremstyle{plain}
\newtheorem{theorem}{Theorem}

\makeatletter
\newcommand{\redeclaretheorem}[3]{% #1 envname, #2 printed singular, #3 printed plural
  \@ifundefined{#1}{}{%
    \expandafter\let\csname #1\endcsname\relax
    \expandafter\let\csname end#1\endcsname\relax
  }%
  \newaliascnt{#1}{theorem}%
  \newtheorem{#1}[#1]{#2}%
  \aliascntresetthe{#1}%
  \crefname{#1}{#2}{#3}%
  \Crefname{#1}{#2}{#3}%
}
\makeatother

\redeclaretheorem{definition}{Definition}{Definitions}
\redeclaretheorem{lemma}{Lemma}{Lemmas}
\redeclaretheorem{proposition}{Proposition}{Propositions}
\redeclaretheorem{corollary}{Corollary}{Corollaries}
\redeclaretheorem{claim}{Claim}{Claims}
\redeclaretheorem{example}{Example}{Examples}
\redeclaretheorem{assumption}{Assumption}{Assumptions}
\redeclaretheorem{property}{Property}{Properties}
\redeclaretheorem{remark}{Remark}{Remarks}

\title{Learning Decision-Sufficient Representations for Linear Optimization}

\author{%
  Yuhan Ye\\
  MIT\\
  \texttt{yyh03@mit.edu}
  \and
  Saurabh Amin\\
  MIT\\
  \texttt{amins@mit.edu}
  \and
  Asuman \"{O}zda\u{g}lar\\
  MIT\\
  \texttt{asuman@mit.edu}
}

\date{}

\begin{document}
\maketitle

\begin{abstract}
We study how to construct compressed datasets that suffice to recover optimal decisions in linear programs with an unknown cost vector $c$ lying in a prior set $\mathcal{C}$. Recent work by \citet{bennouna2025whatdata} gives an exact information-theoretic characterization of sufficient decision datasets (SDDs) via an intrinsic decision-relevant dimension $d^\star$, but it leaves open whether such measurements can be computed, verified, or learned efficiently. In this paper, we establish hardness results showing that computing $d^\star$ is NP-hard and deciding whether a dataset is globally sufficient is coNP-hard, thereby resolving an open problem posed in \cite{bennouna2026datainformativenesslinearoptimization}.

To address this worst-case intractability, we introduce \emph{pointwise sufficiency}, a relaxation that requires sufficiency for an individual cost vector. Under nondegeneracy, we provide a polynomial-time cutting-plane algorithm for constructing pointwise-sufficient decision datasets. In a data-driven regime with i.i.d.\ costs, we further propose a cumulative algorithm that aggregates decision-relevant directions across samples, yielding a stable compression scheme of size at most $d^\star$. This leads to a distribution-free PAC guarantee: with high probability over the training sample, the pointwise sufficiency failure probability on a fresh draw is at most $\tilde{O}(d^\star/n)$, and this rate is tight up to logarithmic factors.

Finally, we apply decision-sufficient representations to contextual linear optimization, obtaining compressed predictors with generalization bounds scaling as $\tilde{O}(\sqrt{d^\star/n})$ rather than $\tilde{O}(\sqrt{d/n})$, where $d$ is the ambient cost dimension.
\end{abstract}


\section{Introduction}
In many real-world decision problems, the optimization objective depends on parameters that are not directly observable and must be inferred from data. With the surge in available data, it has become common to use empirical evidence alongside contextual knowledge to support decision-making. Recent work of \citet{bennouna2025whatdata} settles the information-theoretic side of this question by characterizing exactly when a dataset is sufficient. Our focus is the next layer: the computational and statistical problem of finding or learning such datasets efficiently. Motivated by this, the paper studies the following question:

\smallskip
\noindent\emph{Which objective measurements actually matter for optimal decisions, and can such measurements be computed or learned with provable guarantees in polynomial time?}
\smallskip

We formulate the decision-making problem as a linear optimization with an unknown cost vector and a known feasible region.
The decision-maker solves
\begin{equation}
    \label{eq:lp}
    \min_{x\in\X}\; c^\top x,
\end{equation}
where $\X := \{x\in\RR^d : Ax=b,\ x\ge 0\}$ is a nonempty bounded polytope for some $A\in\RR^{m\times d}$ with full row rank
and $b\in\RR^m$.
The objective vector $c\in\RR^d$ is unknown, but it is known a priori to lie in an uncertainty set $\C\subset\RR^d$. Rather than observing $c$ directly, the decision-maker deploys a fixed dataset (measurement set) $\Dset=\{q_1,\dots,q_l\}\subset\RR^d$ chosen in advance and observes the corresponding inner products
$s(c;\Dset) := (q_1^\top c,\dots,q_l^\top c)\in\RR^l$, as linear measurements of the objective. Given observations $s\in\RR^l$, we restrict plausible cost vectors to the \emph{fiber}
$\C(\Dset,s) := \{c'\in\C : q_i^\top c' = s_i,\ i=1,\dots,l\}$.

Acquiring objective information is often the bottleneck: observing or predicting the full $d$-dimensional vector $c$ can be unnecessary when optimal decisions only depend on a few \emph{decision-relevant} directions. A central geometric message is that decision-making depends on $c$ only through a low-dimensional set of decision-relevant directions. In particular, for open convex $\C$, \citet{bennouna2025whatdata} provide an exact geometric characterization: a dataset is globally sufficient on $\C$ if and only if its span contains the decision-relevant subspace $\Wstar$ (formally introduced in \Cref{def:dir-xstar}), whose intrinsic dimension $d^\star$ is often much smaller than the ambient dimension $d$. However, turning this characterization into an efficient procedure appears difficult; their Algorithm 2 constructs a minimum-size global sufficient dataset by solving a mixed-integer program at each iteration. In a recent work \citep{bennouna2026datainformativenesslinearoptimization}, whether a basis of $\Wstar$ can be constructed by a polynomial-time algorithm is stated as an open problem.

In Section~\ref{sec:hardness}, we show that computing $d^\star$ is $\mathsf{NP}$-hard (Theorem~\ref{thm:np-hard-dir-dim}) and that even the weakest global sufficiency test---deciding whether the empty dataset $\Dset=\varnothing$ is globally sufficient---is $\mathsf{coNP}$-hard already when $\C$ is an open polyhedron specified in $H$-representation\footnote{An \emph{$H$-polyhedron} is specified by finitely many linear inequalities, e.g., $P=\{x\in\RR^d:Hx\le h\}$. A \emph{polytope} is a bounded polyhedron; we use \emph{$H$-polytope} when emphasizing an inequality representation of a polytope. An \emph{open polyhedron} is obtained by replacing non-strict inequalities by strict ones.} (Theorem~\ref{thm:conp-hard-global0}).
Under $\mathsf{P}\neq\mathsf{NP}$, this gives a negative answer to the open problem in \cite{bennouna2026datainformativenesslinearoptimization}:
it is computationally hard in general to construct a basis of $\Wstar$ and hence a minimum-size global sufficient dataset.
The hardness results establish a computational barrier for worst-case global sufficiency but leave open two important questions: (i) Can we efficiently construct sufficient datasets for individual cost instances? (ii) In a data-driven regime with typical costs, can we learn datasets with good average-case guarantees? 

We answer both affirmatively, demonstrating a computational-statistical gap between worst-case complexity and average-case learnability. Following the data-driven algorithm design paradigm \citep{gupta2017pac,gupta2020datadrivenalgdesign,balcan2020datadrivenalgdesignchapter}, we relax from \emph{global (uniform)} sufficiency over all $c\in\C$ to a \emph{distributional} notion: we assume costs are drawn i.i.d.\ from an unknown distribution $P_c$ supported on $\C$, and aim to learn a dataset $\hat \Dset$  that is sufficient for a fresh draw $c\sim P_c$ with high probability. We address this via two intermediate steps: (i) a tractable \emph{pointwise relaxation} for individual instances (answering question (i) above), and (ii) a cumulative learning algorithm with PAC guarantees (answering question (ii)). This leads to a sample-complexity question:

\smallskip
{\centering
\emph{How many training instances are needed to learn such a $\hat \Dset$ in polynomial time?}\par}
\smallskip

To obtain a polynomial-time learning algorithm at the level of realized cost vector samples, we first introduce
\emph{pointwise sufficiency}. Intuitively, for a realized cost vector $c$, we only ask that every cost vector still compatible with the observed measurements of $c$ continue to support a common optimal decision. Equivalently, the remaining fiber of plausible costs should stay inside a single optimality region. Formally, a dataset $\Dset$ is pointwise sufficient at $c$ if every
$c'\in\C(\Dset,s(c;\Dset))$ induces the same set of optimal decisions.
This relaxation is tractable because it only requires the fiber to lie within one optimality cone rather than distinguishing decisions across all $c$ in $\C$.
Under a standard nondegeneracy assumption on $\X$, checking pointwise sufficiency reduces to a geometric containment test that solves at most $d-m$ instances of the face-intersection (FI) subproblem for a fixed optimal basis. When $\C$ is a polytope given in explicit $H$-representation or an ellipsoid, each FI call can be solved in polynomial time, yielding an overall polynomial-time routine; see Property~\ref{prop:polytime}.

In Section~\ref{sec:cuttingplanes}, we develop a sequential cutting-plane algorithm (Algorithm~\ref{alg:pointwise-cp}) that discovers pointwise-sufficient datasets offline. The key algorithmic innovation is our facet-hit cutting-plane rule: when the containment test fails, we identify the first facet of the optimality cone encountered along the segment joining an interior point to the witness of failure. This ensures the queried facet normal is genuinely decision-relevant. A naive approach querying arbitrary violated constraints can fail, as we demonstrate via a counterexample in Example~\ref{ex:why-facet-hit}. Each query produces a direction linearly independent of previous ones, and the procedure terminates after at most $\dstar$ queries.

In Section~\ref{sec:ddr-compression}, in a data-driven regime with i.i.d. costs, we run the pointwise routine cumulatively with warm starts (Algorithm~\ref{alg:cumulative}) and expand the measurement set only on a small subset of ``hard'' instances. This yields a realizable stable compression scheme~\citep{hanneke2021stablecompression} with compression size at most $\dstar$.
Consequently, with probability at least $1-\delta$ over $n$ training samples, the pointwise-sufficiency failure probability on a fresh draw is at most $\widetilde{O}(d^\star/n)$ (Theorem~\ref{thm:certificate}).
 This fast rate exploits realizability: our algorithm achieves zero empirical sufficiency loss via the compression framework, contrasting with data-driven projection approaches \citep{balcan2024howmuchdata,sakaue2024datadrivenprojectionslp} that control objective value gaps and obtain characteristic $\widetilde{O}(1/\sqrt{n})$ rates via uniform convergence. The frameworks are complementary: ours certifies when projections preserve decisions; theirs bound how close projected values are to optimal.

In practice, many decision systems repeatedly solve linear programs over a fixed feasible polytope $\X$ while the cost vector $c\in\C$ varies.
Examples include (i)~\emph{repeated LP} with time-varying costs (routing, resource allocation),
(ii)~\emph{contextual linear optimization}, where one learns a predictor for $c$ and plugs it into the LP,
and (iii)~preference-based decision making, where feedback is limited to comparisons or other aggregate signals.
This motivates the two-stage pipeline below:

\begin{itemize}[leftmargin=1.2em,itemsep=0pt,topsep=2pt]
    \item \textbf{Stage I (Representation Discovery).}
    Learn $\Wstar$ from i.i.d.\ training instances via sequential offline linear queries, together with a
    distribution-free certificate on the probability of sufficiency failure.
    \item \textbf{Stage II (Task-Specific Deployment).}
    Use the learned subspace for dimension reduction in repeated LPs, contextual LPs, and preference-based decision making.
\end{itemize}

In Section~\ref{sec:clo_compression}, as an illustration, we integrate this framework into SPO$+$ training for contextual linear optimization by restricting to the discovered decision-relevant subspace. This reduces the dimension parameter appearing in the generalization bound (see \Cref{thm:misspec_plus_gen}). We summarize our main contributions as follows.

\begin{itemize}[leftmargin=1.2em,itemsep=0pt,topsep=2pt]
\item \textbf{Hardness results.}
Building on the SDD geometry of \citet{bennouna2025whatdata}, we show that computing the intrinsic decision-relevant
dimension $\dstar$ is $\mathsf{NP}$-hard (Theorem~\ref{thm:np-hard-dir-dim}).
We also prove that deciding whether the empty dataset $\Dset=\varnothing$ is globally sufficient (in the decision sense) is
$\mathsf{coNP}$-hard (Theorem~\ref{thm:conp-hard-global0}); for open convex $\C$, this yields that constructing a
minimum-size global SDD is $\mathsf{NP}$- and $\mathsf{coNP}$-hard (Corollary~\ref{cor:final-hardness}).
Finally, verifying pointwise sufficiency is $\mathsf{coNP}$-hard in general (Theorem~\ref{thm:coNP-hard-pointwise-check}).

\item \textbf{Per-instance level: finding pointwise SDDs in polynomial time.}
Under nondegeneracy, we design a facet-hit cutting-plane algorithm
(Algorithm~\ref{alg:pointwise-cp}) that sequentially queries decision-relevant directions. The facet-hit rule ensures each query is linearly independent and lies in $\Wstar$, returning a pointwise-sufficient dataset
(Theorem~\ref{thm:alg1-correct}) in polynomial time.

\item \textbf{Distributional level: learning decision-sufficient representations with fast rates.}
Warm-starting the pointwise routine over i.i.d.\ costs (Algorithm~\ref{alg:cumulative}) yields a realizable stable compression scheme~\citep{hanneke2021stablecompression} with compression size at most $\dstar$. This leads to a distribution-free PAC certificate scaling as $\widetilde{O}(d^\star/n)$ (Theorem~\ref{thm:certificate}), matching our lower bound  (Theorem~\ref{thm:lowerbound}). The fast rate exploits zero empirical loss and linear independence of decision-relevant queries.

\item \textbf{Application to contextual linear optimization.}
We integrate decision-sufficient representation into predictor training for contextual linear optimization, yielding a compressed prediction model with improved generalization guarantees. In particular, the dimension term in the bound is reduced from $d$ to $d^\star$ (\Cref{thm:misspec_plus_gen}).
\end{itemize}

\section{Related literature}
\label{sec:lit_rev}
\textbf{Dimension reduction and model compression for LP.}
Dimension reduction is a classical tool for coping with high-dimensional optimization and learning.
Random projections and sketching preserve geometry with high probability (e.g., Johnson--Lindenstrauss \citep{johnson1984extensions}; see \citep{woodruff2014sketching})
and are ubiquitous in numerical linear algebra as well as learning-theoretic analyses \citep{bartlett2022gendatanla}.
For linear programs, random projections can reduce problem size while approximately preserving feasibility and objective values
\citep{vu2018randomlp,poirion2023randomprojections}, and \citet{sakaue2024datadrivenprojectionslp} propose data-driven projections with generalization guarantees.
In contrast, we seek exact optimizer recovery over a known polytope $\X$ by identifying the directions that can change the optimal solution; we quantify the intrinsic decision dimension by $d^\star$.
Our work is most closely related to \citet{bennouna2025whatdata,bennouna2026datainformativenesslinearoptimization}, which characterizes global sufficient decision datasets for LP under convex open priors and gives an iterative construction with a mixed-integer program at each step.

\textbf{Data-driven algorithm design.}
Beyond worst-case analysis argues that worst-case complexity can be overly pessimistic and instead advocates structured or distributional models of ``relevant'' instances \citep{roughgarden2019beyondworstcase,roughgarden2020bwca}.
Representative examples include perturbation-resilient instances \citep{bilu2012stableinstances}, smoothed analysis \citep{spielman2004smoothedanalysis}, and planted or semi-random models \citep{blum1995coloringsemirandom}.
Data-driven algorithm design is a principled ML instantiation of this viewpoint, learning an algorithmic object (e.g., an algorithm family or configuration) from i.i.d.\ samples while retaining provable performance guarantees
\citep{gupta2020datadrivenalgdesign,balcan2020datadrivenalgdesignchapter,gupta2017pac,balcan2017learningtheoreticfoundations,balcan2024howmuchdata,balcan2024learningtobranch}.
Under bounded loss, typical results yield uniform-convergence generalization gaps on the order of $\tilde{O}(\sqrt{\mathrm{Pdim}(\mathcal A)/n})$, and sharper bounds can be obtained via refined complexity notions such as dispersion and ``knife-edge'' structure \citep{balcan2018dispersion,balcan2020refinedbounds}.
Our setting fits this paradigm by treating the queried directions as the learned object.
For our specific problem, we exploit LP geometry to derive a stable compression scheme, yielding fast-rate certificates scaling as $\widetilde{O}(d^\star/n)$.
At a technical level, our geometry-aware cutting-plane viewpoint is also reminiscent of oracle-based frameworks that access constraints through separation \citep{mhammedi25a}.

\textbf{Compression-based generalization.}
Compression-based analyses provide distribution-free generalization and scenario-type certificates, often yielding fast $k/n$-type rates when the learned object admits a small compression of size $k$
\citep{valiant1984learnable,littlestone1986relating,floyd1995samplecompression,moran2016samplecompression,graepel2005pacbayes}.
In the realizable regime, stable compression can further sharpen logarithmic factors \citep{bousquet2020properlearning,hanneke2021stablecompression}.
In fully agnostic settings, however, $1/n$ rates are impossible in general: sharp lower bounds show worst-case rates of order $\Theta\!\bigl(\sqrt{k\log(n/k)/n}\bigr)$ \citep{hanneke19b}.
Our cumulative algorithm in Section~\ref{sec:ddr-compression} fits naturally into this view: each newly queried direction is triggered by a ``hard'' sample revealing a genuinely new decision-relevant facet, producing a compression set of size at most $d^\star$ and a fast-rate certificate.

\textbf{Polyhedral containment.}
Our hardness results rely on classical complexity phenomena in polyhedral computation.
In particular, deciding containment of an $H$-polytope in a $V$-polytope is $\mathsf{coNP}$-complete \citep{freundorlin1985complexity,gritzmannklee1993radii}.
Sum-of-squares certificates provide powerful relaxations for containment questions and yield refined results for structured instances \citep{kellner2016sos}.
Since pointwise sufficiency can be phrased as a containment statement that a data-consistent slice lies within an optimality region, these results naturally connect to the verification problem studied in our paper.

\textbf{Contextual optimization.}
Our paper is also motivated by decision-focused (predict-then-optimize) learning, where one learns predictive models to support downstream optimization \citep{elmachtoubgrigas2022spo,elbalghiti2023generalization}; see also the survey \citep{sadana2025survey}.
Beyond the batch setting, recent work studies online and active learning variants of contextual linear optimization, including margin-based active learning \citep{liu2025activeclo} and online contextual decision-making with SPO-type surrogates \citep{liu2022onlinep2o}; for background on active learning, see \citet{dasgupta2011twofaces,hanneke2014active}.
Bandit/partial-feedback formulations are studied in \citet{hu2024clobandit}, and related safe-exploration objectives appear in safe linear bandits over unknown polytopes \citep{gangrade24a}.
On the statistical side, \citet{elbalghiti2023generalization} derive uniform generalization bounds for the SPO loss via the Natarajan dimension of the induced decision class; in the polyhedral case, their bounds depend only logarithmically on the number of extreme points, yielding rates on the order of $\tilde{O}\!\left(\sqrt{pd\,\log(n|\X^\angle|)/n}\right)$ for linear predictors.
Recent work addresses model misspecification in contextual optimization \citep{bennouna2025misspecification} and benign generalization behavior in stochastic linear optimization under quadratically bounded losses \citep{telgarsky22a}; see also \citet{schutte2025sufficientproxies} on sufficient proxy representations.

\textbf{Active learning and adaptive measurement.}
Active learning studies how to adaptively acquire information---labels, features, or more general tests---to reduce query cost relative to passive sampling.
A central idea is to query only where candidate hypotheses disagree, as formalized by disagreement-based active learning \citep{hanneke2014active}; see also \citet{dasgupta2011twofaces}.
Classical work distinguishes sample complexity from label (query) complexity in agnostic active learning \citep{balcan2006agnostic,balcan2009agnostic}.
Our per-instance model is closer in spirit to arbitrary-query and experimental-design views of active learning \citep{kulkarni1993active,cohn1996active}, since we design linear measurements $q^\top c$ of a possibly unknown $c$.
Frameworks accounting for heterogeneous query costs are also related to our measurement budget \citep{guillory2009average,guillory2010interactive}.
At the intersection with decision-focused learning, \citet{liu2025activeclo} studies margin-based active learning for contextual linear optimization.
In contrast, we leverage polyhedral LP geometry: our facet-hit rule queries a violated optimality-cone facet normal, guaranteeing exact decision identification after at most $d^\star$ measurements and enabling a stable-compression view across i.i.d.\ instances.

\section{Preliminaries: Characterizing sufficient datasets via LP geometry}
\label{sec:prelim}
Recall that the fundamental question we address is: \emph{which fixed measurement sets $\Dset$ are sufficient to solve the LP}? In other words, we study a measurement-design problem: when do the measurements produced by a single dataset $\Dset$, together with the prior restriction $c\in\C$, already determine an optimal decision? Following \citet{bennouna2025whatdata}, we adopt a \emph{global} notion of informativeness in which one fixed dataset $\Dset$ must work uniformly for all
$c\in\C$. We now formalize this notion.

\begin{definition}[Global sufficient decision dataset]
\label{def:bennouna-sdd}
A dataset $\Dset := \{q_1,\dots,q_l\}\subseteq \RR^d$ is a \textbf{sufficient decision dataset} (SDD) for
the uncertainty set $\C\subseteq \RR^d$ and decision set $\X$ if there exists a decision rule
$\widehat{X}:\RR^l\to \mathcal{P}(\X)$, where $\mathcal{P}(\X)$ denotes the collection of subsets of $\X$, such that
\[
\forall\, c\in\C,\qquad
\widehat{X}\bigl(c^\top q_1,\dots,c^\top q_l\bigr)
\;=\;
\argmin_{x\in\X} c^\top x.
\]
\end{definition}

To characterize sufficient datasets, we recall a few notions from LP geometry. Let $\X^\angle$ denote the set of \emph{extreme points} of $\X$.
For $x^\star\in\X^\angle$, define the feasible direction cone
$FD(x^\star):=\{\delta\in\RR^d:\exists\varepsilon>0,\ x^\star+\varepsilon\delta\in\X\}$
and define the \emph{optimality cone}
$\Lambda(x^\star):=\{c\in\RR^d:\ x^\star\in\argmin_{x\in\X} c^\top x\}$.

Let $D(x^\star)$ be the set of \emph{extreme directions} of the polyhedral cone $FD(x^\star)$. For $\delta\in D(x^\star)$, define the corresponding boundary face $F(x^\star,\delta):=\Lambda(x^\star)\cap \{\delta\}^\perp=\{c\in\Lambda(x^\star):\ c^\top\delta=0\}$, and the set of \emph{relevant extreme directions}
\begin{equation}
\label{def:Delta}
\Delta(\X,\C)
:=\Bigl\{\delta\in\RR^d:\ \exists x^\star\in\X^\angle,\ \delta\in D(x^\star),\ \text{and}\ F(x^\star,\delta)\cap \C\neq\emptyset\Bigr\}.
\end{equation}
Equivalently, $\delta$ is relevant if the corresponding boundary face between optimality regions is attainable by some cost in $\C$. 
For any cost vector $c$, let $\X^\star(c):=\argmin_{x\in\X} c^\top x$ denote the (possibly set-valued) set of optimal solutions.
For a set $\C\subseteq\RR^d$, define
\begin{equation}
\label{def:dir-xstar}
\begin{aligned}
& \X^\star(\C):=\bigcup_{c\in\C}\bigl(\X^\star(c)\cap \X^\angle\bigr),
&\dir(\X^\star(\C)):=\spanop\{x-x':\ x,x'\in\X^\star(\C)\}.
\end{aligned}
\end{equation}

We refer to $\Wstar:=\dir(\X^\star(\C))$ as the \emph{decision-relevant subspace}, and set $\dstar:=\dim(\Wstar)$.
The following two results are due to \citet{bennouna2025whatdata}. The first theorem states that global sufficiency is
equivalent to spanning all directions along which optimality can change, as formalized by $\Delta(\X,\C)$.

\begin{theorem}[\citet{bennouna2025whatdata}, Theorem 1]
\label{thm:bennouna-delta}
Let $\C$ be open and convex. A dataset $\Dset$ is an SDD for $(\X,\C)$
if and only if
\[
\Delta(\X,\C)\subseteq \spanop(\Dset).
\]
\end{theorem}

While $\Delta(\X,\C)$ is defined via faces of optimality cones, the next theorem gives an equivalent characterization
directly in decision space: the span of these relevant directions matches the span of \emph{differences of reachable optima}.

\begin{theorem}[\citet{bennouna2025whatdata}, Theorem 2]
\label{thm:bennouna-spanDelta}
For any convex set $\C\subset\RR^d$,
\[
\spanop\Delta(\X,\C)=\dir(\X^\star(\C)).
\]
\end{theorem}

Combining \Cref{thm:bennouna-delta,thm:bennouna-spanDelta} yields the following subspace characterization,
which we will use throughout.

\begin{corollary}[Subspace characterization \citet{bennouna2025whatdata}, Corollary 1]
\label{cor:dir-char}
Let $\C$ be open and convex. A dataset $\Dset$ is an SDD for $(\X,\C)$ if and only if
\[
\dir(\X^\star(\C)) \subseteq \spanop(\Dset),
\]
hence, the minimal size of global SDD is $d^\star$.
\end{corollary}

In particular, the necessity direction of \Cref{thm:bennouna-delta} uses that $\C$ is open.
Accordingly, any argument or algorithm that invokes this direction requires an openness assumption on $\C$. In contrast, \Cref{thm:bennouna-spanDelta} requires only convexity and applies to both open and closed convex sets.\footnote{This distinction is important: our pointwise-sufficiency algorithms (\Cref{sec:cuttingplanes}) leverage \Cref{thm:bennouna-spanDelta} under a closed convex prior, while the global SDD size characterization requires $\C$ to be open and convex.}

\section{Hardness and Relaxation}
\label{sec:hardness}
\subsection{Computational hardness.}
The geometric characterizations in Section~\ref{sec:prelim} are information-theoretic in nature: they identify the subspace that any global sufficient dataset must capture.
Algorithmically, a natural goal is to compute the minimum size of a global SDD and to construct one.
By \Cref{cor:dir-char}, when $\C$ is open and convex, this minimum size equals the intrinsic decision-relevant dimension $d^\star=\dim\dir\!\bigl(\X^\star(\C)\bigr)$.
We state informal versions of our hardness results below and provide brief proof sketches to highlight the structure of the reductions; full formal statements and complete proofs are deferred to Appendix~\ref{app:hardness}.
Our reductions construct highly structured instances: shortest-path flow polytopes with budgeted arc-length perturbations.

\begin{theorem}[Informal]
\label{thm:np-hard-dir-dim}
It is $\mathsf{NP}$-hard to compute the intrinsic decision-relevant dimension $d^\star$,
given as input a polytope $\X\subseteq\RR^d$ and a polyhedral uncertainty set $\C$ specified in $H$-representation.
This hardness persists whether $\C$ is given as a closed polyhedron or as an open polyhedron.
\end{theorem}

\paragraph{Proof sketch for Theorem~\ref{thm:np-hard-dir-dim}.}
We start from 3-SAT and use the PISPP-W+ construction of~\cite[Theorem~3.1]{ley2025solutionmethods}.
Given a formula $\varphi$, the reduction outputs a directed acyclic graph (DAG) together with baseline arc-lengths $d$,
a budget vector $\kappa$, a budget $B$, and a required arc $r$, defining a budgeted uncertainty set of admissible length vectors
$c=d+w$ (with $w\in\mathbb{Q}^A_{\ge 0}$ and $\kappa^\top w \le B$). The formula $\varphi$ is satisfiable if and only if there exists
such a modification $w$ for which some shortest $s$--$t$ path with respect to $d+w$ contains the required arc $r$. We then show that deciding the feasibility of the resulting PISPP-W+ instance reduces to checking whether
$\dim\dir(X^\star(\mathcal{C})) > \dim\dir(X_r^\star(\mathcal{C}))$, where $X_r := \{x\in X : x_r=0\}$; see \eqref{eq:np-hard-compare}.
Geometrically, restricting to the face $x_r=0$ removes a decision-relevant extreme direction if and only if there exists a feasible
modification that makes some shortest $s$--$t$ path use $r$.
For the open set $\mathcal{C}_{\mathrm{op}}$, \Cref{lem:openify-pispp}  “opens” the budget via a small topological perturbation without changing the
shortest-path structure and thus maintaining the answer. \hfill $\square$

\smallskip

When $\C$ is open and convex (in particular, for the open polyhedral instances in \Cref{thm:np-hard-dir-dim}), combining \Cref{cor:dir-char} and \Cref{thm:np-hard-dir-dim} implies that both computing the size of a minimum global SDD and constructing a minimum global SDD are $\mathsf{NP}$-hard.

\subsection{A relaxation: pointwise sufficient datasets}
The computational hardness of finding minimum global SDDs motivates us to relax the concept to an instance-wise notion. Instead of asking one dataset to work uniformly for every $c\in\C$, we ask whether it suffices for the particular realized cost vector currently being faced. This is the right granularity for our later learning problem, where the training data consist of i.i.d. cost vectors and the loss is evaluated samplewise.
\begin{definition}[Pointwise sufficient decision dataset]
\label{def:pointwise-sdd}
Fix a (possibly unknown) $c\in\C$.
A finite query dataset $\Dset$ is \textbf{pointwise sufficient at $c$} if there exists a decision $x^\star\in\X$ such that
\[
x^\star \in \X^\star(c')\qquad \forall\,c'\in\C\bigl(\Dset,s(c;\Dset)\bigr).
\]
Equivalently, the data-consistent fiber $\C(\Dset,s(c;\Dset))$ is contained in a single optimality region of $\X$.
\end{definition}

\begin{remark}
Although the definition is indexed by a cost vector $c$, the algorithmic test below will turn it into a cone-containment question for an optimal basis at $c$. We formulate the notion at a cost vector because later the statistical problem is posed over i.i.d. draws $c_1,\dots,c_n$ and the loss is evaluated per realized instance.
\end{remark}

\begin{remark}
Since $c\in \C(\Dset,s(c;\Dset))$, Definition~\ref{def:pointwise-sdd} can equivalently be stated as:
$\Dset$ is pointwise sufficient at $c$ if there exists $x^\star\in \X^\star(c)$ such that $x^\star$ remains optimal for every
$c'\in\C$ that produces the same measurements, i.e., $s(c';\Dset)=s(c;\Dset)$.
This parallels the structure of Definition~\ref{def:bennouna-sdd} but applies to a single cost vector rather than all $c\in\C$.
\end{remark}

The following basic properties of pointwise sufficiency follow immediately from the definition.
\begin{property}
\label{prop:pointwise-basic}
\textbf{(i) Monotonicity.} If $\Dset$ is pointwise sufficient at $c$ and $\Dset'\supseteq \Dset$, then $\Dset'$ is also pointwise sufficient at $c$, since
$\C(\Dset',s(c;\Dset'))\subseteq \C(\Dset,s(c;\Dset))$. \textbf{(ii) Global $\Rightarrow$ pointwise.} If $\Dset$ is an SDD for $(\X,\C)$, then $\Dset$ is pointwise sufficient at every $c\in\C$.
\end{property}

\begin{proof}[Proof of Property~\ref{prop:pointwise-basic}(i)]
Let $\Dset\subseteq\Dset'$ and fix $c\in\C$.
Write $s:=s(c;\Dset)$ and $s':=s(c;\Dset')$.
By the definition of the fiber,
\[
\C(\Dset',s')
= \{c'\in\C : q^\top c' = q^\top c\ \ \forall q\in\Dset'\}.
\]
Since $\Dset\subseteq\Dset'$, any $c'\in\C(\Dset',s')$ satisfies $q^\top c'=q^\top c$ for all $q\in\Dset$, and hence
$c'\in\C(\Dset,s)$.
Therefore $\C(\Dset',s')\subseteq\C(\Dset,s)$.

Now assume $\Dset$ is pointwise sufficient at $c$.
Then there exists $x^\star\in\X$ such that $x^\star\in\X^\star(c'')$ for all $c''\in\C(\Dset,s)$.
By the containment above, the same $x^\star$ is optimal for all $c''\in\C(\Dset',s')$, so $\Dset'$ is also pointwise sufficient at $c$.
\end{proof}

\begin{proof}[Proof of Property~\ref{prop:pointwise-basic}(ii)]
Let $\Dset$ be a global SDD for $(\X,\C)$ in the sense of \Cref{def:bennouna-sdd}.
By definition, there exists a mapping $\widehat X:\RR^{|\Dset|}\to\mathcal{P}(\X)$ such that for every $c\in\C$,
\[
\widehat X\bigl(s(c;\Dset)\bigr) \;=\; \X^\star(c).
\]
Fix any $c\in\C$ and write $s:=s(c;\Dset)$.
For any $c'\in\C(\Dset,s)$ we have $s(c';\Dset)=s$ by definition of the fiber, and therefore
\[
\widehat X(s)
\;=\;
\widehat X\bigl(s(c';\Dset)\bigr)
\;=\;
\X^\star(c').
\]
In particular, $\widehat X(s)$ is nonempty, and any choice of $x^\star\in\widehat X(s)$ satisfies
$x^\star\in\X^\star(c')$ for all $c'\in\C(\Dset,s)$.
Thus the single decision $x^\star$ is optimal for {all} costs in the fiber, meaning that $\Dset$ is pointwise
sufficient at $c$ in the sense of \Cref{def:pointwise-sdd}.
\end{proof}

A natural verification problem is: given $(\X,\C)$, a dataset $\Dset$, and a cost vector $c\in\C$, decide whether
$s(c;\Dset)$ already suffices to determine an optimal decision.
The next theorem shows that even the verification problem is intractable in full generality.

\begin{theorem}[Informal]
\label{thm:coNP-hard-pointwise-check}
It is $\mathsf{coNP}$-hard to decide, given a bounded polytope $\X\subseteq\RR^d$, a polyhedral uncertainty set $\C\subseteq\RR^d$ specified in $H$-representation,
a dataset $\Dset$, and a cost vector $c\in\C$, whether $\Dset$ is pointwise sufficient at $c$.
\end{theorem}

\paragraph{Proof sketch for Theorem~\ref{thm:coNP-hard-pointwise-check}.}
We reduce from the classical \emph{$H$-in-$V$ polytope containment} problem.\footnote{A \emph{$V$-polytope} is a polytope specified by its vertices (a vertex representation), e.g., \(Q=\operatorname{conv}\{v_1,\ldots,v_M\}\). The $H$-in-$V$ containment problem asks whether \(P\subseteq Q\) given \((H,h)\) and \(\{v_j\}_{j=1}^M\).}
This problem is $\mathsf{coNP}$-complete \citep{freundorlin1985complexity,gritzmannklee1993radii} even for the restricted family
$P=[-1,1]^d\subseteq Q$ with $0\in\operatorname{int}(Q)$.
Given such an instance, we construct a bounded \emph{standard-form} polytope $X=\{z:Az=b,\ z\ge 0\}$ with a distinguished vertex $x_0$ whose optimality cone encodes $Q$ via a lifting:
\[
((y,1),0,0)\in\Lambda(x_0)
\quad\Longleftrightarrow\quad
(y,1)\in \cone\{(v_i,1)\}_{i=1}^M
\quad\Longleftrightarrow\quad
y\in Q.
\]
We then set $\C:=\{((y,1),0,0):y\in P\}$ and take $\Dset=\varnothing$, so the data-consistent fiber equals all of $\C$.
Since $x_0$ is the unique minimizer for a reference cost $c_0=((0,1),0,0)$, pointwise sufficiency at $c_0$ reduces to the containment $\C\subseteq\Lambda(x_0)$, which holds if and only if $P\subseteq Q$.
See Appendix~\ref{app:coNP-hard-pointwise-check}. \hfill $\square$

\smallskip
The same containment-based construction can be strengthened to show that even the weakest \emph{global} problem---deciding whether \emph{no data}
($\Dset=\varnothing$) already suffices---is $\mathsf{coNP}$-hard under an open, full-dimensional polyhedral prior, as stated next.

\begin{theorem}[Informal]
\label{thm:conp-hard-global0}
It is $\mathsf{coNP}$-hard to decide whether the empty dataset $\Dset=\varnothing$ is globally sufficient for $(\X,\C)$,
 when $\X\subseteq\RR^d$ is a bounded polytope and $\C$ is an open polyhedron specified in $H$-representation.
Moreover, computing the intrinsic decision-relevant dimension $d^\star$ is $\mathsf{coNP}$-hard.
\end{theorem}

\paragraph{Proof sketch for Theorem~\ref{thm:conp-hard-global0}.}
Starting from the same containment instance, we replace the low-dimensional slice prior in the pointwise reduction by an \emph{open, full-dimensional} polyhedron $\C_{\mathrm{op}}$ constructed through an \emph{effective-cost} linear map $T$.
The LP over $X$ is designed so that its objective depends on $c$ only through $T(c)$.
We choose an open polyhedron $B_{\mathrm{op}}$ of effective costs with strictly positive last coordinate and set
$\C_{\mathrm{op}}:=\{c:\ T(c)\in B_{\mathrm{op}}\}$.
In the YES case $P\subseteq Q$, openness implies $B_{\mathrm{op}}\subseteq \operatorname{int}\cone\{(v_i,1)\}_{i=1}^M$, making $x_0$ the unique optimizer for every $c\in\C_{\mathrm{op}}$ and hence allowing a decoder with $|\Dset|=0$.
In the NO case, $B_{\mathrm{op}}$ contains an effective cost outside $\cone\{(v_i,1)\}_{i=1}^M$, producing two costs in $\C_{\mathrm{op}}$ with different optimizers and ruling out zero-query global sufficiency.
See Appendix~\ref{app:conp-hard-global} for details. \hfill $\square$

Consequently, computing the minimum size of global SDDs and outputting such a dataset are $\mathsf{coNP}$-hard. Combining \Cref{thm:np-hard-dir-dim,thm:conp-hard-global0} with \Cref{cor:dir-char} yields the following final statement.

\begin{corollary}
\label{cor:final-hardness}
For convex and open $\C$, constructing a minimum-size global SDD and computing its minimum size are both
$\mathsf{NP}$-hard and $\mathsf{coNP}$-hard.
\end{corollary}

The $\mathsf{coNP}$-hardness in \Cref{thm:coNP-hard-pointwise-check} arises from degenerate LP geometry, where a single optimal extreme point may correspond to many bases, and its optimality region can have a complicated representation.
Going forward, we adopt a standard nondegeneracy assumption to recover tractability.

\begin{assumption}
\label{ass:nondeg}
The polytope $\X=\{x\in\RR^d:Ax=b,\ x\ge 0\}$ is nondegenerate: every extreme point $x^\star\in\X^\angle$ has exactly
$m$ strictly positive components.
\end{assumption}

Under Assumption~\ref{ass:nondeg}, each extreme point has a single associated optimality cone that admits a simple linear
inequality description. This makes it possible to test whether a fiber $\C(\Dset,s)$ is contained in a candidate cone by
solving only a polynomial number of LPs (as we do in \Cref{alg:pointwise-cp}). In contrast, without Assumption~\ref{ass:nondeg}, even deciding pointwise sufficiency is $\mathsf{coNP}$-hard by \Cref{thm:coNP-hard-pointwise-check}.

\section{A Tractable Algorithm that Finds Pointwise SDD}
\label{sec:cuttingplanes}
In this section, we present a sequential cutting-plane routine that is run \emph{offline} to construct a pointwise-sufficient dataset for a fixed (and possibly unknown) cost vector $c\in\C$. The resulting dataset itself remains non-adaptive: once it is returned, its measurements are prescribed in advance. At iteration $k$, with current query set $\Dset_k$, the routine maintains the data-consistent fiber $\C_k:=\C(\Dset_k,s(c;\Dset_k))$, that is, the set of cost vectors still consistent with the measurements collected so far, and repeatedly checks whether $\C_k$ lies inside the optimality cone of a single LP basis; when containment fails, it adds the normal of a reachable violated facet to shrink the fiber. Under Assumption~\ref{ass:nondeg}, each iteration solves one LP over $\X$ and at most $d-m$ convex minimization subproblems over the current fiber. Moreover, every new queried direction is decision-relevant (lies in $\dir(\X^\star(\C))$) and is linearly independent of the previous queries; consequently, the procedure makes at most $d^\star$ queries.

Unlike the global SDD characterization in Section~\ref{sec:prelim}, pointwise sufficiency is a containment statement about a single fiber and does not require $\C$ to be open.
Accordingly, the only structural assumption we use in this section for the prior set is convexity; we take it closed only for algorithmic convenience (Remark~\ref{rem:closure-ok}).
Remaining proofs for this section appear in Appendix~\ref{app:sec5}.

\begin{assumption}
\label{ass:C-convex}
The uncertainty set $\C\subseteq\RR^d$ is convex.
\end{assumption}
This covers the main priors considered in the paper, including polyhedral uncertainty sets, ellipsoids, and other convex confidence regions for the cost vector. To make the logic of the algorithm transparent, we first translate pointwise sufficiency into a cone-containment test around one optimal basis, and then isolate a subproblem that checks the individual facet inequalities of that cone.

% ------------------------------------------------------------
\subsection{Pointwise sufficiency as optimality-cone containment}
Fix a basis $B\subseteq\{1,\dots,d\}$ with $|B|=m$ and let $N$ denote its complement. Write $A=[A_B\ A_N]$, and let
$A_j$ be the $j$th column of $A$. The corresponding basic feasible solution is $x(B)$ with $x_N(B)=0$ and
$x_B(B)=A_B^{-1}b$. For each nonbasic index $j\in N$, let $\delta(B,j)\in\RR^d$ denote the standard edge direction obtained by increasing $x_j$ from~$0$, i.e., $\delta_N(B,j)=e_j$ and $\delta_B(B,j)=-A_B^{-1}A_j$. Under Assumption~\ref{ass:nondeg}, feasible bases are in one-to-one correspondence with vertices of $\X$.
Moreover, for any feasible basis $B$, the optimality region of its corresponding $x(B)$ is the polyhedral cone
\begin{equation}
\label{eq:cone-delta}
\Lambda(B) := \{c\in\RR^d:\ c^\top\delta(B,j)\ge 0\ \ \forall j\in N\}.
\end{equation}
After querying directions $q_1,\dots,q_k$, let $Q_k=[q_1\ \cdots\ q_k]\in\RR^{d\times k}$ and $s_k=Q_k^\top c$. The
current fiber is $\C_k :=\{c'\in\C:\ Q_k^\top c'=s_k\}.$ By Definition~\ref{def:pointwise-sdd}, the dataset is pointwise sufficient at $c$ once there exists a basis $B$ such that
$\C_k\subseteq\Lambda(B)$.
Using \eqref{eq:cone-delta}, this containment reduces to checking that
$\min_{c'\in\C_k} (c')^\top\delta(B,j) \ge 0$ for all $j\in N$, which motivates the face-intersection subproblem below.

% ------------------------------------------------------------
\paragraph{The face-intersection (FI) subproblem.}
\label{subsec:face-intersection}
The subproblem tests one facet inequality of the candidate cone by optimizing the corresponding linear functional over the current fiber.
For a direction $\delta\in\RR^d$ and a fiber $\C_k$, define the face-intersection value
\begin{equation}
\label{eq:FI}
m_k(\delta) \;:=\; \inf_{c'\in\C_k} (c')^\top\delta.
\end{equation}
Whenever the infimum is attained, we write
\[
c^{\mathrm{out}}_k(\delta)\in\arg\min_{c'\in\C_k} (c')^\top\delta.
\]
More generally, if $m_k(\delta)<0$, any point $c^{\mathrm{out}}_k(\delta)\in\C_k$ satisfying $(c^{\mathrm{out}}_k(\delta))^\top\delta<0$ serves as a witness of failure.
Under Assumption~\ref{ass:C-convex}, the value problem defining $m_k(\delta)$ is a convex optimization problem. In the algorithmic settings of interest, when $\C$ is polyhedral it is an LP; when $\C$ is an ellipsoid, $\mathrm{FI}(\delta;\C_k)$ admits a closed form (Proposition~\ref{prop:FI-ellipsoid}), which can speed up implementations.

% ------------------------------------------------------------
\subsection{A facet-hit cutting-plane algorithm}

\begin{algorithm}[h]
\caption{A cutting-plane algorithm that finds pointwise SDD}
\label{alg:pointwise-cp}
\small
\begin{algorithmic}[1]
\REQUIRE LP data $(A,b)$, prior set $\C$, a cost vector $c\in\C$,
and an initial dataset $\Dset_{\mathrm{init}}\subset\RR^d$.
\STATE Initialize $\Dset \gets \Dset_{\mathrm{init}}$.
\STATE Let $k \gets |\Dset|$; form $Q_k = [q_1\ \cdots\ q_k]$ from $\Dset$ (any fixed order) and set $s_k \gets Q_k^\top c$.
\WHILE{true}
    \STATE Set $\C_k := \{c'\in\C : Q_k^\top c' = s_k\}$ and set $c^{\mathrm{in}} \gets c$.
    \STATE Solve $\min\{(c^{\mathrm{in}})^\top x : Ax=b,\ x\ge 0\}$ and obtain an optimal basis $B$ with nonbasis $N$.
    \STATE \textbf{(Containment test via face-intersection subproblem)} For each $j\in N$, form $\delta_j := \delta(B,j)$ and compute
        \[
m_j := \min_{c'\in \C_k} (c')^\top \delta_j, 
\quad c^{\mathrm{out}}_j \in \arg\min_{c'\in \C_k} (c')^\top \delta_j .
        \tag{$\mathrm{FI}(\delta_j;\C_k)$}
        \]
    \STATE Let $j_0 \in \arg\min_{j\in N} m_j$, set
    $m_{\min} := m_{j_0}$, $c^{\mathrm{out}} := c^{\mathrm{out}}_{j_0}$.
    \IF{$m_{\min} \ge 0$}
        \STATE \textbf{Return} dataset $\Dset$ and certificate basis $B$ (decision $x(B)$); \textbf{break}.
    \ELSE
        \STATE \textbf{(Facet-hit rule)} For each $j\in N$ with $(c^{\mathrm{out}})^\top \delta_j < 0$, set
           $\alpha_j :=
        \frac{(c^{\mathrm{in}})^\top \delta_j}
        {(c^{\mathrm{in}})^\top \delta_j - (c^{\mathrm{out}})^\top \delta_j}
        \in [0,1).$
        
        \STATE Let $\alpha^\star := \min \alpha_j$ and pick any $j^\star \in \arg\min \alpha_j$.
        \STATE Set $q_{k+1} := \delta_{j^\star}$ and set $\sigma_{k+1} \gets q_{k+1}^\top c$.
        \STATE Update:
        $\Dset \gets \Dset \cup \{q_{k+1}\}$,
        $Q_{k+1} \gets [Q_k\ \ q_{k+1}]$,
        $s_{k+1} \gets (s_k^\top,\ \sigma_{k+1})^\top$,
        $k \gets k+1$.
    \ENDIF
\ENDWHILE
\end{algorithmic}
\end{algorithm}

Algorithm~\ref{alg:pointwise-cp} gives the full procedure.
At each iteration it anchors at the realized cost vector $c$ (i.e., we take $c^{\mathrm{in}}:=c\in\C_k$),
solves the LP under $c^{\mathrm{in}}$ to obtain an optimal vertex solution $x(B)$; by Assumption~\ref{ass:nondeg},
this vertex uniquely determines the corresponding feasible basis $B$, and tests whether \emph{every} cost vector in the fiber
is contained in the corresponding cone $\Lambda(B)$.
This is done by evaluating each facet inequality via the face-intersection subproblem.
If the minimum violation is nonnegative, the fiber is fully inside $\Lambda(B)$ and pointwise sufficiency holds.
Otherwise, a witness point $c^{\mathrm{out}}\in\C_k$ may violate several facet inequalities of $\Lambda(B)$.

\begin{remark}[Closedness is not necessary]
\label{rem:closure-ok}
Pointwise sufficiency and Algorithm~\ref{alg:pointwise-cp} interact with the prior only through
containment tests of the form $\C_k\subseteq \Lambda(B)$.
Since $\Lambda(B)$ is closed, $\C_k\subseteq \Lambda(B)$ holds iff $\cl(\C_k)\subseteq \Lambda(B)$, where $\cl(\cdot)$ denotes closure.
Thus, for these containment tests, one may replace a (possibly nonclosed) fiber by its closure
without changing any certification outcome.
Moreover, any violating witness $c_j^{\mathrm{out}}\in \C_k\setminus\Lambda(B)$ (whenever it exists) remains valid 
after this replacement.
\end{remark}

The facet-hit rule identifies the first facet of $\Lambda(B)$ encountered along the segment $[c^{\mathrm{in}},c^{\mathrm{out}}]$, and thus guarantees the existence of a boundary point
$c^{\mathrm{hit}}\in\C_k\cap\Lambda(B)$ with $(c^{\mathrm{hit}})^\top\delta(B,j^\star)=0$ (Lemma~\ref{lem:query-relevant}),
which is what makes the new query direction decision-relevant. A concrete counterexample showing that an arbitrary violated
facet can fail is given below:

\begin{example}[A counterexample motivating the facet-hit rule]
\label{ex:why-facet-hit}
    Let $\X=[0,1]^2$ and consider the vertex $x=(0,0)$, whose optimality cone is $\Lambda=\{c\in\RR^2:\ c_1\ge 0,\ c_2\ge 0\}$, with facet hyperplanes $c_1=0$ and $c_2=0$.
Let $\varepsilon\in(0,1)$ and define $c^{\mathrm{in}}=(1,\varepsilon)$ and $c^{\mathrm{out}}=(-1,-1)$.
Consider any convex fiber $\C_k$ whose intersection with $\Lambda$ is the segment $\conv\{c^{\mathrm{in}},c^{\mathrm{out}}\}$
(e.g., take $\C_k$ to be exactly this segment). Then $c^{\mathrm{out}}$ violates both inequalities $c_1\ge 0$ and $c_2\ge 0$.
Along the segment $c_\alpha=(1-\alpha)c^{\mathrm{in}}+\alpha c^{\mathrm{out}}$, the coordinate $c_{2,\alpha}$ hits $0$ at a very small $\alpha$ while $c_{1,\alpha}$ is still strictly positive; thus $\C_k\cap\Lambda$ reaches the boundary only through the facet $c_2=0$.
In contrast, $c_\alpha$ intersects $c_1=0$ only after $c_{2,\alpha}$ has already become negative, i.e., outside $\Lambda$.
Therefore, querying the normal of the ``wrong'' violated facet ($c_1=0$) does not correspond to a boundary that the fiber can reach while keeping $x$ optimal.
The facet-hit rule avoids this issue by selecting the first facet reached from an interior anchor point.
\end{example}

\begin{remark}[Picking arbitrary $c^{\mathrm{in}} \in \C_k$]
\label{rem:alg1-query-only}
Algorithm~\ref{alg:pointwise-cp} is stated in a \emph{fixed-anchor} form (we set $c^{\mathrm{in}}:=c$ at each iteration).
The same facet-hit cutting-plane idea and all results in this section also apply when the cost vector $c$ is \emph{unknown} and one only has oracle access to inner products $q^\top c$:
in that case, line~4 may pick \emph{any} anchor $c^{\mathrm{in}}\in\C_k$, and then proceed identically.
\end{remark}

\begin{remark}[A deterministic tie-breaking convention]
\label{rem:tie-breaking}
Throughout Algorithms~\ref{alg:pointwise-cp},
whenever a choice is non-unique, we can further assume that ties are
resolved according to a fixed deterministic rule.
\end{remark}

% ------------------------------------------------------------
\subsection{Correctness and basic properties}
We now record a few basic facts about Algorithm~\ref{alg:pointwise-cp}.
The next two lemmas are technical but important for our later analysis in
Section~\ref{sec:ddr-compression}: Lemma~\ref{lem:query-relevant} shows that every new queried direction is genuinely decision-relevant. Lemma~\ref{lem:li-auto} guarantees that the dataset grows with linearly independent directions.
Theorem~\ref{thm:alg1-correct} shows the algorithm terminates and indeed certifies pointwise sufficiency.

\begin{lemma}
\label{lem:query-relevant}
In Algorithm~\ref{alg:pointwise-cp}, any newly added direction $q_{k+1}$ lies in $\Delta(\X,\C)\subseteq \dir(\X^\star(\C))$.
\end{lemma}
\begin{proof}
Because $B$ is an optimal basis for $c^{\mathrm{in}}\in\C_k$, we have $c^{\mathrm{in}}\in\Lambda(B)$, i.e.
$(c^{\mathrm{in}})^\top\delta(B,j)\ge 0$ for all $j\in N$.
Since we are in the \texttt{Else} branch, there exists $c^{\mathrm{out}}\in\C_k$ with
$(c^{\mathrm{out}})^\top\delta(B,j_0)=m_{\min}<0$ for at least one $j_0$.
Define the segment $c_\alpha := (1-\alpha)c^{\mathrm{in}}+\alpha c^{\mathrm{out}}$, $\alpha\in[0,1]$.
Because $\C_k$ is convex, $c_\alpha\in\C_k$ for all $\alpha\in[0,1]$.

By construction of $\alpha^\star$ and $j^\star$, we have $(c_{\alpha^\star})^\top\delta(B,j^\star)=0$
and $(c_{\alpha^\star})^\top\delta(B,j)\ge 0$ for all $j\in N$ (first-hit property).
Thus $c^{\mathrm{hit}}:=c_{\alpha^\star}\in \C_k\cap \Lambda(B)$ and lies on the face
$\Lambda(B)\cap \{\delta(B,j^\star)\}^\perp$.

By \Cref{def:Delta} (with $x^\star=x(B)$ and $\C=\C_k$), this implies $\delta(B,j^\star)\in \Delta(\X,\C_k)$.
Since $\C_k\subseteq \C$, we also have $\Delta(\X,\C_k)\subseteq \Delta(\X,\C)$.
Finally, by \Cref{thm:bennouna-spanDelta},
\[
\delta(B,j^\star)\in \spanop\Delta(\X,\C)=\dir(\X^\star(\C)),
\]
and in particular $\delta(B,j^\star)\in\dir(\X^\star(\C))$.
\end{proof}

\begin{lemma}
\label{lem:li-auto}
In Algorithm~\ref{alg:pointwise-cp}, the queried directions are linearly independent. In particular, the algorithm makes at most $\dstar$ queries.
\end{lemma}
\begin{proof}
Assume for contradiction that $q_{k+1}\in \spanop(Q_k)$.
Then $(q_{k+1})^\top c'$ is constant over $\C_k=\{c'\in\C:\ Q_k^\top c'=s_k\}$.
In particular, $(q_{k+1})^\top c^{\mathrm{in}}=(q_{k+1})^\top c^{\mathrm{out}}$.
But $c^{\mathrm{in}}\in\Lambda(B)$ implies $(q_{k+1})^\top c^{\mathrm{in}}\ge 0$,
while the facet-hit rule guarantees $(q_{k+1})^\top c^{\mathrm{out}}<0$.
This contradiction shows $q_{k+1}\notin \spanop(Q_k)$, and therefore $\rank(Q_{k+1})=\rank(Q_k)+1$. Moreover, since $\dstar=\dim(\dir(\X^\star(\C)))$ and $q_{k+1}\in \dir(\X^\star(\C))$ by \Cref{lem:query-relevant}, the algorithm makes at most $\dstar$ queries.
\end{proof}

\begin{theorem}[Correctness]
\label{thm:alg1-correct}
Algorithm~\ref{alg:pointwise-cp} terminates after at most $d^\star+1$ iterations and returns a dataset $\Dset$ that is pointwise sufficient at the (possibly unknown) $c$.
\end{theorem}

\begin{proof}
At termination we have $m_{\min}\ge 0$, hence for every $j\in N$,
\[
\min_{c'\in\C_k} (c')^\top \delta(B,j)\ \ge\ 0
\quad\Rightarrow\quad
(c')^\top\delta(B,j)\ge 0\ \ \forall c'\in\C_k.
\]
By \eqref{eq:cone-delta}, this implies $\C_k\subseteq \Lambda(B)$.
Therefore the fixed decision $x(B)$ is optimal for every $c'\in\C_k$, where $\C_k$ is the fiber associated with the current dataset $\Dset$, so $\Dset$ is pointwise
sufficient at $c$.

By \Cref{lem:li-auto}, the algorithm makes at most $\dstar$ new queries. Since each non-terminating iteration adds exactly one new query direction, the \texttt{while} loop executes at most $\dstar+1$ iterations.
\end{proof}

\begin{property}
\label{prop:polytime}
Assume that $\C$ is either (i) a polytope given in $H$-representation, or (ii) an ellipsoid.
Then Algorithm~\ref{alg:pointwise-cp} can be implemented to run in time \textbf{polynomial} in the input size.
In particular, it makes at most $d^\star$ oracle queries of the form $q^\top c$ and solves at most $(d^\star+1)$ LPs over $\X$ and
$(d^\star+1)(d-m)$ face-intersection subproblems (LPs in case~(i), and closed form in case~(ii)).
\end{property}

\section{Learning from Distributional Data}
\label{sec:ddr-compression}

Section~\ref{sec:cuttingplanes} constructs a pointwise-sufficient query set for a single cost vector $c$.
We now move to a data-driven regime in which the LP~\eqref{eq:lp} is solved repeatedly with random  $c\sim P_c$ supported on $\C$, and we observe i.i.d.\ samples $c_1,\dots,c_n$.
Our goal is to learn a \emph{compressed, decision-sufficient representation}: a small set of query
directions that are pointwise sufficient for a fresh draw from $P_c$ with high probability, together
with a \emph{distribution-free certificate} on its failure probability.
For a given dataset $\Dset$ as query directions, define the \emph{0--1 sufficiency loss}
\begin{equation}
\label{eq:01loss}
\ell(\Dset,c)\;:=\;\mathbf{1}\{\Dset \text{ is not pointwise sufficient at } c\}.
\end{equation}
We aim to output a dataset $\Dset$ with small \emph{risk} $R(\Dset):=\PP_{c\sim P_c}[\ell(\Dset,c)=1].$

% ------------------------------------------------------------
\subsection{A cumulative algorithm}
We now run the pointwise cutting-plane routine sequentially on each training sample $c_i$ and \emph{accumulate} the queried directions.
Algorithm~\ref{alg:cumulative} initializes with an empty dataset and, for $i=1,\dots,n$, invokes Algorithm~\ref{alg:pointwise-cp} on $c_i$
using the current dataset as a warm start. If $c_i$ is already certified by the current dataset, nothing changes; otherwise, new directions
are added until $c_i$ becomes pointwise sufficient. We call an index $i$ \emph{hard} if processing $c_i$ adds at least one new direction, i.e., $\Dset_i\neq \Dset_{i-1}$.

\begin{algorithm}[h]
\caption{Learning sufficient decision datasets over samples}
\label{alg:cumulative}
\small
\begin{algorithmic}[1]
\REQUIRE Prior $\C$, LP data $(A,b)$, i.i.d.\ samples $c_1,\dots,c_n$ (via oracle access to $q^\top c_i$).
\STATE Initialize dataset $\Dset_0 \gets \emptyset$, hard index set $T\gets\emptyset$.
\FOR{$i=1$ to $n$}
    \STATE Run Algorithm~\ref{alg:pointwise-cp} on $c_i$ with initialization $\Dset_{\mathrm{init}}=\Dset_{i-1}$.
    Let the returned dataset be $\Dset_i$.
    \IF{$\Dset_i \neq \Dset_{i-1}$}
        \STATE Mark $i$ as hard: $T \gets T\cup\{i\}$.
    \ENDIF
\ENDFOR
\STATE \textbf{Return} final dataset $\Dset_n$ and hard set $T$.
\end{algorithmic}
\end{algorithm}

The next three lemmas formalize the learning-theoretic structure underlying this cumulative procedure.
Together, they show that Algorithm~\ref{alg:cumulative} induces a \emph{stable, realizable} sample
compression scheme \citep[Definitions~7--8]{hanneke2021stablecompression} of compression size at most $d^\star$.
This is the key mechanism we use in the next subsection to obtain a distribution-free fast-rate
certificate on the true failure probability $R(\Dset_n)=\Pr_{c\sim P_c}[\ell(\Dset_n,c)=1]$.

\begin{lemma}[Realizability]
\label{lem:zero-empirical}
Algorithm~\ref{alg:cumulative} returns $\Dset_n$ with $\ell(\Dset_n,c_i)=0$ for all $i=1,\dots,n$.
\end{lemma}

\begin{proof}
When processing $c_i$, the inner run of Algorithm~\ref{alg:pointwise-cp} certifies pointwise sufficiency for $c_i$,
so $\ell(\Dset_i,c_i)=0$.
For $t>i$, the cumulative procedure only enlarges the dataset, $\Dset_t\supseteq \Dset_i$, and
\Cref{prop:pointwise-basic}(i) implies $\ell(\Dset_t,c_i)=0$ for all later $t$, hence also at $t=n$.
\end{proof}

\begin{lemma}[Stability]
\label{lem:depend-only-on-T}
The final dataset $\Dset_n$ returned by Algorithm~\ref{alg:cumulative} is fully determined by the compressed subsequence
$(c_i)_{i\in T}$, independent of the remaining samples.
\end{lemma}

\begin{proof}
By the deterministic tie-breaking convention (\cref{rem:tie-breaking}), Algorithm~\ref{alg:pointwise-cp}
defines a deterministic update map, and hence Algorithm~\ref{alg:cumulative}
is a deterministic function of the sample sequence.

In Algorithm~\ref{alg:cumulative}, the dataset only changes at indices in $T$ by definition.
Removing an index $t\notin T$ removes an iteration that would have run
Algorithm~\ref{alg:pointwise-cp} with initialization $\Dset_{t-1}$ and returned the same dataset
$\Dset_t=\Dset_{t-1}$.
Therefore, deleting all non-hard iterations leaves the dataset entering each hard iteration unchanged,
and thus reproduces the same sequence of dataset updates and the same final dataset.
\end{proof}

\begin{lemma}[Compression size bound]
\label{lem:hard-bound}
Under Assumption~\ref{ass:nondeg}, $|T|\le |\Dset_n|\le d^\star$.
\end{lemma}

\begin{proof}
By Lemma~\ref{lem:li-auto}, each time Algorithm~\ref{alg:pointwise-cp} appends a new query,
that direction is linearly independent of the previously queried directions in that run.
Since Algorithm~\ref{alg:cumulative} warm-starts each run at $\Dset_{i-1}$, this implies by induction over
$i=1,\dots,n$ that the cumulative datasets $\Dset_i$ remain linearly independent.

By definition, $i\in T$ implies $\Dset_i\neq \Dset_{i-1}$, hence at least one new independent direction
was added while processing $c_i$, so $|T|\le |\Dset_n|$.

Finally, by Lemma~\ref{lem:query-relevant}, every query direction that Algorithm~\ref{alg:pointwise-cp} can append
lies in $\dir(\X^\star(\C))$. Since $\Dset_n$ is the union of all appended directions across the cumulative run, we have
$\Dset_n\subseteq \dir(\X^\star(\C))$.
Therefore, linear independence yields
$|\Dset_n|=\dim\spanop(\Dset_n)\le \dim\dir(\X^\star(\C))=d^\star$.
\end{proof}

% ------------------------------------------------------------
\subsection{A distribution-free certificate via stable compression}
We can now invoke the clean stable-compression generalization
bound of \citet[Corollary~11]{hanneke2021stablecompression} to certify the true failure probability
$R(\Dset_n)=\PP_{c\sim P_c}[\ell(\Dset_n,c)=1]$.
For an alternative viewpoint, see \citealp{campigaratti2023compression,paccagnan2025p2l}.

\begin{theorem}[Certificate via stable compression]
\label{thm:certificate}
For any $\delta\in(0,1)$, with probability at least $1-\delta$ over the draw of $c_1,\dots,c_n$,
the output $(\Dset_n,T)$ of Algorithm~\ref{alg:cumulative} satisfies
\begin{equation}
\label{eq:certificate_explicit}
R(\Dset_n)\ \le\ \frac{4}{n}\Bigl(6|T|+\ln\frac{e}{\delta}\Bigr)\le\ \frac{4}{n}\Bigl(6d^\star+\ln\frac{e}{\delta}\Bigr),
\end{equation}
where the second inequality uses $|T|\le d^\star$ from Lemma~\ref{lem:hard-bound}.
\end{theorem}

\begin{proof}
We apply the stable sample compression bound of \citet[Corollary~11]{hanneke2021stablecompression}.

To be specific, we view any queried dataset $\Dset$ as inducing a binary prediction rule
$h_\Dset:\mathcal C\to\{0,1\}$ defined by $h_\Dset(c):=\ell(\Dset,c)\in \{0,1\}$.
Thus $R(\Dset)=\Pr_{C\sim P_c}[h_\Dset(C)=1]=\Pr_{C\sim P_c}[\ell(\Dset,C)=1]$
is exactly the associated $0$--$1$ risk.
Equivalently, one may regard this as a supervised distribution over $(C,Y)$
with $C\sim P_c$ and $Y=0$ almost surely. Let $S=(c_1,\dots,c_n)$ be the training sequence and let $(\Dset_n,T)$ be the output of Algorithm~\ref{alg:cumulative} on $S$.
Define a compression function $\kappa$ by letting $\kappa(S)$ be the subsequence of {hard} samples $(c_i)_{i\in T}$ (in their original order).
Define a reconstruction function $\rho$ that maps any subsequence $S'$ to the prediction rule $h_{\Dset'}$ induced by the final dataset $\Dset'$
returned by running Algorithm~\ref{alg:cumulative} on $S'$.

By Lemma~\ref{lem:depend-only-on-T}, running Algorithm~\ref{alg:cumulative} on $\kappa(S)$ reproduces the same final dataset $\Dset_n$,
so $\rho(\kappa(S))=h_{\Dset_n}$.
Moreover, Lemma~\ref{lem:depend-only-on-T} also implies the \emph{stability} property of \citet[Definition~8]{hanneke2021stablecompression}:
removing any subset of the non-hard samples (i.e., elements of $S\setminus \kappa(S)$) does not affect the reconstructed output. Lemma~\ref{lem:zero-empirical} gives $\ell(\Dset_n,c_i)=0$ for all $i=1,\dots,n$, i.e., the empirical $0$--$1$ risk of $\rho(\kappa(S))=h_{\Dset_n}$
on $S$ is zero. Apply Corollary~11 of \citealp{hanneke2021stablecompression}, with probability at least $1-\delta$ over the draw of $S\sim P_c^n$,
\[
R(\Dset_n)\ =\ R(h_{\Dset_n})
\ =\ R(\rho(\kappa(S)))
\ \le\ \frac{4}{n}\Bigl(6|\kappa(S)|+\ln\frac{e}{\delta}\Bigr)
\ =\ \frac{4}{n}\Bigl(6|T|+\ln\frac{e}{\delta}\Bigr).
\]
Finally, Lemma~\ref{lem:hard-bound} implies $|T|\le d^\star$, which gives the last sentence of \Cref{thm:certificate}.
\end{proof}

Equation~\eqref{eq:certificate_explicit} gives the fast-rate certificate $R(\Dset_n)\le O\!\bigl((d^\star+\ln(1/\delta))/n\bigr)$. Recent data-driven projection approaches for LPs learn a low-dimensional embedding that preserves feasibility and objective values approximately and then control a downstream error via uniform-convergence-style analyses; this leads to the characteristic $1/\sqrt{n}$ dependence on sample size and complexity terms tied to the pseudo-dimension of the learned projection family \citep{balcan2024howmuchdata,sakaue2024datadrivenprojectionslp}.
Our guarantee is complementary. We target decision sufficiency (a binary property) and exploit polyhedral geometry to ensure only decision-relevant directions can ever be queried. This guarantees a stable compression scheme, yielding $\widetilde{O}(d^\star/n)$ certificates that depend on the intrinsic dimension $d^\star$. This fast rate is also characteristic of realizability; in agnostic settings where zero empirical loss is unattainable, compression-based methods are subject to $\Omega\!\bigl(\sqrt{k\log(n/k)/n}\bigr)$ lower bounds~\citep{hanneke19b}, where $k$ is the compression size.

Finally, it is natural to ask whether the dependence on $d^\star$ and $n$ can be improved.
The next theorem shows that, at least for the concrete cumulative procedure in Algorithm~\ref{alg:cumulative}, the fast-rate certificate is tight up to constants: one needs $n=\Omega(d^\star/\varepsilon)$ samples to drive the pointwise-sufficiency failure probability below $\varepsilon$ with constant confidence.
\begin{theorem}
\label{thm:lowerbound}
Fix any integer $d^\star\ge 2$ and any $\varepsilon\in(0,1/4)$.
There exist an ambient dimension $d\ge d^\star$, a nondegenerate LP polytope $\X\subseteq\RR^{d}$,
a convex uncertainty set $\C\subseteq\RR^{d}$ with $\dim(\dir(\X^\star(\C)))=d^\star$,
and a distribution $P_c$ supported on $\C$ such that the following holds. 
If $\Dset_n$ is the output of Algorithm~\ref{alg:cumulative} on $n$ i.i.d.\ samples from $P_c$ and $n \le \frac{d^\star-1}{8\varepsilon}$, then $\PP\!\left( R(\Dset_n) > \varepsilon \right) \ge \frac{1}{2}$.
\end{theorem}

Proof of Theorem~\ref{thm:lowerbound} and additional technical details appear in Appendix~\ref{app:sec6}. At a high level, the proof follows the classical ``rare-types'' construction used to obtain lower bounds for realizable sample-compression schemes \citep{littlestone1986relating}.
However, we cannot simply import a generic compression lower bound as a black box, because here the compression map is not arbitrary: it must arise from LP optimality geometry through Algorithm~\ref{alg:pointwise-cp}.
Accordingly, we explicitly construct a linear program, a convex prior set $\C$ with $\dim(\dir(\X^\star(\C)))=d^\star$, and a distribution $P_c$ over $c$ such that each rare event forces the discovery of a distinct decision-relevant direction.

\section{Application: Model Compression for Contextual Linear Optimization}
\label{sec:clo_compression}

In this section, we illustrate how decision-sufficient representations yield a principled model-compression guarantee for contextual linear
optimization (CLO). Let $(\xi,c)\sim P$, where $\xi\in\Xi\subseteq\RR^p$ is a context and
$c\in\C\subseteq\RR^d$ (a.s.) is the cost vector of a downstream linear program over a known bounded polytope
$\X\subseteq\RR^d$. We assume that $\C$ is convex. Let $P_c$ denote the marginal distribution of~$c$. The goal in CLO is to train a predictor $f:\Xi\to\RR^d$ from contextual data so that the induced plug-in decisions $x^\star(f(\xi))$ achieve low out-of-sample loss.  Throughout, we fix a deterministic oracle $x^\star:\RR^d\to\X^\angle$ such that
$x^\star(v)\in\argmin_{x\in\X} v^\top x$ for all $v$ (e.g., with lexicographic tie-breaking rule).

\paragraph{Ellipsoidal prior with a canonical lifting map.}
Throughout this section, we specialize the prior set $\C$ to an ellipsoid
$\C:=\{c\in\RR^d:\ (c-c_0)^\top \Sigma^{-1}(c-c_0)\le 1\}$, for some $\Sigma\succ 0,\ c_0\in\RR^d.$
Let $W\subseteq\RR^d$ be any $t$-dimensional subspace with orthonormal basis $U\in\RR^{d\times t}$.
We define the \emph{lifting matrix} $\mathcal{L}_U:=\Sigma U(U^\top \Sigma U)^{-1}$, and the corresponding \emph{canonical lifting map} (Appendix~\ref{app:lift})
\begin{equation}
\label{eq:lift_operator_main}
\operatorname{lift}_U(s):=c_0+\mathcal{L}_U s,\qquad s\in\RR^t.
\end{equation}
%Then $\operatorname{lift}_U(s)\in\C$ whenever $s\in U^\top(\C-c_0)$.

The lifting map is introduced for a simple but important reason:
To use such a prediction in the original optimization problem, we must map it back to a
full cost vector in $\RR^d$.
When $\C$ is not centered at the origin, a purely linear compressed predictor would typically have range contained in a linear subspace through the origin and hence may not lie in $\C$.
The canonical lifting map provides a principled way to ``complete'' a low-dimensional coordinate into a cost vector
that is feasible for the prior set $\C$.

\paragraph{Two-stage pipeline.}
Our deployment follows a two-stage pipeline that separates decision-sufficient representations from predictor training. In \emph{Stage I}, we construct a dataset $\hat\Dset$ that is
decision-sufficient on the prior set~$\C$, and then form the induced subspace $\hat W:=\spanop(\hat\Dset)$ with orthonormal basis $\hat U$.
In \emph{Stage II}, we draw independent contextual samples
$S:=\{(\xi_i,c_i)\}_{i=1}^n\sim P$ and train a predictor that
first predicts a coordinate in the learned subspace $\hat W$. 
%We then lift it back to $\mathbb{R}^d$ via~\eqref{eq:lift_operator_main} (with $U=\hat U$).
%Moreover, if we constrain the predicted coordinate to lie in $\hat U^\top(\mathcal{C}-c_0)$,
%then the lifted prediction lies in $\mathcal{C}$ by \Cref{lem:lifting_main}.

To highlight the role of the decision-relevant subspace, we first analyze Stage~II under an idealized assumption that Stage~I has recovered the global decision-relevant subspace.
Concretely, let $\Dset$ be a minimal-size global sufficient dataset on $\C$ and define $W_\star:=\spanop(\Dset)$ with intrinsic dimension $\dstar:=\dim(W_\star)$.
Let $U_\star\in\RR^{d\times \dstar}$ be an orthonormal basis for $W_\star$.
In Section~\ref{sec:clo_stage1}, we then show how to implement Stage~I from contextual samples $(\xi,c)$ via conditional-mean regression and Algorithm~\ref{alg:cumulative}, yielding an additional representation-estimation term.

\subsection{SPO training in a decision-sufficient subspace}
Given a predictor $\hat c=f(\xi)$, the plug-in decision is $x^\star(\hat c)$ and the SPO loss is
$\ell_{\mathrm{SPO}}(\hat c,c):=c^\top x^\star(\hat c)-c^\top x^\star(c)\ge 0$.
The corresponding \emph{SPO risk} of a predictor $f:\Xi\to\RR^d$ is $R_{\mathrm{SPO}}(f)
:=\EE_{(\xi,c)\sim P}\bigl[\ell_{\mathrm{SPO}}(f(\xi),c)\bigr]
=\EE_{(\xi,c)\sim P}\bigl[c^\top x^\star(f(\xi)) - c^\top x^\star(c)\bigr].$ In practice, given i.i.d.\ samples $S=\{(\xi_i,c_i)\}_{i=1}^n$ in Stage~II, we train $\theta$ by minimizing either the empirical SPO risk
$\hat R_{\mathrm{SPO}}(f):=\frac1n\sum_{i=1}^n \ell_{\mathrm{SPO}}(f(\xi_i),c_i)$
or its convex surrogate
$\hat R_{\mathrm{SPO}+}(f):=\frac1n\sum_{i=1}^n \ell_{\mathrm{SPO}+}(f(\xi_i),c_i)$.
Following \citet{elmachtoubgrigas2022spo}, the convex surrogate $\ell_{\mathrm{SPO+}}(\hat c,c)
 := \max_{x\in\X} (c-2\hat c)^\top x + 2\hat c^\top x^\star(c) - c^\top x^\star(c).$ Let $x^0=x^\star(c)$ and let $x^1\in\argmax_{x\in\X} (c-2\hat c)^\top x$, equivalently $x^1=x^\star(2\hat c-c)$.
By Danskin's theorem,
\begin{equation}
\label{eq:spoplus_subgrad_main}
2(x^0-x^1)\ \in\ \partial_{\hat c}\,\ell_{\mathrm{SPO+}}(\hat c,c),
\end{equation}
so each stochastic subgradient step requires at most two oracle calls, at $c$ and $2\hat c-c$. Focusing on the decision-sufficient subspace, one can parametrize a cost predictor by first predicting a
$d^\star$-dimensional centered coordinate
$g_\theta:\Xi\to\RR^{d^\star}$ and then lifting it back to $\RR^d$ via
$$\hat c_\theta(\xi)
=\operatorname{lift}_{U_\star}(g_\theta(\xi))
=c_0+\mathcal{L}_{U_\star} g_\theta(\xi),
\qquad \mathcal{L}_{U_\star}:=\Sigma U_\star(U_\star^\top \Sigma U_\star)^{-1}.$$
For linear coordinate models of the form $g_\theta(\xi)=B_\theta\xi$ with $B_\theta\in\RR^{d^\star\times p}$, this reduces the number
of trainable parameters from $dp$ to $d^\star p$.
Composing \eqref{eq:spoplus_subgrad_main} with the affine lifting map and applying the subgradient chain rule yields a valid update in the compressed coordinates: for any choice of
$v\in \partial_{\hat c}\,\ell_{\mathrm{SPO+}}(\hat c_\theta(\xi),c)$,
the linearity of the lifting map implies
$\mathcal{L}_{U_\star}^\top v\in \partial_{g}\,\ell_{\mathrm{SPO+}}(c_0+\mathcal{L}_{U_\star} g,c)\big|_{g=g_\theta(\xi)}$.
If $g_\theta$ is differentiable in $\theta$, then
$$(\nabla_\theta g_\theta(\xi))^\top \mathcal{L}_{U_\star}^\top v
\ \in\
\partial_\theta\,\ell_{\mathrm{SPO+}}(\hat c_\theta(\xi),c),
\quad
v:=2\!\left(x^\star(c)-x^\star\!\bigl(2(c_0+\mathcal{L}_{U_\star} g_\theta(\xi))-c\bigr)\right).$$

For the linear coordinate model $g_\theta(\xi)=B_\theta \xi$, a valid stochastic subgradient for $B_\theta$ at sample $(\xi,c)$ is $(\mathcal{L}_{U_\star}^\top v)\,\xi^\top$.
Algorithm~\ref{alg:streaming_spoplus} summarizes the resulting Stage~II training routine.

\begin{algorithm}[htbp]
\caption{Compressed SPO\texorpdfstring{$+$}{+} training (Stage II)}
\label{alg:streaming_spoplus}
\begin{algorithmic}[1]
\REQUIRE Compressed dataset $\hat\Dset$, basis $U\in\RR^{d\times t}$ for $\hat W=\spanop(\hat\Dset)$, stepsizes $\{\eta_k\}_{k\ge 0}$
\STATE Compute $\mathcal{L}_U\gets \Sigma U(U^\top \Sigma U)^{-1}$ and define $\operatorname{lift}_U(s)=c_0+\mathcal{L}_U s$
\FOR{$k=0,1,2,\dots$}
    \STATE Sample $(\xi_k,c_k)$ uniformly from $S$
    \STATE Predict a centered coordinate in $\RR^{t}$: $\hat s_k\gets g_\theta(\xi_k)$ 
    \STATE Lift to $\RR^{d}$: $\hat c_k\gets \operatorname{lift}_U(\hat s_k)=c_0+\mathcal{L}_U\hat s_k$
    \STATE $x^0\gets x^\star(c_k)$, $x^1\gets x^\star(2\hat c_k-c_k)$
    \STATE $v_k\gets 2(x^0-x^1)$
    \STATE $\theta\gets \theta-\eta_k(\nabla_\theta g_\theta(\xi_k))^\top \mathcal{L}_U^\top v_k$
\ENDFOR
\end{algorithmic}
\end{algorithm}

% ------------------------------------------------------------
\subsection{Improved generalization bound}
\label{sec:clo_gen_bound}

Consider the compressed affine-linear hypothesis class $\mathcal{H}_{U_\star,\dstar}:=\{ f_B(\xi)=c_0+\mathcal{L}_{U_\star} B\xi:\ B\in\RR^{\dstar\times p}\}.$ We define the SPO range constant $\omega_{\X}(\C):=\sup_{c\in\C}\left(\max_{x\in\X} c^\top x-\min_{x\in\X} c^\top x\right).$ With these notations in place, we can state the following guarantee for Stage~II training.

\begin{theorem}
\label{thm:misspec_plus_gen}
Let $\mathcal{H}_{U_\star,\dstar}$ be as above and define the $\C$-valued affine-linear classes
\[
\mathcal{H}(\C):=\{f_A(\xi)=c_0+A\xi:\ A\in\RR^{d\times p},\ f_A(\xi)\in\C\ \text{a.s.}\},
\quad
\mathcal{H}_{U_\star,\dstar}(\C):=\{f\in\mathcal{H}_{U_\star,\dstar}:\ f(\xi)\in\C\ \text{a.s.}\}.
\]

\begin{enumerate}[label=(\arabic*)]
\item \textbf{No misspecification loss.} Let $f_\star\in\argmin_{f\in\mathcal{H}(\C)} R_{\mathrm{SPO}}(f)$ and define its compressed version $\hat f_\star(\xi)
:=\operatorname{lift}_{U_\star}\!\bigl(U_\star^\top(f_\star(\xi)-c_0)\bigr)
=c_0+\mathcal{L}_{U_\star}U_\star^\top(f_\star(\xi)-c_0)\in\mathcal{H}_{U_\star,\dstar}(\C).$

Then $\hat f_\star\in\argmin_{f\in\mathcal{H}_{U_\star,\dstar}(\C)} R_{\mathrm{SPO}}(f)$ and $R_{\mathrm{SPO}}(f_\star)=R_{\mathrm{SPO}}(\hat f_\star)$.

\item \textbf{Improved generalization bound in the compressed class.}
For any $\delta\in(0,1)$, with probability at least $1-\delta$ over $S$, the following bound holds
\emph{simultaneously for all} $f\in\mathcal{H}_{U_\star,\dstar}$:
\begin{equation}
\label{eq:total_fstar_misspec_plus_gen}
R_{\mathrm{SPO}}(f)
\le
\hat R_{\mathrm{SPO}}(f)
+2\,\omega_{\X}(\C)\,\sqrt{\frac{2(\dstar p+1) \log(n|\X^\angle|^2)}{n}}
+\omega_{\X}(\C)\,\sqrt{\frac{\log(1/\delta)}{2n}}.
\end{equation}
%In particular, combining this inequality with part~(1) yields
%\begin{equation}
%\label{eq:total_fstar_misspec_plus_gen}
%R_{\mathrm{SPO}}(f_\star)=R_{\mathrm{SPO}}(\hat f_\star)
%\le
%\hat R_{\mathrm{SPO}}(\hat f_\star)
%+2\,\omega_{\X}(\C)\,\sqrt{\frac{2(\dstar (p+1)+1)\log(n|\X^\angle|^2)}{n}}
%+\omega_{\X}(\C)\,\sqrt{\frac{\log(1/\delta)}{2n}}.
%\end{equation}
\end{enumerate}
\end{theorem}

\paragraph{Proof sketch.}
The proof has three ingredients:
(i) we bound the complexity of the induced decision class $x^\star\circ\mathcal H_{U_\star,\dstar}$ via its Natarajan dimension,
(ii) we plug this bound into the Natarajan-dimension generalization theorem for the SPO loss due to \citet{elbalghiti2023generalization}, and
(iii) we show that under a global sufficient dataset, projecting to $W_\star$ and lifting back is decision-preserving, so restricting to the compressed class incurs no approximation error.
A complete proof is in Appendix~\ref{app:spo_gen_dec_suff_subspace}.

Compared to \citet{elbalghiti2023generalization}, the bound in \eqref{eq:total_fstar_misspec_plus_gen} replaces the ambient dimension $d$ in the dominant
$\widetilde{O}(1/\sqrt{n})$ generalization error term by the decision-relevant intrinsic dimension $d^\star$.

\subsection{Learning decision-sufficient representation from contextual samples}
\label{sec:clo_stage1}

So far, our Stage~II analysis assumed access to the global decision-relevant subspace $W_\star$.
We now explain how Stage~I can be implemented from labeled contextual samples
$\{(\xi_i,c_i)\}_{i=1}^N$ and quantify the resulting \emph{representation-estimation} error.

Recall that for contextual linear optimization under the SPO loss, the Bayes-optimal decision rule is
$x^\star(\mu(\xi))$, where $\mu(\xi):=\EE[c\mid \xi]$.
Since $\C$ is convex and $c\in\C$ a.s., we have $\mu(\xi)\in\C$ a.s.
If we could draw i.i.d.\ samples from the (unobserved) distribution of $\mu(\xi)$, then we could run Algorithm~\ref{alg:cumulative} on those samples and---by our certificate guarantee (Theorem~\ref{thm:certificate})---obtain a dataset that is pointwise sufficient with high probability under the distribution of $\mu(\xi)$.
In practice, we only observe noisy costs $c$, so we proceed in two steps:
\begin{enumerate}[label=(\roman*)]
\item estimate $\mu$ from contextual samples via a $\C$-valued regression model $\hat\mu$, and
\item treat the predictions $\hat\mu(\xi)$ on fresh contexts as pseudo-cost samples and run Algorithm~\ref{alg:cumulative}.
\end{enumerate}
The process is summarized in Algorithm~\ref{alg:stage1_condmean}.

\begin{algorithm}[htbp]
\caption{Learning decision-sufficient representation from contextual samples (Stage I)}
\label{alg:stage1_condmean}
\begin{algorithmic}[1]
\REQUIRE Regression sample $\{(\xi_i,c_i)\}_{i=1}^{n_{\mu}}$, discovery contexts $\{\xi^{\disc}_j\}_{j=1}^{n_{\disc}}$.
\STATE Fit a regression model $\hat \mu$ for $\mu(\xi):=\EE[c\mid \xi]$ (e.g., the centered linear model as discussed).
\STATE Form pseudo-costs $\hat c_j\gets \hat\mu(\xi^{\disc}_j)$ for $j=1,\dots,n_{\disc}$.
\STATE Run Algorithm~\ref{alg:cumulative} on $\{\hat c_j\}_{j=1}^{n_{\disc}}$ and return $(\hat\Dset,T)$.
\end{algorithmic}
\end{algorithm}

\paragraph{A centered linear conditional-mean model.}
A convenient choice for $\hat\mu$ is multi-response ordinary least squares (OLS).
Because the prior set is the shifted ellipsoid $\C=\{c:(c-c_0)^\top\Sigma^{-1}(c-c_0)\le 1\}$, it is natural to write the conditional-mean model in centered form around $c_0$:
\begin{equation}
\label{eq:centered_condmean_model}
c-c_0 \;=\; A_\mu \xi + \epsilon,
\qquad
\EE[\epsilon\mid \xi]=0,
\qquad
\mu(\xi)=\EE[c\mid \xi]=c_0 + A_\mu\xi.
\end{equation}
Equivalently, we regress the centered response $y:=c-c_0$ onto $\xi$, and then set $\hat\mu(\xi):=c_0+\hat A_\mu\xi$.
We quantify the regression accuracy via the mean-squared prediction error $\varepsilon_\mu^2:=\EE_{\xi}\bigl[\|\hat\mu(\xi)-\mu(\xi)\|_2^2\bigr].$
For the sharpened Stage~I analysis below, we will use a stronger high-probability control: under bounded-design OLS, Appendix~\ref{app:bound_ols} (Lemma~\ref{lem:ols_bound_eps_mu}) provides simultaneous bounds on $\|\hat A_\mu-A_\mu\|_F$, on the uniform prediction radius $\sup_{\|\xi\|_2\le 1}\|\hat\mu(\xi)-\mu(\xi)\|_2$, and hence on $\varepsilon_\mu$.

\paragraph{A lifted compressed predictor.}
Let $\hat W:=\spanop(\hat\Dset)$ and let $\hat U\in\RR^{d\times t}$ be an orthonormal basis of $\hat W$, where $t :=\dim(\hat W)$.
Define the \emph{lifted compressed} conditional-mean predictor
\begin{equation}
\label{eq:tilde_mu_def}
\tilde \mu(\xi)
\ :=\
\operatorname{lift}_{\hat U}\!\left(\hat U^\top\bigl(\hat\mu(\xi)-c_0\bigr)\right)
=
c_0+\mathcal{L}_{\hat U}\hat U^\top\bigl(\hat\mu(\xi)-c_0\bigr).
\end{equation}
Under the centered linear model \eqref{eq:centered_condmean_model}, we have $\hat\mu(\xi)-c_0=\hat A_\mu\xi$ by the definition of $\hat\mu$, hence
\begin{equation}
\label{eq:tilde_mu_in_H}
\tilde\mu(\xi)
=
c_0+\mathcal{L}_{\hat U}\,(\hat U^\top \hat A_\mu)\,\xi.
\end{equation}
Therefore $\tilde\mu$ belongs to the compressed affine-linear class $\mathcal{H}_{\hat U,t}:=\{f_B(\xi)=c_0+\mathcal{L}_{\hat U}B\xi:\ B\in\RR^{t\times p}\}$.

\paragraph{A cone-boundary margin condition.}
The oracle $x^\star(\cdot)$ is locally constant on the interior of each optimality cone of~$\X$ and may change discontinuously on boundaries between cones.
Define the \emph{cone-boundary set}
\[
\mathcal{B}_{\X}
\ :=\
\left\{c\in\RR^d:\ \exists\,x\neq x'\in\X^\angle\ \text{s.t.}\ x,x'\in\argmin_{x\in\X}c^\top x\right\}.
\]
Our transfer argument from regression error to decision error relies on the standard margin assumption:

\begin{assumption}[Margin condition]
\label{assump:cone_boundary_margin}
There exist constants $C_{\mathrm{marg}}>0$ and $\alpha>0$ such that, for all $\eta>0$,
\[
\PP_{\xi}\!\left[\dist\bigl(\mu(\xi),\mathcal{B}_{\X}\bigr)\le \eta\right]\ \le\ C_{\mathrm{marg}}\,\eta^{\alpha}.
\]

Moreover, almost surely over the realized Stage-I samples, the Stage-I
regression predictor satisfies
\[
\PP_{\xi}\!\left[\hat \mu(\xi)\in \C \right] = 1,
\qquad
\PP_{\xi}\!\left[\tilde \mu(\xi)\in \mathcal{B}_{\X}\right] = 0.
\]

\end{assumption}

\paragraph{Stage-I representation error.}
The next theorem bounds the additional decision loss introduced by learning the representation from contextual samples.
Under the bounded-design OLS conditions from Appendix~\ref{app:bound_ols}, the decision error decomposes into
(i)~a \emph{certificate error} term due to running Algorithm~\ref{alg:cumulative} on finitely many pseudo-costs, and
(ii)~a \emph{regression-to-decision} term controlled by a uniform prediction radius and the cone-boundary margin.

\begin{theorem}[Stage-I representation error under bounded-design OLS]
\label{thm:misspec_plus_gen_condmean}
Suppose Stage~I runs Algorithm~\ref{alg:stage1_condmean} with discovery sample size $n_{\disc}$ and returns $(\hat\Dset,T)$,
where $T$ is the compression subsequence produced by Algorithm~\ref{alg:cumulative}.
Under Assumption~\ref{assump:cone_boundary_margin}, suppose in addition that the bounded-design OLS conditions of Lemma~\ref{lem:ols_bound_eps_mu} hold. Fix $\delta_\mu,\delta\in(0,1)$ and define
\[
r_{\mu,\delta_\mu}
\ :=\
C_{\mathrm{reg}}\cdot\frac{\sigma}{\sqrt{\kappa}}
\sqrt{\frac{d\Bigl(p+\log\!\frac{4d}{\delta_\mu}\Bigr)}{n_\mu}}.
\]
Then, with probability at least $1-\delta_\mu-\delta$ over the regression sample and the i.i.d.\ discovery contexts,
\begin{equation}
\label{eq:repr_term_mu}
\PP_{\xi}\bigl[x^\star(\tilde\mu(\xi))\neq x^\star(\mu(\xi))\bigr]
\ \le\
\frac{4}{n_{\disc}}\left(6|T|+\log(e/\delta)\right)
\ +\
C_{\mathrm{marg}}\,r_{\mu,\delta_\mu}^{\alpha}.
\end{equation}
In particular, since $|T|\le |\hat\Dset|\le \dstar$,
\[
\frac{4}{n_{\disc}}\left(6|T|+\log(e/\delta)\right)
\ \le\
\frac{4}{n_{\disc}}\left(6\dstar+\log(e/\delta)\right).
\]

Moreover, letting $f_\star(\xi):=\mu(\xi)$ be a Bayes-optimal SPO predictor and letting
$\hat f_\star\in\argmin_{f\in\mathcal{H}_{\hat U,t}}R_{\mathrm{SPO}}(f)$ be the best predictor restricted to the learned representation, we have the representation error bound
\begin{equation}
\label{eq:stage1_misspec_loss}
0\ \le\ R_{\mathrm{SPO}}(\hat f_\star)-R_{\mathrm{SPO}}(f_\star)
\ \le\ 
\omega_{\X}(\C)\cdot
\left[
\frac{4}{n_{\disc}}\left(6\dstar+\log(e/\delta)\right)
+
C_{\mathrm{marg}}\,r_{\mu,\delta_\mu}^{\alpha}
\right].
\end{equation}
\end{theorem}

\paragraph{Proof idea.}
The certificate term in \eqref{eq:repr_term_mu} follows by applying Theorem~\ref{thm:certificate} to the pseudo-cost sample
$\{\hat c_j=\hat\mu(\xi_j^{\disc})\}_{j=1}^{n_{\disc}}$.
The regression-to-decision term uses the tail-form transfer bound proved in Appendix~\ref{app:clo_condmean_stage1} together with
the uniform prediction-radius bound \eqref{eq:ols_bound_mu_sup} from Lemma~\ref{lem:ols_bound_eps_mu}: on the OLS high-probability event,
we take $\eta=r_{\mu,\delta_\mu}$ so that the regression tail vanishes and only the cone-boundary margin term remains.
The complete proof is in Appendix~\ref{app:clo_condmean_stage1}.

\paragraph{Stage~I representation error rate under OLS.}
In \Cref{alg:stage1_condmean}, we use $n_\mu$ contextual samples for regression and $n_{\disc}$ contextual samples for discovery.
Under the centered linear model \eqref{eq:centered_condmean_model}, Theorem~\ref{thm:misspec_plus_gen_condmean} yields an explicit bound for the
representation term in \eqref{eq:repr_term_mu}: up to logarithmic factors, the additional Stage~I representation error scales as
\[
\widetilde{O}\!\left(n_\mu^{-\alpha/2} + n_{\disc}^{-1}\right).
\]
Let $n_I := n_\mu + n_{\disc}$ denote the total sample size used in Stage~I.
Under a constant-fraction split (e.g., $n_\mu=\lfloor n_I/2\rfloor$ and $n_{\disc}=n_I-n_\mu$), we have
$n_\mu=\Theta(n_I)$ and $n_{\disc}=\Theta(n_I)$; hence this term simplifies to
\[
\widetilde{O}\!\left(n_I^{-1}+n_I^{-\alpha/2}\right)
=
\widetilde{O}\!\left(n_I^{-\min\{1,\alpha/2\}}\right).
\]
Therefore, if $\alpha>1$ and Stage~I and Stage~II use comparable sample sizes
(e.g., $n_I=\Theta(n_{\mathrm{train}})$), the additional Stage~I representation error is lower-order than the
$n_{\mathrm{train}}^{-1/2}$ generalization term in Stage~II.
Consequently, the overall statistical rate is governed by Stage~II, whose dominant term depends on the intrinsic dimension $\dstar$
rather than the ambient dimension $d$; see \cref{thm:misspec_plus_gen}. This improves sample efficiency for CLO.
Moreover, our decision-sufficient representation learning framework reduces the number of trainable parameters in Stage~II from $dp$
to $\dstar p$ as an additional advantage.

\subsection{Numerical experiment}
\label{sec:numerics}

We provide a small synthetic shortest-path CLO experiment to illustrate the sample-efficiency gains suggested by the intrinsic $\dstar$ dependence in \Cref{thm:misspec_plus_gen}.
We consider a monotone shortest-path instance on a $5\times 5$ grid ($g=5$), so the cost dimension is
$d=2g(g-1)=40$.
The feasible polytope has $|\X^\angle|=\binom{2(g-1)}{g-1}=70$ extreme points, each corresponding to a monotone path.
Contexts are drawn i.i.d.\ as $\xi\sim\mathcal N(0,I_p)$ with $p=5$.
We take $\C=\{c\in\RR^d:\|c-c_0\|_2\le 1\}$, where $c_0$ assigns cost $10$ on a fixed low-cost corridor and
$100$ elsewhere. This forces all shortest paths to remain within the corridor, and by enumeration on the
$5\times 5$ grid, the resulting intrinsic dimension is $\dstar=7$.

We compare (i) \textbf{full-$d$ SPO$+$}: a linear predictor $\hat c(\xi)\in\RR^d$ trained by SGD on the
SPO$+$ surrogate, and (ii) \textbf{ours (learn $\hat W$ then SPO$+$)}: first learn a subspace $\hat W$ online
from observed contexts and costs, and then train a reduced predictor in the learned subspace.
We use $n_{\mathrm{train}}=300$ labeled context--cost pairs for Stage~II and an independent test set of size
$n_{\mathrm{test}}=2000$, repeating over $10$ random trials and reporting mean $\pm$ $90\%$ confidence intervals.

Stage~I: \Cref{fig:numerics_dimW} reports the learned dimension $t=\dim(\hat W)$.
Stage~II: \Cref{fig:numerics_spo_risk} plots the test SPO risk versus the number of labeled samples used to train the predictor.
Consistent with our theory, restricting training to the learned subspace improves performance at a fixed sample size,
and $\hat W$ quickly stabilizes near the true intrinsic dimension $\dstar=7$.

\vspace{10mm}

\begin{figure}[htbp]
    \centering
    \IfFileExists{learnW.png}{\includegraphics[width=0.8\linewidth]{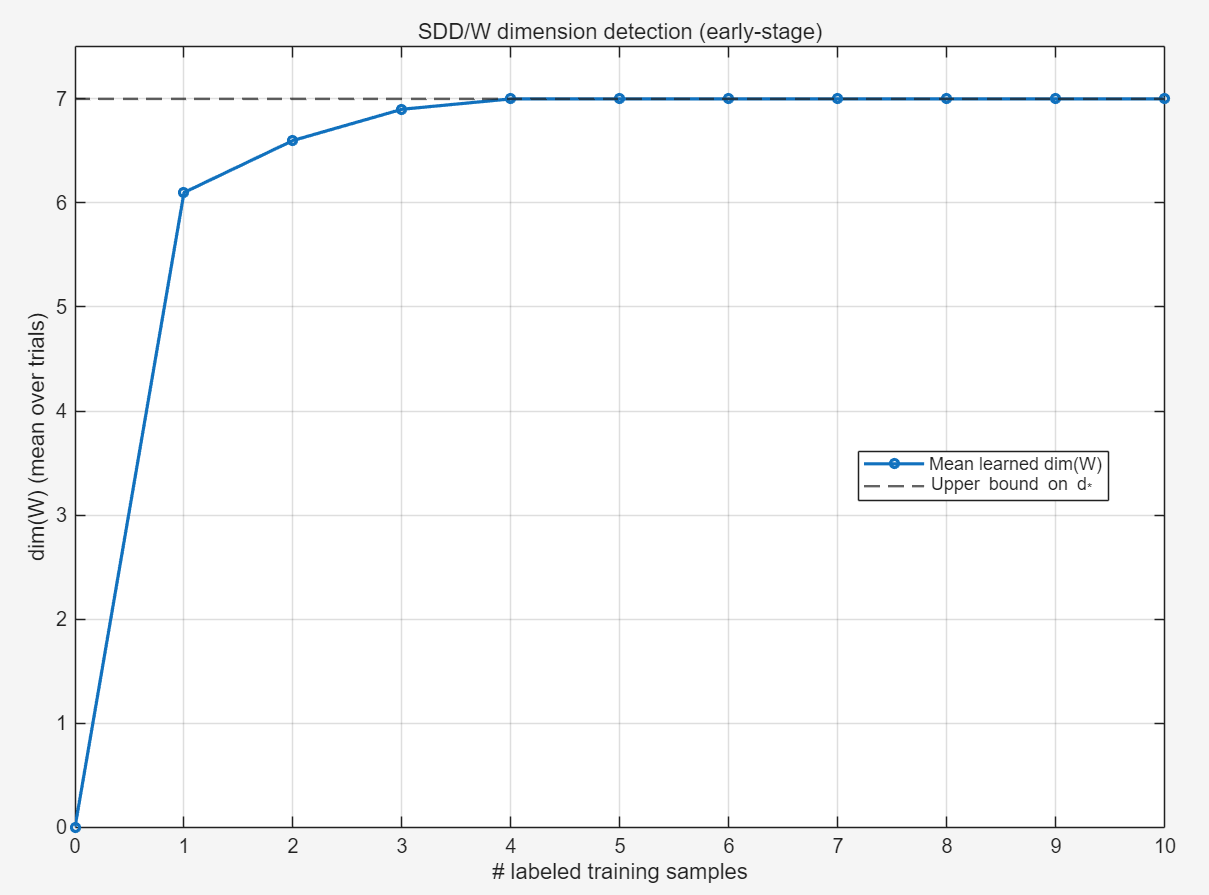}}{\fbox{\parbox{0.8\linewidth}{\centering\texttt{Missing file: learnW.png}}}}
    \caption{Stage I. Learned dimension $t=\dim(\hat W)$ (mean over $10$ trials).}
    \label{fig:numerics_dimW}
\end{figure}

\begin{figure}[htbp]
    \centering
    \IfFileExists{figureSPOloss.png}{\includegraphics[width=0.8\linewidth]{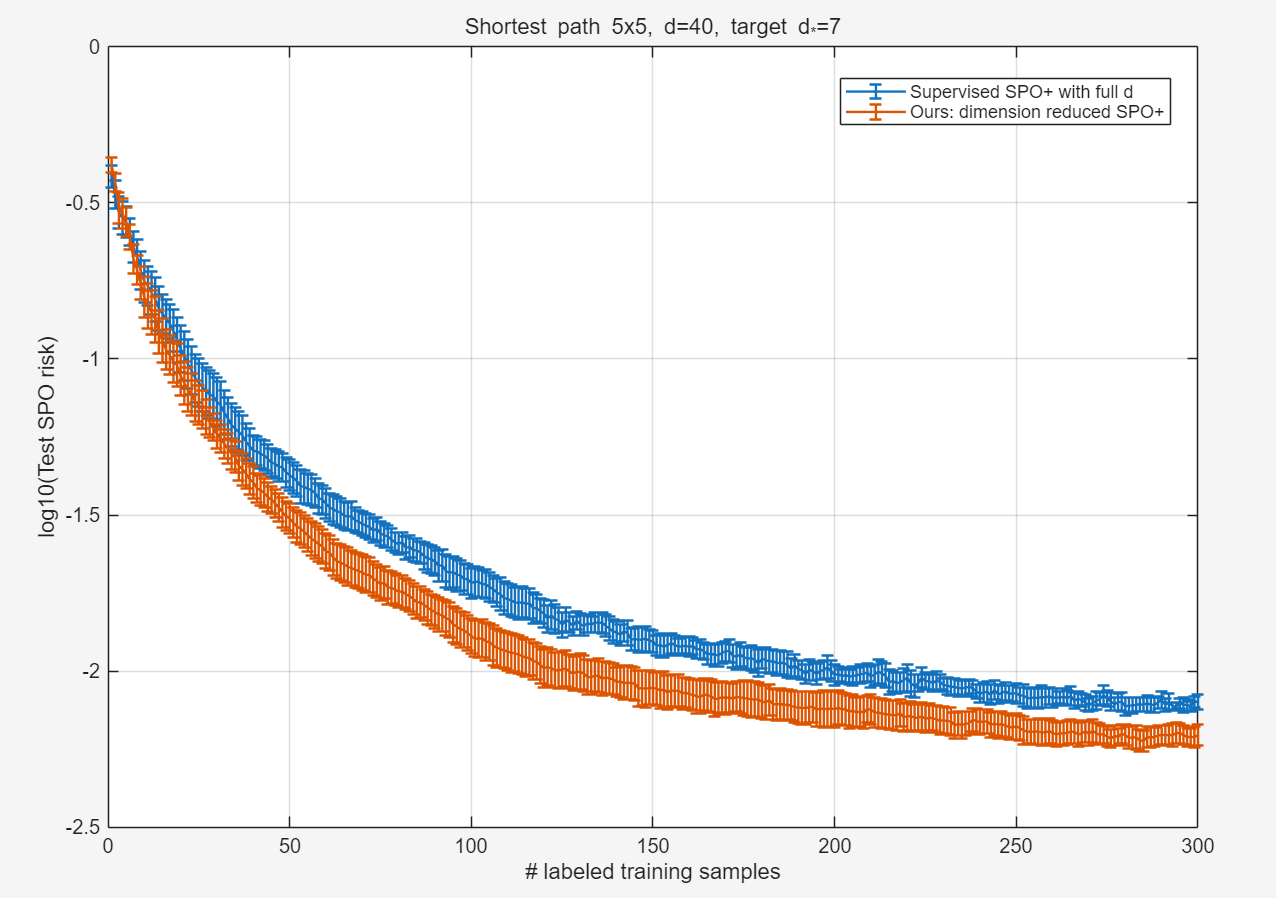}}{\fbox{\parbox{0.8\linewidth}{\centering\texttt{Missing file: figureSPOloss.png}}}}
    \caption{Stage II. SPO risk vs.\ number of labeled samples (mean $\pm$ $90\%$ CIs over $10$ trials).}
    \label{fig:numerics_spo_risk}
\end{figure}

\section{Concluding Remarks and Open Problems}
Our current hardness results for computing the intrinsic decision-relevant dimension $d^\star$ (and hence constructing minimum-size global SDDs)
rely on highly degenerate hard instances. It therefore remains open whether our hardness results persist under the nondegeneracy assumption on~$\X$. Second, our framework focuses on the noiseless setting, where each linear measurement
$q_i^\top c$ is observed exactly. Extending our algorithms to noisy observations and characterizing the resulting sample complexity are open directions. Intuitively, such an extension would likely require an additional
margin condition on $P_c$ that controls how often $c$ lies close to the boundary of an
optimality cone. Finally, we restrict attention to linear optimization in this paper. Extending the decision-sufficient
representation framework to broader problem classes, such as mixed-integer or convex programs,
is an important direction for future work, potentially via approximate notions of sufficiency.

\bibliographystyle{plainnat}
\bibliography{ref}
\newpage
\appendix

\section{Formal statement for hardness and other proofs for Section~\ref{sec:hardness}}
\label{app:hardness}

\subsection{Proof of Theorem~\ref{thm:np-hard-dir-dim}}
\label{app:np-hard-dir-dim}
\begin{theorem}[Formal version of Theorem~\ref{thm:np-hard-dir-dim}]
\label{thm:np-hard-dir-dim-formal}
Fix a coordinate index $r\in[n]$.
The following decision problem is $\mathsf{NP}$-hard: given a bounded polytope $X\subseteq\mathbb{R}^n$ and a polyhedral uncertainty set $\mathcal{C}\subseteq\mathbb{R}^n$ specified in $H$-representation, decide whether
\begin{equation}
\label{eq:np-hard-compare}
\dim\dir(X^\star(\mathcal{C})) \;>\; \dim\dir(X_r^\star(\mathcal{C})),
\end{equation}
where $X_r:=\{x\in X:\ x_r=0\}$ and $X_r^\star(\mathcal{C})$ is defined as in \Cref{def:dir-xstar} with $X$ replaced by $X_r$.
Consequently, computing $\dim\dir(X^\star(\mathcal{C}))$ is $\mathsf{NP}$-hard under polynomial-time Turing reductions.

The hardness persists even when $X$ is the $s$--$t$ unit-flow polytope of a directed acyclic graph and $\mathcal{C}$ is a budgeted set of
arc-length increases of the form
\[
\mathcal{C}_{\mathrm{cl}}=\{d+w:\ w\ge 0,\ \kappa^\top w\le B\}
\qquad\text{or}\qquad
\mathcal{C}_{\mathrm{op}}=\{d+w:\ w>0,\ \kappa^\top w < B+\eta\},
\]
for any fixed rational $\eta\in(0,1)$.
\end{theorem}

\begin{proof}
We prove NP-hardness via the \textsc{3-SAT}$\Leftrightarrow$\textsc{PISPP-W+} construction of
\citet[Theorem~3.1]{ley2025solutionmethods}.
Given a \textsc{3-SAT} instance $\varphi$, their reduction produces a partial inverse shortest path instance with only weight
increases (PISPP-W+), consisting of
a directed acyclic graph $G=(V,A)$ with source $s$ and sink $t$,
initial arc-lengths $d\in\mathbb{Q}^A_{\ge 0}$,
modification costs $\kappa\in\mathbb{Q}^A_{>0}$,
a budget $B\in\mathbb{Q}_{>0}$,
and a single required arc $r\in A$.
The equivalent PISPP-W+ decision question is whether there exists a modification vector $w\in\mathbb{Q}^A_{\ge 0}$ with
$\kappa^\top w\le B$ such that some shortest $s$--$t$ path with respect to the modified arc-lengths $d+w$
contains $r$. We proceed with the proof in three steps.
\begin{enumerate}[label=(\arabic*), leftmargin=1.5em]
\item \textbf{Restating the hard instance and key properties:} recall the explicit \textsc{3-SAT}$\to$\textsc{PISPP-W+} construction and the structural properties we will use.
\item \textbf{Openifying the uncertainty set:} show that passing from the closed budgeted set $\mathcal{C}_{\mathrm{cl}}$ to the slightly relaxed open set $\mathcal{C}_{\mathrm{op}}$ does not change the answer to the PISPP-W+ decision question on these instances.
\item \textbf{Dimension comparison:} encode $s$--$t$ paths as extreme points of a flow polytope and reduce the PISPP-W+ decision question to the strict inequality \eqref{eq:np-hard-compare}.
\end{enumerate}

\smallskip
\noindent\textbf{Step 1: Hard instance and structural properties.}
We first restate the explicit \textsc{3-SAT}$\to$\textsc{PISPP-W+} construction in the notation needed here.
Let $\varphi$ have variables $x_1,\dots,x_n$ and clauses $B_1,\dots,B_m$.
Write each clause as $B_j=\{\ell_{j1},\ell_{j2},\ell_{j3}\}$ where each literal
$\ell_{jk}\in\{x_i,\bar x_i:\ i\in[n]\}$ (if a clause has fewer than three literals, duplicate one).
For each $\ell_{jk}$ we create a \emph{clause--literal vertex} $b_{jk}$ labeled by $\ell_{jk}$.

\paragraph{Vertices.}
Let $V$ consist of
\[
\{s_0,\dots,s_n\}\ \cup\ \{t_0,\dots,t_m\}\ \cup\ \{x_i,\bar x_i:\ i\in[n]\}\ \cup\ \{b_{jk}:\ j\in[m],\,k\in[3]\}.
\]
Set the source $s:=s_0$ and the sink $t:=t_m$.

\paragraph{Arcs.}
We build a DAG $G=(V,A)$ with a variable layer (from $s_0$ to $s_n$) and a clause layer (from $t_0$ to $t_m$),
connected by one \emph{required} arc and additional \emph{shortcut} arcs.
\begin{itemize}[leftmargin=1.25em]
\item \emph{Variable gadgets:} for each $i\in[n]$, add the four arcs
\[
(s_{i-1},x_i),\ (x_i,s_i),\ (s_{i-1},\bar x_i),\ (\bar x_i,s_i).
\]
Thus, any $s_0$--$s_n$ path chooses exactly one of $\{x_i,\bar x_i\}$ for each $i$.

\item \emph{Clause gadgets:} for each $j\in[m]$ and $k\in[3]$, add
\[
(t_{j-1},b_{jk}),\ (b_{jk},t_j).
\]
Thus, any $t_0$--$t_m$ path chooses exactly one literal vertex $b_{jk}$ per clause $j$.

\item \emph{Required arc:} add $r:=(s_n,t_0)$, which is the unique required arc ($R=\{r\}$).

\item \emph{Shortcut arcs:}
for each clause--literal vertex $b_{jk}$ labeled by a literal on variable $x_i$, add exactly one arc
from the \emph{opposite} variable vertex into $b_{jk}$:
\[
\ell_{jk}=x_i\ \Rightarrow\ (\bar x_i,b_{jk})\in A,
\qquad
\ell_{jk}=\bar x_i\ \Rightarrow\ (x_i,b_{jk})\in A.
\]
Equivalently, the tail of the shortcut arc is the variable vertex that makes $\ell_{jk}$ \emph{false}.
\end{itemize}

\paragraph{Initial arc-lengths and modification costs.}
Set the initial lengths $d$ by
\[
d(a)=1\ \text{ for every non-shortcut arc }a,
\qquad
d(a)=2(n-i+j)\ \text{ for a shortcut arc }a=(\cdot,b_{jk})\text{ with }\ell_{jk}\in\{x_i,\bar x_i\}.
\]
Let the budget be $B:=n+2m$ and set modification costs $\kappa$ by
\[
\kappa_a=1\ \text{ for every non-shortcut arc }a,
\qquad
\kappa_a=B+1\ \text{ for every shortcut arc }a.
\]

\begin{lemma}[Structural properties]
\label{lem:ley-merkert-props}
The above instance satisfies the following properties (cf.\ Observations~1--4 and the claim in the proof of \cite[Theorem~3.1]{ley2025solutionmethods}):

\begin{enumerate}[label=(\roman*), leftmargin=1.5em]
\item \textbf{Degree constraints.}
Each $x_i$ and $\bar x_i$ has exactly one ingoing arc (from $s_{i-1}$), and each $b_{jk}$ has exactly one outgoing arc (to $t_j$).

\item \textbf{Path structure (``$r$ or one shortcut'').}
Every $s$--$t$ path contains either the required arc $r$ or exactly one shortcut arc (but not both).

\item \textbf{Unit gap under $d$.}
Every $s$--$t$ path using $r$ has length $2n+2m+1$ under $d$, while every $s$--$t$ path using a shortcut arc has
length $2n+2m$ under $d$.
In particular (with $w=0$), all shortest $s$--$t$ paths avoid $r$.

\item \textbf{Encoding of assignments + ``false literal $\Rightarrow$ available shortcut''.}
If an $s$--$t$ path $P$ uses $r$, then it visits exactly one of $\{x_i,\bar x_i\}$ for every $i\in[n]$
and exactly one $b_{jk}$ in every clause gadget $j\in[m]$.
Define the induced truth assignment by $x_i^P=1$ iff $x_i\in P$.
If $P$ visits a literal vertex $b_{jk}$ that is false under $x^P$, then the unique shortcut arc entering $b_{jk}$
has its tail on $P$.

\item \textbf{3-SAT $\Leftrightarrow$ PISPP-W+.} The mapping
$\varphi \mapsto (G,d,\kappa,B,r)$ is computable in polynomial time and satisfies the following
equivalence:
the formula $\varphi$ is satisfiable if and only if there exists a modification vector
$w\in\mathbb{Q}^A_{\ge 0}$ with $\kappa^\top w \le B$ such that some shortest $s$--$t$ path
with respect to the modified arc-lengths $d+w$ contains the required arc $r$.
Equivalently, if $\varphi$ is unsatisfiable then for every $w\in\mathbb{Q}^A_{\ge 0}$ with
$\kappa^\top w \le B$, every shortest $s$--$t$ path under $d+w$ avoids $r$.
\end{enumerate}
\end{lemma}
\begin{proof}
(i) is immediate from the arc construction.
(ii) The only arcs that can enter the clause layer are $r=(s_n,t_0)$ and the shortcut arcs into some $b_{jk}$.
Once the path enters the clause layer, it cannot return to the variable layer (there are no such arcs),
and each $b_{jk}$ has only the outgoing arc $(b_{jk},t_j)$; hence no $s$--$t$ path can use two shortcuts, and no path can use both
$r$ and a shortcut.
(iii) Any $s$--$t$ path using $r$ traverses $2n$ unit-length arcs in the variable gadgets, then $r$ (length $1$),
then $2m$ unit-length arcs in the clause gadgets, totaling $2n+2m+1$.
A path using a shortcut from variable $i$ into clause $j$ has prefix length $2(i-1)+1$ up to the tail variable vertex,
then shortcut length $2(n-i+j)$, then suffix length $1+2(m-j)$ in the clause layer, totaling
\[
(2(i-1)+1)+2(n-i+j)+(1+2(m-j))=2n+2m.
\]

(iv) The first statement follows since each variable gadget offers exactly two disjoint choices and each clause gadget offers exactly three disjoint choices.
For the last statement, if $b_{jk}$ is false under $x^P$, then $P$ must have visited the opposite variable vertex
($\bar x_i$ if $\ell_{jk}=x_i$, and $x_i$ if $\ell_{jk}=\bar x_i$), which is exactly the tail of the unique shortcut arc into $b_{jk}$.

(v) is exactly the claim of Theorem~3.1 in \citet{ley2025solutionmethods}, which shows NP-completeness of PISPP-W+. We provide a proof sketch here:

($\Rightarrow$) Let $x^*$ satisfy $\varphi$. W.l.o.g.\ reorder literals in each clause so that $b_{j1}$ is true.
Build the $s$--$t$ path $P$ that uses $r$ by following $x^*$ in the variable gadgets and $b_{j1}$ in each clause gadget.
Set $w=1$ on the unique entry arc into the unchosen variable vertex (one per variable) and set $w=1$ on
$(b_{j2},t_j)$ and $(b_{j3},t_j)$ (two per clause); set $w=0$ on all remaining arcs (in particular on shortcut arcs).
Then $\kappa^\top w = n+2m = B$, and any shortcut path must traverse at least one penalized arc, so its length increases by $\ge 1$,
closing the initial gap $d(Q)=d(P)-1$ and implying $P$ is (one of) the shortest paths.

($\Leftarrow$) Suppose $\kappa^\top w\le B$ and some shortest path $P$ under $d+w$ uses $r$; let $x^P$ be the assignment induced by $P$.
If some clause is false under $x^P$, then the visited literal vertex $b_{jk}$ is false, and by (iv) the entering shortcut arc $e$
has its tail on $P$. Replacing the corresponding subpath $R$ of $P$ by $e$ gives a shortcut path $Q$ with $d(Q)=d(P)-1$ (by (iii)).
Thus $(d+w)(Q)-(d+w)(P) = -1 + w(e) - w(R) \le -1 + w(e)$, so shortestness of $P$ forces $w(e)\ge 1$.
Since $e$ is a shortcut arc with $\kappa_e=B+1$, this implies $\kappa^\top w \ge \kappa_e w(e) > B$, a contradiction.
\end{proof}

\smallskip
\noindent\textbf{Step 2: Openifying the uncertainty set does not change the answer.}
This step is only needed to ensure hardness persists for an open uncertainty set.
We consider two uncertainty sets of admissible arc-length vectors:
\[
\mathcal{C}_{\mathrm{cl}}:=\{\,c=d+w:\ w\in\mathbb{R}^A_{\ge 0},\ \kappa^\top w\le B\,\}
\quad\text{and}\quad
\mathcal{C}_{\mathrm{op}}:=\{\,c=d+w:\ w\in\mathbb{R}^A_{>0},\ \kappa^\top w < B+\eta\,\},
\]
where we fix any rational constant $\eta\in(0,1)$ (e.g. $\eta=\tfrac12$).
The lemma below shows that for the hard instances above, replacing $\mathcal{C}_{\mathrm{cl}}$ by $\mathcal{C}_{\mathrm{op}}$
does not change the answer to the underlying question.

\begin{lemma}
\label{lem:openify-pispp}
Fix any $\eta\in(0,1)$. For the PISPP-W+ instance produced by \citet{ley2025solutionmethods} from $\varphi$,
the formula $\varphi$ is satisfiable if and only if there exists $c\in\mathcal{C}_{\mathrm{op}}$ such that
some shortest $s$--$t$ path with respect to $c$ contains $r$.
\end{lemma}
\begin{proof}
($\Rightarrow$)  Suppose $\varphi$ is satisfiable. By \Cref{lem:ley-merkert-props}(v),
there exists $w_0\in\mathbb{Q}^A_{\ge 0}$ with $\kappa^\top w_0 \le B$ such that for
$c_0 := d+w_0$ there is a shortest $s$--$t$ path using $r$.

Since $G$ is a DAG, fix any topological ordering $\rho:V\to\{0,1,\dots,|V|-1\}$.
For (a rational) $\varepsilon>0$, define a perturbation $\delta\in\mathbb{Q}^A$ by
\[
\delta(u,v):=\varepsilon(\rho(v)-\rho(u))\qquad\forall (u,v)\in A.
\]
Then $\delta>0$ componentwise. Moreover, for any $s$--$t$ path
$P=(v_0=s,v_1,\dots,v_k=t)$, the perturbation telescopes:
\[
\sum_{i=0}^{k-1}\delta(v_i,v_{i+1})
=\varepsilon\sum_{i=0}^{k-1}(\rho(v_{i+1})-\rho(v_i))
=\varepsilon(\rho(t)-\rho(s)),
\]
which is a constant independent of $P$. Hence, adding $\delta$ shifts the length of every
$s$--$t$ path by the same constant, so the set of shortest $s$--$t$ paths is unchanged when we
replace $c_0$ by $c_0+\delta$.

Finally, choose $\varepsilon$ small enough so that $\kappa^\top\delta<\eta$. For instance, since
$\rho(v)-\rho(u)\le |V|-1$ for every arc,
\[
\kappa^\top\delta=\sum_{a\in A}\kappa_a\delta_a
\le \varepsilon(|V|-1)\sum_{a\in A}\kappa_a,
\]
so it suffices to take $\varepsilon:=\eta/\bigl(2(|V|-1)\sum_{a\in A}\kappa_a\bigr)$.
With this choice, $w:=w_0+\delta$ satisfies $w>0$ and $\kappa^\top w < B+\eta$, i.e.
$c:=d+w\in \mathcal{C}_{\rm op}$, and there still exists a shortest path using $r$.

($\Leftarrow$) We prove the contrapositive. Suppose $\varphi$ is not satisfiable, assume that there exists a shortest $s$--$t$ path $P$ with respect to
$c:=d+w\in \mathcal{C}_{\rm op}$ that contains $r$.

By \Cref{lem:ley-merkert-props}(ii), $P$ contains no shortcut arc. Define the induced assignment $x^P$ as in
\Cref{lem:ley-merkert-props}(iv). Since $\varphi$ is unsatisfiable, there exists a clause index $j$ that is not satisfied
by $x^P$. By \Cref{lem:ley-merkert-props}(iv), $P$ visits exactly one literal vertex $b_{jk}$ in clause gadget $j$;
because clause $j$ is unsatisfied, this visited literal vertex $b_{jk}$ is false under $x^P$.
Therefore \Cref{lem:ley-merkert-props}(iv) yields that the unique shortcut arc $e$ entering $b_{jk}$ has its tail on $P$.

Let $R$ denote the (nonempty) subpath of $P$ from the tail of $e$ to the vertex $b_{jk}$, and define
a new $s$--$t$ path $Q$ by following $P$ up to the tail of $e$, then traversing $e$, and then
following $P$ from $b_{jk}$ to $t$. Then $Q$ uses a shortcut arc, so by \Cref{lem:ley-merkert-props}(iii) we have
$d(Q)=d(P)-1$.

Since $P$ and $Q$ coincide outside of $R$ and $e$, we obtain
\[
(d+w)(Q)-(d+w)(P)=(d(Q)-d(P))+(w(Q)-w(P))=-1+w(e)-w(R).
\]
Because $R$ contains at least one arc and $w>0$ componentwise, we have $w(R)>0$.
Thus if $P$ is shortest we must have $0\le (d+w)(Q)-(d+w)(P)$, implying $w(e)>1$.

As $e$ is a shortcut arc, $\kappa_e=B+1$, hence
\[
\kappa^\top w \ge \kappa_e w(e) > (B+1)\cdot 1 = B+1.
\]
In particular, for any $\eta\in(0,1)$ we have $\kappa^\top w > B+\eta$, contradicting the
assumption $\kappa^\top w < B+\eta$ required for $d+w\in \mathcal{C}_{\rm op}$. Therefore, no such
$w>0$ can make a shortest $s$--$t$ path use $r$.
\end{proof}

\smallskip
\noindent\textbf{Step 3: Reduction from \textsc{PISPP-W+} to a comparison of decision-relevant dimensions.}
We now translate the PISPP-W+ instance into a linear optimization problem over the $s$--$t$ unit-flow polytope. The coordinate indexed by the required arc $r$
will play the role of the distinguished coordinate in \eqref{eq:np-hard-compare}, and we will compare (the dimensions of) reachable optimal directions with and without imposing $x_r=0$.

Given the instance $(G,d,\kappa,B,r)$, let $p:=|A|$ and index coordinates of $\mathbb{R}^p$ by arcs.
Define the $s$--$t$ unit-flow polytope
\[
X:=\Bigl\{x\in\mathbb{R}^p_{\ge 0}:\ \sum_{a\in\delta^+(v)}x_a-\sum_{a\in\delta^-(v)}x_a=
\begin{cases}
1 & v=s,\\
-1& v=t,\\
0 & \text{otherwise,}
\end{cases}\Bigr\}.
\]
Because $G$ is acyclic, the extreme points of $X$ are exactly the incidence vectors of $s$--$t$ paths.
Let $X_r:=\{x\in X:\ x_r=0\}$ be the face of flows that avoid the required arc. Note that $X_r$ is a face of $X$, hence
\begin{equation}
\label{eq111}
X_r^\angle = X^\angle \cap X_r = \{x\in X^\angle:\ x_r=0\}.
\end{equation}
For $\mathcal{C}\in\{\mathcal{C}_{\mathrm{cl}},\mathcal{C}_{\mathrm{op}}\}$, define the reachable optimal sets
$X^\star(\mathcal{C})$ and $X_r^\star(\mathcal{C})$ as in \Cref{def:dir-xstar}. 

The next lemma is the key structural property for the dimension comparison: it identifies the reachable optimal extreme points of the restricted face $X_r$ and shows that all of them are also reachable for $X$.

\begin{lemma}
\label{lem:identify-xrstar}
For $\mathcal{C}\in\{\mathcal{C}_{\mathrm{cl}},\mathcal{C}_{\mathrm{op}}\}$ constructed above, we have
\[
X_r^\star(\mathcal{C})=X_r^\angle
\qquad\text{and}\qquad
X_r^\angle\subseteq X^\star(\mathcal{C}).
\]
\end{lemma}
\begin{proof}
For $\mathcal{C}=\mathcal{C}_{\mathrm{cl}}$, let $c^0:=d\in\mathcal{C}$ (take $w=0$).
For $\mathcal{C}=\mathcal{C}_{\mathrm{op}}$, let $\delta$ be the topological perturbation from the proof of
Lemma~\ref{lem:openify-pispp} and set $c^0:=d+\delta\in\mathcal{C}_{\mathrm{op}}$.
In either case, by the unit-gap property under $d$ and by telescoping of $\delta$, every $s$--$t$ path that avoids $r$ is
shortest under $c^0$, and every path that uses $r$ is strictly longer. Combine with \eqref{eq111}, we have
\[
X^\star(c^0)=\{x\in X^\angle:x_r=0\}=X_r^\angle,
\]
and in particular $X_r^\angle\subseteq X^\star(\mathcal{C})$.

Moreover, for this same $c^0$, every extreme point of $X_r$ (i.e., every $s$--$t$ path avoiding $r$) is optimal for
$\min\{(c^0)^\top x:x\in X_r\}$, so $X_r^\angle\subseteq X_r^\star(\mathcal{C})$.
The reverse inclusion $X_r^\star(\mathcal{C})\subseteq X_r^\angle$ holds by definition, hence $X_r^\star(\mathcal{C})=X_r^\angle$.
\end{proof}

We are now ready to complete the reduction. To prove \Cref{thm:np-hard-dir-dim-formal}, it suffices to establish the following claim.

\begin{claim}
\label{clm:sat-iff-dim}
The formula $\varphi$ is satisfiable if and only if
$\dim\dir(X^\star(\mathcal{C}))>\dim\dir(X_r^\star(\mathcal{C}))$.
\end{claim}
\begin{proof}
By \cref{lem:identify-xrstar}, we always have $X_r^\star(\mathcal{C})\subseteq X^\star(\mathcal{C})$, hence
\begin{equation}
\label{eq:dir-inclusion}
\dir(X_r^\star(\mathcal{C}))\subseteq \dir(X^\star(\mathcal{C})).
\end{equation}

\smallskip
\emph{($\Rightarrow$)} 
Assume $\varphi$ is satisfiable.
We claim that there exists $c^+\in\mathcal{C}$ such that some shortest $s$--$t$ path w.r.t.\ $c^+$ contains $r$.
If $\mathcal{C}=\mathcal{C}_{\mathrm{cl}}$, this follows from \Cref{lem:ley-merkert-props}(v)
(equivalently, \citet[Theorem~3.1]{ley2025solutionmethods});
if $\mathcal{C}=\mathcal{C}_{\mathrm{op}}$, it follows from Lemma~\ref{lem:openify-pispp}.

Let $x^+\in X^\star(c^+)\cap X^\angle\subseteq X^\star(\mathcal{C})$ be the incidence vector of such a shortest path,
so $x^+_r=1$.
By Lemma~\ref{lem:identify-xrstar}, we also have $X_r^\angle\subseteq X^\star(\mathcal{C})$; pick any $x^0\in X_r^\angle$,
so $x^0_r=0$.
Then $v:=x^+-x^0\in\dir\!\bigl(X^\star(\mathcal{C})\bigr)$ and satisfies $v_r=1$.

On the other hand, every $u\in\dir\!\bigl(X_r^\star(\mathcal{C})\bigr)$ satisfies $u_r=0$, and therefore
$v\notin\dir\!\bigl(X_r^\star(\mathcal{C})\bigr)$.
Combining this with \eqref{eq:dir-inclusion} yields
\[
\dir\!\bigl(X_r^\star(\mathcal{C})\bigr)\subsetneq \dir\!\bigl(X^\star(\mathcal{C})\bigr),
\]
and thus $\dim\dir(X^\star(\mathcal{C}))>\dim\dir(X_r^\star(\mathcal{C}))$.

\smallskip
\emph{($\Leftarrow$)} We prove the contrapositive.
Assume $\varphi$ is not satisfiable. Then, for $\mathcal{C}=\mathcal{C}_{\mathrm{cl}}$ the corresponding PISPP-W+ instance
is infeasible by \citet{ley2025solutionmethods}, and for $\mathcal{C}=\mathcal{C}_{\mathrm{op}}$ it is infeasible by
Lemma~\ref{lem:openify-pispp}.
Equivalently, for every $c\in\mathcal C$, every extreme optimal solution of $\min\{c^\top x:x\in X\}$
avoids the arc $r$, i.e., satisfies $x_r=0$.
Since $X_r$ is a face of $X$, we have $X_r^\angle=\{x\in X^\angle:x_r=0\}$, and therefore
\[
X^\star(c)\cap X^\angle \subseteq X_r^\angle\qquad\forall\,c\in\mathcal C.
\]
Taking the union over $c\in\mathcal C$ gives $X^\star(\mathcal C)\subseteq X_r^\angle$.
On the other hand, Lemma~\ref{lem:identify-xrstar} yields $X_r^\angle\subseteq X^\star(\mathcal C)$ and
$X_r^\star(\mathcal C)=X_r^\angle$.
Hence $X^\star(\mathcal C)=X_r^\star(\mathcal C)$, and thus
$\dir(X^\star(\mathcal C))=\dir(X_r^\star(\mathcal C))$.
\end{proof}

This completes the reduction. Therefore deciding whether \eqref{eq:np-hard-compare} holds is $\mathsf{NP}$-hard.

Finally, to deduce the hardness of computing $\dim\dir(X^\star(\mathcal{C}))$ as a function problem, note that if we could compute
$\dim\dir(X^\star(\mathcal{C}))$ in polynomial time, we could compute both sides of \eqref{eq:np-hard-compare} (one call on $(X,\mathcal{C})$
and one call on $(X_r,\mathcal{C})$) and decide the inequality, implying $\mathsf{NP}$-hardness under polynomial-time Turing reductions.
\end{proof}

\subsection{Proof of Theorem~\ref{thm:coNP-hard-pointwise-check}}
\label{app:coNP-hard-pointwise-check}
\begin{theorem}[Formal version of Theorem~\ref{thm:coNP-hard-pointwise-check}]
The following decision problem is $\mathsf{coNP}$-hard: given a bounded polytope $X\subseteq\RR^n$, a polyhedral uncertainty set $\mathcal{C}\subseteq\RR^n$ specified in $H$-representation, a dataset $\Dset$,
and a cost vector $c\in\mathcal{C}$, decide whether $\Dset$ is pointwise sufficient at $c$ in the sense of Definition~\ref{def:pointwise-sdd}.
Consequently, computing the size of a minimum pointwise SDD (and the corresponding search problem of finding one) is
$\mathsf{coNP}$-hard.
\end{theorem}

\paragraph{The $H$-in-$V$ polytope containment problem.}
An instance consists of two polytopes $P,Q\subseteq\RR^d$, where
\[
P=\{z\in\RR^d:\ Hz\le h\}\qquad\text{and}\qquad Q=\conv\{v_1,\dots,v_M\},
\]
i.e., $P$ is given in \emph{$H$-representation} and $Q$ is given in \emph{$V$-representation}.
The decision problem asks whether $P\subseteq Q$.
This problem is $\mathsf{coNP}$-complete \citep{freundorlin1985complexity}.
Moreover, the $\mathsf{coNP}$-hardness persists under strong structural restrictions; in particular, it is already
$\mathsf{coNP}$-complete to decide whether the standard cube is contained in an affine image of a cross polytope
\citep{gritzmannklee1993radii}. A convenient modern reference is \citet[Proposition~2.1]{kellner2016sos}.
In our reduction, we work with the following convenient hard family: given $Q=\conv\{v_1,\dots,v_M\}\subseteq\RR^d$
that is full-dimensional and satisfies $0\in\operatorname{int}(Q)$, decide whether the hypercube $P=[-1,1]^d$ is
contained in $Q$.

\paragraph{Construction of a standard-form LP instance.}
Let $P=[-1,1]^d$ and $Q=\conv\{v_1,\dots,v_M\}$ be such a hard instance.
Set $n_0:=d+1$ and define the homogenized vectors $\bar v_i:=(v_i,1)\in\RR^{n_0}$.
Let $\bar V\in\RR^{M\times n_0}$ be the matrix whose $i$th row is $\bar v_i^\top$, and set $\beta:=\bar V\1\in\RR^M$, where 
$\1$ denotes the all-ones vector. 
We introduce variables $w,r\in\RR^{n_0}$ and $s\in\RR^M$, and write $z:=(w,r,s)\in\RR^n$ with $n:=2n_0+M$.
Define the standard-form polytope
\begin{equation}
\label{eq:pointwise-hard-X}
X
:=\Bigl\{z\in\RR^n:\ 
\underbrace{\begin{bmatrix}
I_{n_0} & I_{n_0} & 0\\
\bar V  & 0      & -I_M
\end{bmatrix}}_{=:A}\,
z
=
\underbrace{\begin{bmatrix}
2\cdot \1\\
\beta
\end{bmatrix}}_{=:b},
\ \ z\ge 0
\Bigr\}.
\end{equation}
The matrix $A$ has full row rank, and $X$ is nonempty and bounded (indeed, $w+r=2\cdot \1$ implies $0\le w\le 2\cdot \1$ and $0\le r\le 2\cdot \1$,
while $s=\bar Vw-\beta$ is then bounded as well).
Let
\[
z_0:=(\1,\1,0)\in X.
\]
We consider the empty dataset $\Dset=\varnothing$ and the cost vector
\[
c_0 := (e_{n_0},0,0)\in\RR^{n},
\]
where $e_{n_0}\in\RR^{n_0}$ is the last standard basis vector.
Finally, define the polyhedral uncertainty set
\[
\mathcal{C}:=\bigl\{\,((y,1),0,0)\in\RR^n:\ y\in P\,\bigr\}.
\]
Note that $c_0\in\mathcal{C}$ since $0\in[-1,1]^d$.

We claim that $\Dset=\varnothing$ is pointwise sufficient at $c_0$ if and only if $P\subseteq Q$.

\paragraph{Step 1: the optimality cone at $z_0$.}
For a standard-form polytope $X=\{z:Az=b,\ z\ge0\}$ and an extreme point $z_0$, the optimality cone can be written via KKT as
\begin{equation}
\label{eq:Lambda-standard-form}
\Lambda(z_0)
=
\Bigl\{c\in\RR^n:\ \exists\,y,\rho\ \text{s.t.}\ c=A^\top y+\rho,\ \rho\ge0,\ \rho_j=0\ \text{whenever}\ (z_0)_j>0\Bigr\}.
\end{equation}
Here $z_0=(\1,\1,0)$ has strictly positive components on the $w$- and $r$-coordinates and zeros on the $s$-coordinates.
Writing $y=(\alpha,\mu)\in\RR^{n_0}\times\RR^M$, we have
\[
A^\top y = \bigl(\alpha+\bar V^\top\mu,\ \alpha,\ -\mu\bigr).
\]
Therefore, for costs of the form $((\tilde y,\tilde t),0,0)$ we obtain the characterization
\begin{equation}
\label{eq:Lambda-slice}
((\tilde y,\tilde t),0,0)\in\Lambda(z_0)
\quad\Longleftrightarrow\quad
(\tilde y,\tilde t)\in\cone\{\bar v_1,\dots,\bar v_M\}.
\end{equation}

\paragraph{Step 2: $z_0$ is the unique minimizer for $c_0$.}
Let $u:=(x,t):=w-\1\in\RR^d\times\RR$.
From \eqref{eq:pointwise-hard-X}, feasibility implies $w\in[0,2]^{n_0}$ and
\[
s_i=\bar v_i^\top w-\beta_i=\bar v_i^\top(w-\1)=v_i^\top x+t\ \ge 0
\qquad\forall i\in[M].
\]
Thus the projection $z\mapsto u=w-\1$ identifies $X$ with the bounded polytope
\[
\widetilde X
:=\{(x,t)\in\RR^{d}\times\RR:\ v_i^\top x+t\ge 0\ \forall i\in[M],\ \ -1\le (x,t)\le 1\}.
\]
Moreover, minimizing $c_0^\top z$ over $X$ is equivalent (up to an additive constant) to minimizing $t$ over $\widetilde X$,
since $c_0^\top z = e_{n_0}^\top w = t+1$.

We now show that $(0,0)$ is the unique minimizer of $\min\{t:(x,t)\in\widetilde X\}$.
First, we claim that every feasible point satisfies $t\ge 0$.
Indeed, if $t<0$ then $v_i^\top x\ge -t>0$ for all $i$, hence $z^\top x>0$ for all $z\in Q=\conv\{v_1,\dots,v_M\}$.
But $0\in\operatorname{int}(Q)$ implies that for any $x\neq 0$ there exists $\varepsilon>0$ such that
$-\varepsilon x/\|x\|\in Q$, which yields $(-\varepsilon x/\|x\|)^\top x<0$, a contradiction. Thus $t\ge 0$.

Therefore the minimum value of $t$ equals $0$ (since $(0,0)\in\widetilde X$).
When $t=0$, feasibility requires $v_i^\top x\ge 0$ for all $i$, which again forces $x=0$ by the same argument using
$0\in\operatorname{int}(Q)$.
Hence $(x,t)=(0,0)$ is the unique minimizer in $\widetilde X$, and consequently $z_0=(\1,\1,0)$ is the unique minimizer in $X$:
\[
X^\star(c_0)=\{z_0\}.
\]

\paragraph{Step 3: pointwise sufficiency reduces to cone containment.}
Since $\Dset=\varnothing$, the data-consistent fiber equals $\mathcal C$.
Pointwise sufficiency at $c_0$ requires a decision $z^\star\in X$ that is optimal for all costs in $\mathcal C$
(and in particular for $c_0$).
By Step~2, this forces $z^\star=z_0$.
Therefore,
\[
\Dset=\varnothing\ \text{is pointwise sufficient at }c_0
\quad\Longleftrightarrow\quad
z_0\in X^\star(c)\ \ \forall c\in\mathcal{C}
\quad\Longleftrightarrow\quad
\mathcal{C}\subseteq \Lambda(z_0).
\]

\paragraph{Step 4: $\mathcal{C}\subseteq \Lambda(z_0)$ iff $P\subseteq Q$.}
By \eqref{eq:Lambda-slice}, for any $y\in\RR^d$ we have
\[
((y,1),0,0)\in \Lambda(z_0)
\quad\Longleftrightarrow\quad
(y,1)\in \cone\{\bar v_1,\dots,\bar v_M\}
\quad\Longleftrightarrow\quad
y\in Q,
\]
where the last equivalence uses that $(y,1)=\sum_{i=1}^M\alpha_i(v_i,1)$ with $\alpha_i\ge 0$ holds if and only if
$\sum_{i=1}^M\alpha_i=1$ and $y=\sum_{i=1}^M\alpha_i v_i$.
Applying this pointwise over $y\in P=[-1,1]^d$ yields $\mathcal{C}\subseteq\Lambda(z_0)\iff P\subseteq Q$.

This completes a polynomial-time reduction from $H$-in-$V$ containment to checking pointwise sufficiency, proving
$\mathsf{coNP}$-hardness.

Finally, since our reduction uses $\Dset=\varnothing$, deciding whether the minimum pointwise SDD size equals $0$ is already
$\mathsf{coNP}$-hard; hence, computing the minimum size (and producing a minimum pointwise SDD) is $\mathsf{coNP}$-hard.

\subsection{Proof of Theorem~\ref{thm:conp-hard-global0}}
\label{app:conp-hard-global}
\begin{theorem}[Formal version of Theorem~\ref{thm:conp-hard-global0}]
The following decision problem is $\mathsf{coNP}$-hard: given a bounded polytope $X\subseteq\RR^n$ and a polyhedral uncertainty set $\mathcal{C}\subseteq\RR^n$ specified in $H$-representation,
decide whether the empty dataset $\Dset=\varnothing$ is a global SDD for $(X,\mathcal C)$ in the sense of Definition~\ref{def:bennouna-sdd}.
The hardness persists even when $\mathcal C$ is an open, full-dimensional polyhedron.
Consequently, computing the size of a minimum global SDD (and the corresponding search problem of finding one) is
$\mathsf{coNP}$-hard.
\end{theorem}

\begin{proof}
We reduce from the same restricted $H$-in-$V$ polytope containment problem used in the proof of
\Cref{thm:coNP-hard-pointwise-check}: given a full-dimensional polytope $Q=\conv\{v_1,\dots,v_M\}\subseteq\RR^d$ with
$0\in\operatorname{int}(Q)$, decide whether $P=[-1,1]^d\subseteq Q$, which is $\mathsf{coNP}$-complete. We reuse the standard-form polytope $X$ from \eqref{eq:pointwise-hard-X}.
In particular, $X=\{z=(w,r,s)\in\RR^{n}:Az=b,\ z\ge0\}$ with $n=2(d+1)+M$, and it contains the distinguished vertex
$z_0=(\1,\1,0)$.

\paragraph{A full-dimensional open uncertainty set.}
Write costs as $c=(c_w,c_r,c_s)\in\RR^{n_0}\times\RR^{n_0}\times\RR^{M}$ with $n_0=d+1$.
Define the \emph{effective cost on $w$} by the linear map
\begin{equation}
\label{eq:effective-cost-map}
T(c)\;:=\;c_w-c_r+\bar V^\top c_s\ \in\ \RR^{n_0}.
\end{equation}
Indeed, for any feasible $z=(w,r,s)\in X$ we have $r=2\cdot \1-w$ and $s=\bar Vw-\beta$, so
\begin{equation}
\label{eq:objective-reduction}
c^\top z
=
c_w^\top w + c_r^\top(2\cdot \1-w) + c_s^\top(\bar Vw-\beta)
=
T(c)^\top w\ +\ \underbrace{2c_r^\top\1 - c_s^\top\beta}_{\text{constant over }X}.
\end{equation}
Therefore $\argmin_{z\in X} c^\top z$ depends on $c$ only through $T(c)$.

Let $B_{\mathrm{op}}\subseteq\RR^{n_0}$ denote the open polyhedron
\[
B_{\mathrm{op}}
:=\Bigl\{(\tilde y,\tilde t)\in\RR^d\times\RR:\ \tfrac12<\tilde t<\tfrac32,\ \ -\tilde t<\tilde y_j<\tilde t\ \ \forall j\in[d]\Bigr\}.
\]
We define the uncertainty set as the preimage
\begin{equation}
\label{eq:global0-Cop}
\mathcal C_{\mathrm{op}}
:=\{\,c\in\RR^{n}:\ T(c)\in B_{\mathrm{op}}\,\}.
\end{equation}
Since $T$ is linear and $B_{\mathrm{op}}$ is an open polyhedron, $\mathcal C_{\mathrm{op}}$ is also an open polyhedron.
Moreover, it is full-dimensional because it is open and nonempty (e.g., $((0,1),0,0)\in\mathcal C_{\mathrm{op}}$).

We claim that
\begin{equation}
\label{eq:global0-equiv-containment}
\Dset=\varnothing\ \text{is a global SDD for }(X,\mathcal C_{\mathrm{op}})
\quad\Longleftrightarrow\quad
P\subseteq Q.
\end{equation}
Since $P\subseteq Q$ is $\mathsf{coNP}$-hard, this proves the theorem.

\paragraph{Step 1: Containment is equivalent to a cone inclusion.}
Recall that $\cone\{\bar v_1,\dots,\bar v_M\}=\{(t z,t): t\ge0,\ z\in Q\}$.
In particular, for any $\tilde t>0$ we have
\[
(\tilde y,\tilde t)\in \cone\{\bar v_1,\dots,\bar v_M\}
\quad\Longleftrightarrow\quad
\tilde y/\tilde t \in Q.
\]
Therefore,
\[
B_{\mathrm{op}}\subseteq \cone\{\bar v_1,\dots,\bar v_M\}
\quad\Longleftrightarrow\quad
(-1,1)^d\subseteq Q.
\]
Because $Q$ is closed, $(-1,1)^d\subseteq Q$ holds if and only if $[-1,1]^d\subseteq Q$, i.e., $P\subseteq Q$.

\paragraph{Step 2: If $P\subseteq Q$, then $\Dset=\varnothing$ is a global SDD.}
Assume $P\subseteq Q$.
By Step~1 we have $B_{\mathrm{op}}\subseteq \cone\{\bar v_1,\dots,\bar v_M\}$, and since $B_{\mathrm{op}}$ is open this implies
\[
B_{\mathrm{op}}\subseteq \operatorname{int}\bigl(\cone\{\bar v_1,\dots,\bar v_M\}\bigr).
\]

Next, note that $w$ uniquely determines $(r,s)$ via the equalities $w+r=2\cdot \1$ and $\bar Vw-s=\beta$.
Thus minimizing $c^\top z$ over $X$ is equivalent to minimizing $T(c)^\top w$ over the projected feasible set
\[
W:=\{w\in\RR^{n_0}:\ \exists (r,s)\ge0\ \text{s.t. } (w,r,s)\in X\}.
\]
Equivalently, with the change of variables $u:=w-\1=(x,t)$, this projected set corresponds to the bounded polytope
\[
\widetilde X
=\{(x,t)\in\RR^{d}\times\RR:\ v_i^\top x+t\ge 0\ \forall i\in[M],\ \ -1\le (x,t)\le 1\}.
\]

At $u_0:=(0,0)\in\widetilde X$, the box constraints are slack and the active constraints are $v_i^\top x+t\ge0$, i.e.,
$-\bar v_i^\top u\le 0$.
Hence $\Lambda(u_0)=\cone\{\bar v_1,\dots,\bar v_M\}$.
If $(\tilde y,\tilde t)\in\operatorname{int}(\Lambda(u_0))$, then $u_0$ is the \emph{unique} minimizer for this cost:
for any $u\in\widetilde X\setminus\{u_0\}$ the direction $u-u_0$ is a nonzero feasible direction at $u_0$,
so $(\tilde y,\tilde t)^\top(u-u_0)>0$ and hence $(\tilde y,\tilde t)^\top u>(\tilde y,\tilde t)^\top u_0$.

Now fix any $c\in\mathcal C_{\mathrm{op}}$.
Then $T(c)\in B_{\mathrm{op}}\subseteq\operatorname{int}(\Lambda(u_0))$, so the unique minimizer of $\min\{T(c)^\top u:\ u\in\widetilde X\}$
is $u_0$.
By \eqref{eq:objective-reduction}, the minimizer set of $\min_{z\in X} c^\top z$ is therefore the singleton $\{z_0\}$.
Thus $X^\star(c)=\{z_0\}$ for all $c\in\mathcal C_{\mathrm{op}}$, and the constant rule $\widehat X(\emptyset):=\{z_0\}$ makes
$\Dset=\varnothing$ a global SDD.

\paragraph{Step 3: If $P\not\subseteq Q$, then no zero-query global SDD exists.}
Now assume $P\not\subseteq Q$.
By Step~1 there exists $(\tilde y,\tilde t)\in B_{\mathrm{op}}$ such that $(\tilde y,\tilde t)\notin\cone\{\bar v_1,\dots,\bar v_M\}$.
Let $c_1:=((\tilde y,\tilde t),0,0)\in\mathcal C_{\mathrm{op}}$, so $T(c_1)=(\tilde y,\tilde t)$.
Since $(\tilde y,\tilde t)\notin \Lambda(u_0)$, we have $u_0\notin \widetilde X^\star((\tilde y,\tilde t))$, and thus $z_0\notin X^\star(c_1)$.

On the other hand, $c_0:=((0,1),0,0)\in\mathcal C_{\mathrm{op}}$ and by the argument in the proof of \Cref{thm:coNP-hard-pointwise-check} (Step 2 there), we have $X^\star(c_0)=\{z_0\}$.
Therefore $X^\star(c_1)\neq X^\star(c_0)$, i.e., the optimal solution set is not constant over $c\in\mathcal C_{\mathrm{op}}$.

Since $|\Dset|=0$, any candidate decoder $\widehat X:\RR^{|\Dset|}\to\mathcal{P}(X)$ is necessarily constant.
It cannot match two distinct optimal solution sets, so no such map can satisfy Definition~\ref{def:bennouna-sdd}.
This establishes \eqref{eq:global0-equiv-containment}.

\paragraph{Consequence for minimum-size global SDDs.}
Finally, if one could compute the size of a minimum global SDD (or output such a dataset) in polynomial time, then one could
decide whether the optimum size equals $0$, which is exactly the $\mathsf{coNP}$-hard decision problem in
\eqref{eq:global0-equiv-containment}.
\end{proof}

\section{Proofs and technical details for Section~\ref{sec:cuttingplanes}}
\label{app:sec5}

\subsection{Closed-form FI for ellipsoids}
\label{app:FI-ellipsoid}
Recall the face-intersection subproblem
\[
\mathrm{FI}(\delta;Q,s)\;:=\;\min\{\delta^\top c:\ c\in C,\ Q^\top c=s\}.
\]
\begin{proposition}\label{prop:FI-ellipsoid}
Let $C=\{c\in\RR^d:\ (c-\bar c)^\top \Sigma^{-1}(c-\bar c)\le R^2\}$ with $\Sigma\succ0$.
Fix $Q\in\RR^{d\times k}$ with $\rank(Q)=k$ and $s\in\RR^k$, and assume
$C(Q,s):=\{c\in C:\ Q^\top c=s\}\neq\emptyset$.
Define
\[
c^\perp:=\bar c+\Sigma Q(Q^\top\Sigma Q)^{-1}(s-Q^\top\bar c),\qquad
M^\perp:=\Sigma-\Sigma Q(Q^\top\Sigma Q)^{-1}Q^\top\Sigma\succeq0,
\]
and
\[
\rho:=\sqrt{R^2-(c^\perp-\bar c)^\top\Sigma^{-1}(c^\perp-\bar c)}.
\]
Then for any $\delta\in\RR^d$,
\[
\min_{c\in C(Q,s)}\delta^\top c
=
\delta^\top c^\perp-\rho\sqrt{\delta^\top M^\perp\delta}.
\]
If $\delta^\top M^\perp\delta>0$, a minimizer is
\[
c^{\mathrm{out}}(\delta)=c^\perp-\rho\,\frac{M^\perp\delta}{\sqrt{\delta^\top M^\perp\delta}}.
\]
If $\delta^\top M^\perp\delta=0$, then $\delta^\top c$ is constant over $C(Q,s)$.
\end{proposition}

\begin{proof}
Let $\Sigma^{1/2}$ be the symmetric square root of $\Sigma$ (so $\Sigma=\Sigma^{1/2}\Sigma^{1/2}$).
Change variables
\[
z := \Sigma^{-1/2}(c-\bar c)\quad\Longleftrightarrow\quad c=\bar c+\Sigma^{1/2}z.
\]
Then the ellipsoid constraint becomes $\|z\|_2\le R$.
The equality constraint becomes
\[
Q^\top c=s
\quad\Longleftrightarrow\quad
Q^\top(\bar c+\Sigma^{1/2}z)=s
\quad\Longleftrightarrow\quad
(\Sigma^{1/2}Q)^\top z = s-Q^\top \bar c.
\]
Define $\widetilde Q := \Sigma^{1/2}Q$ and $\widetilde s := s-Q^\top \bar c$, and also
$\widetilde\delta := \Sigma^{1/2}\delta$.
Up to the constant $\delta^\top \bar c$, the FI problem is equivalent to
\[
\min\{\widetilde\delta^\top z:\ \|z\|_2\le R,\ \widetilde Q^\top z=\widetilde s\}.
\]

Let
\[
z_0 := \widetilde Q(\widetilde Q^\top \widetilde Q)^{-1}\widetilde s,
\]
the unique minimum-$\ell_2$-norm solution of $\widetilde Q^\top z=\widetilde s$.
Every feasible $z$ can be written uniquely as $z=z_0+v$ with $v\in\ker(\widetilde Q^\top)$.
Since $z_0\in\mathrm{span}(\widetilde Q)$ and $\ker(\widetilde Q^\top)=\mathrm{span}(\widetilde Q)^\perp$,
we have the orthogonal decomposition $\|z\|_2^2=\|z_0\|_2^2+\|v\|_2^2$.
Thus feasibility is equivalent to $\|z_0\|_2\le R$ (which holds since the fiber is nonempty) and
\[
\|v\|_2 \le \rho_z := \sqrt{R^2-\|z_0\|_2^2}.
\]

Let $P:=I-\widetilde Q(\widetilde Q^\top\widetilde Q)^{-1}\widetilde Q^\top$ denote the orthogonal projector onto
$\ker(\widetilde Q^\top)$.
Then for $v\in\ker(\widetilde Q^\top)$, $\widetilde\delta^\top v=(P\widetilde\delta)^\top v$.
Therefore,
\[
\min_{\substack{v\in\ker(\widetilde Q^\top)\\ \|v\|_2\le \rho_z}}\ \widetilde\delta^\top v
=
\min_{\|v\|_2\le \rho_z}\ (P\widetilde\delta)^\top v
=
-\rho_z\,\|P\widetilde\delta\|_2,
\]
with minimizer $v^\star=-\rho_z\,\frac{P\widetilde\delta}{\|P\widetilde\delta\|_2}$ when $P\widetilde\delta\neq 0$
(and any $v$ when $P\widetilde\delta=0$).

Hence, the optimal value in $z$-space is
\[
\widetilde\delta^\top z_0 - \rho_z\,\|P\widetilde\delta\|_2,
\]
and the optimal $c$ is $c=\bar c+\Sigma^{1/2}(z_0+v^\star)$.

It remains to express everything back in the original variables.
First, note that $\widetilde Q^\top \widetilde Q = Q^\top \Sigma Q$ and
\[
\Sigma^{1/2}z_0
=
\Sigma^{1/2}\widetilde Q(\widetilde Q^\top\widetilde Q)^{-1}\widetilde s
=
\Sigma Q(Q^\top \Sigma Q)^{-1}(s-Q^\top\bar c).
\]
Thus
\[
c^\perp
=
\bar c+\Sigma^{1/2}z_0
=
\bar c+\Sigma Q(Q^\top \Sigma Q)^{-1}(s-Q^\top\bar c).
\]

Second,
\[
\|z_0\|_2^2
=
z_0^\top z_0
=
(c^\perp-\bar c)^\top \Sigma^{-1}(c^\perp-\bar c),
\]
so $\rho_z=\rho$ as defined in the statement.

Finally,
\[
\|P\widetilde\delta\|_2^2
=
\widetilde\delta^\top P\widetilde\delta
=
\delta^\top\Bigl(\Sigma-\Sigma Q(Q^\top \Sigma Q)^{-1}Q^\top \Sigma\Bigr)\delta
=
\delta^\top M^\perp \delta.
\]
Also,
\[
\Sigma^{1/2}P\widetilde\delta
=
\Bigl(\Sigma-\Sigma Q(Q^\top \Sigma Q)^{-1}Q^\top \Sigma\Bigr)\delta
=
M^\perp \delta.
\]
Substituting these identities yields the claimed closed-form expressions for the minimum value and the optimizer.
\end{proof}

\subsection{Proof of Property~\ref{prop:polytime}}
\label{app:polytime}
\begin{proof}
We bound the computational work in one iteration of Algorithm~\ref{alg:pointwise-cp}.
Each iteration performs the following optimization subroutines.
\begin{enumerate}[label=(\roman*),leftmargin=2.2em,itemsep=2pt,topsep=2pt]
\item \textbf{One LP over $\X$.} Line~5 solves
\[
\min\{(c^{\mathrm{in}})^\top x:\ Ax=b,\ x\ge 0\}.
\]
Since $\X$ is given in standard form, this LP can be solved in time polynomial in the
bit complexity of the input.

\item \textbf{$d-m$ face-intersection subproblems over $\C$.} For each $j\in N$ (so $|N|=d-m$), line~6 of ~\cref{alg:pointwise-cp} solves
\[
\min\{\delta_j^\top c':\ c'\in\C,\ Q_k^\top c'=s_k\}.
\]
If $\C$ is a polytope given in $H$-representation, $\C=\{c:\ Gc\le h\}$, then this is the LP
\[
\min\{\delta_j^\top c':\ Gc'\le h,\ Q_k^\top c'=s_k\},
\]
whose size is polynomial in the input (in particular, it has dimension $d$ and $k\le d^\star$ equality constraints).
Hence, each FI call can be solved in polynomial time.
If instead $\C$ is an ellipsoid, Proposition~\ref{prop:FI-ellipsoid} gives a closed-form expression for the optimal value and a minimizer,
which can be computed in polynomial time (it amounts to solving a $k\times k$ linear system and basic matrix--vector operations).
\end{enumerate}

Thus, each iteration runs in polynomial time and makes at most $d-m$ FI calls.
Finally, by \Cref{thm:alg1-correct}, Algorithm~\ref{alg:pointwise-cp} executes at most $d^\star+1\le d+1$ iterations and makes at most
$d^\star$ oracle queries of the form $q^\top c$.
Combining the per-iteration bound with the iteration bound yields the claimed overall polynomial running time.
\end{proof}

\section{Proofs and technical details for Section~\ref{sec:ddr-compression}}
\label{app:sec6}
\begin{proof}[Proof of Theorem~\ref{thm:lowerbound}]
\label{app:lowerbound}

We give an explicit construction in which the intrinsic dimension $d^\star$ and the ambient dimension can be chosen independently.

\paragraph{Feasible region.}
Fix integers $d\ge d^\star\ge 2$. Let $m:=d$, set $A:=[I_d\ I_d]\in\RR^{d\times 2d}$ and $b:=\mathbf{1}\in\RR^d$, and define
the lifted polytope
\[
\X:=\{z=(x,s)\in\RR^{2d}:Az=b,\ z\ge0\}
=\{(x,s):x+s=\mathbf{1},\ x\ge0,\ s\ge0\},
\]
which is an extended formulation of the hypercube $[0,1]^d$ obtained by introducing slack variables.
Then $\X$ is bounded and nondegenerate:
at every extreme point, for each $j\in[d]$ exactly one of $x_j,s_j$ equals $1$ and the other equals $0$,
so every extreme point has exactly $m=d$ strictly positive components.

\paragraph{Prior set and a rare-types distribution.}
Let $\mu\in\RR^d$ be defined coordinatewise by
\[
    \mu_j \;=\;
    \begin{cases}
        0.99, & j\in\{1,\dots,d^\star\},\\
        10, & j\in\{d^\star+1,\dots,d\},
    \end{cases}
\]
and define the lifted center $\bar\mu:=(\mu,0)\in\RR^{2d}$.
Let the convex prior set be the lifted radius-$1$ Euclidean ball
\[
    \C \;:=\; \{(c,0)\in\RR^{2d}:\ \|c-\mu\|_2 \le 1\}.
\]
For each $i\in\{1,\dots,d^\star\}$ define the lifted costs and query directions
\[
    c^{(i)} \;:=\; (\mu-e_i,0)\in\C,
    \qquad
    \delta_i \;:=\; (-e_i,e_i)\in\RR^{2d}.
\]
Fix $\varepsilon\in(0,1/4)$ and let $k:=d^\star-1$. Define a distribution $P_c$ supported on
$\{c^{(1)},\dots,c^{(d^\star)}\}\subseteq\C$ by
\[
    \PP(c=c^{(1)}) \;=\; 1-2\varepsilon,
    \qquad
    \PP(c=c^{(i)}) \;=\; \frac{2\varepsilon}{k},\quad i=2,\dots,d^\star.
\]
We call $c^{(1)}$ the \emph{common} type and $\{c^{(2)},\dots,c^{(d^\star)}\}$ the \emph{rare} types.

\paragraph{Step 1: Prove that $\dim\dir(\X^\star(\C))=d^\star$.}

\begin{lemma}
\label{lem:lb-cube-dim}
For every $(c,0)\in\C$, every minimizer $z^\star=(x^\star,s^\star)\in\argmin_{(x,s)\in\X} (c,0)^\top(x,s)$ satisfies
$x^\star_j=0$ and $s^\star_j=1$ for all $j>d^\star$.
Moreover, $\X^\star(\C)\supseteq \{(0,\mathbf{1}),\ (e_1,\mathbf{1}-e_1),\dots,(e_{d^\star},\mathbf{1}-e_{d^\star})\}$.
Consequently,
\[
    \dim\dir(\X^\star(\C)) \;=\; d^\star.
\]
\end{lemma}

\begin{proof}
Fix $(c,0)\in\C$ and any feasible $(x,s)\in\X$. Since $x+s=\mathbf{1}$, we have
\[
(c,0)^\top(x,s)=c^\top x.
\]
Thus, the LP objective depends only on $x\in[0,1]^d$.
For each $j>d^\star$ we have $c_j\ge \mu_j-1=9>0$, hence any minimizer must set $x^\star_j=0$
(and therefore $s^\star_j=1$) for all $j>d^\star$.

Next, $\bar\mu=(\mu,0)\in\C$ and $\mu$ has strictly positive coordinates, so the unique minimizer is
$(x,s)=(0,\mathbf{1})$.
For each $i\le d^\star$, the cost $c^{(i)}=(\mu-e_i,0)$ has $c^{(i)}_i=-0.01<0$ and all other $c^{(i)}_j>0$,
so the unique minimizer is $(x,s)=(e_i,\mathbf{1}-e_i)$.
This proves the stated inclusion of $\X^\star(\C)$.

Therefore $\dir(\X^\star(\C))$ contains $\spanop\{(e_i,-e_i): i=1,\dots,d^\star\}$, so
$\dim\dir(\X^\star(\C))\ge d^\star$.
On the other hand, we already showed that every optimizer has $x_j=0$ for $j>d^\star$, hence every difference of reachable optima lies in
$\spanop\{(e_i,-e_i): i=1,\dots,d^\star\}$, giving $\dim\dir(\X^\star(\C))\le d^\star$.
\end{proof}

\paragraph{Step 2: Without querying $\delta_i$, type $i$ cannot be certified.}
\begin{lemma}
\label{lem:lb-need-delta}
Fix any $i\in\{2,\dots,d^\star\}$ and let $\Dset\subseteq\RR^{2d}$ be any dataset that does not contain $\delta_i$.
If $\Dset\subseteq\{\delta_1,\dots,\delta_{d^\star}\}$, then
\[
    \bar\mu\in \C\bigl(\Dset, s(c^{(i)};\Dset)\bigr).
\]
Consequently, $\Dset$ is \emph{not} pointwise sufficient at $c^{(i)}$.
\end{lemma}

\begin{proof}
Assume $\Dset\subseteq\{\delta_1,\dots,\delta_{d^\star}\}$ and $\delta_i\notin\Dset$.
Every query in $\Dset$ is of the form $\delta_j=(-e_j,e_j)$ with $j\neq i$.
For such $j$,
\[
\delta_j^\top \bar\mu
=(-e_j,e_j)^\top(\mu,0)=-\mu_j
=-(\mu-e_i)_j
=(-e_j,e_j)^\top(\mu-e_i,0)
=\delta_j^\top c^{(i)}.
\]
Hence $\bar\mu$ is consistent with the same measurements as $c^{(i)}$, i.e.,
$\bar\mu\in \C(\Dset,s(c^{(i)};\Dset))$.

But by Lemma~\ref{lem:lb-cube-dim}, $x^\star(\bar\mu)=0$ (so the unique optimizer is $(0,\mathbf{1})$)
while $x^\star(c^{(i)})=e_i$ (unique optimizer $(e_i,\mathbf{1}-e_i)$).
Thus, the fiber contains two costs with different unique minimizers, so no single decision can be optimal for all costs in the fiber.
Therefore $\Dset$ is not pointwise sufficient at $c^{(i)}$.
\end{proof}

\paragraph{Step 2': On type $i$, the pointwise routine adds only $\delta_i$.}
\begin{lemma}\label{lem:typei-only-deltai}
Fix $i\in\{1,\dots,d^\star\}$. Run Algorithm~\ref{alg:pointwise-cp} on $c=c^{(i)}=(\mu-e_i,0)$ with
initialization $D_{\mathrm{init}}\subseteq\{\delta_1,\dots,\delta_{d^\star}\}$ such that
$\delta_i\notin D_{\mathrm{init}}$.
Then the call performs exactly one augmentation and returns
$D=D_{\mathrm{init}}\cup\{\delta_i\}$.
In particular, during this call, the algorithm cannot add any $\delta_j$ with $j\neq i$.
\end{lemma}

\begin{proof}
Let $Q_k$ be the matrix whose columns are the directions in $D_{\mathrm{init}}$ and set $s_k:=Q_k^\top c^{(i)}$.
Since each $\delta_j=(-e_j,e_j)$, the fiber
\[
    \C_k := \{c'\in\C:\ Q_k^\top c'=s_k\}
\]
fixes the coordinates $c'_j=(\mu-e_i)_j=\mu_j$ for every $\delta_j\in D_{\mathrm{init}}$ and leaves coordinate $i$ unconstrained;
all $c'\in\C_k$ have the form $(\tilde c,0)$ with $\tilde c\in\RR^d$.

\paragraph{(a) The first LP solve yields the vertex $(e_i,\mathbf{1}-e_i)$ and its cone.}
With $c_{\mathrm{in}}=c^{(i)}=(\mu-e_i,0)$, we have $(\mu-e_i)_i<0$ and $(\mu-e_i)_j>0$ for all $j\neq i$,
so the LP over $\X$ has the unique minimizer
\[
z^\star=(x^\star,s^\star)=(e_i,\mathbf{1}-e_i).
\]
For the matrix $A=[I_d\ I_d]$, this vertex corresponds to the unique feasible basis
\[
B(i):=\{x_i\}\cup\{s_j:\ j\neq i\},
\qquad
N(i):=\{s_i\}\cup\{x_j:\ j\neq i\}.
\]
For $j\neq i$, increasing the nonbasic variable $x_j$ from $0$ decreases $s_j$ from $1$ to $0$, hence
\[
\delta(B(i),x_j)=(e_j,-e_j).
\]
Increasing the nonbasic variable $s_i$ from $0$ decreases $x_i$ from $1$ to $0$, hence
\[
\delta(B(i),s_i)=(-e_i,e_i)=\delta_i.
\]
Therefore, the optimality cone \eqref{eq:cone-delta} takes the explicit form
\[
\Lambda(B(i))
=\Bigl\{(u,v)\in\RR^{2d}:\ (u_j-v_j)\ge 0\ \forall j\neq i,\ \ (-u_i+v_i)\ge 0\Bigr\}.
\]
Since every $c'\in\C_k$ has $v=0$, we have
\[
\C_k\subseteq \{(u,0):u\in\RR^d\},
\qquad
\Lambda(B(i))\cap \{(u,0)\}
=\{(u,0):\ u_j\ge 0\ \forall j\neq i,\ u_i\le 0\}.
\]

\paragraph{(b) Among facets of $\Lambda(B(i))$, only the facet for $\delta_i$ can be hit from within $\C_k$.}
Because coordinate $i$ is free in $\C_k$, the point $\tilde c:=(\mu+e_i,0)$ belongs to $\C_k$
(and to $\C$) but has $\tilde c_i>0$, hence $\tilde c\notin\Lambda(B(i))\cap\{(u,0)\}$.
Therefore $\C_k\not\subseteq \Lambda(B(i))$, the containment test fails, and Algorithm~\ref{alg:pointwise-cp} enters the \textsc{Else} branch,
producing some witness $c_{\mathrm{out}}\in\C_k\setminus\Lambda(B(i))$ and considering the segment
$c_\alpha:=(1-\alpha)c_{\mathrm{in}}+\alpha c_{\mathrm{out}}$.

We claim that for every $j\neq i$,
\[
\bigl(\C\cap \Lambda(B(i))\bigr)\ \cap\ \{(u,v):\ (u_j-v_j)=0\} \;=\; \emptyset.
\]
Indeed, take any $c=(u,v)\in\C\cap\Lambda(B(i))$. Since $c\in\C$ we have $v=0$ and $\|u-\mu\|_2\le 1$.
Moreover, $c\in\Lambda(B(i))$ implies $u_i\le 0$. Since $\mu_i=0.99>0$,
\[
(u_i-\mu_i)^2\ge (0-\mu_i)^2=\mu_i^2.
\]
Thus,
\[
\sum_{t\neq i}(u_t-\mu_t)^2 \le 1-(u_i-\mu_i)^2 \le 1-\mu_i^2,
\]
and hence for every $j\neq i$,
\[
|u_j-\mu_j|\le \sqrt{1-\mu_i^2}
\quad\Rightarrow\quad
u_j\ge \mu_j-\sqrt{1-\mu_i^2}.
\]
With $\mu_j\in\{0.99,10\}$ and $\sqrt{1-\mu_i^2}=\sqrt{1-0.99^2}<0.15$, we get $u_j>0$ for all $j\neq i$.
Therefore $u_j-v_j=u_j>0$ for all $j\neq i$, proving the claim.

Now let $\alpha^\star$ be the first parameter where the segment leaves $\relint(\Lambda(B(i)))$.
Then $c_{\alpha^\star}\in \C_k\cap \Lambda(B(i))$ lies on a facet hyperplane of $\Lambda(B(i))$.
By the claim above, it cannot lie on any facet $(u_j-v_j)=0$ with $j\neq i$; hence, it must lie on the facet
$(-u_i+v_i)=0$, whose normal is exactly $\delta(B(i),s_i)=\delta_i$.
Therefore, the facet-hit rule appends $q_{k+1}=\delta_i$.

\paragraph{(c) After adding $\delta_i$, the fiber becomes a singleton and the routine terminates.}
Appending $\delta_i=(-e_i,e_i)$ adds the constraint
$\delta_i^\top(c',0)=\delta_i^\top(\mu-e_i,0)$, i.e., $-u_i=(1-\mu_i)$ and hence $u_i=\mu_i-1$.
Then $(u_i-\mu_i)^2=1$, and the radius-$1$ constraint $\|u-\mu\|_2^2\le 1$ forces
$\sum_{t\neq i}(u_t-\mu_t)^2=0$, hence $u=\mu-e_i$ and $c'=c^{(i)}$.
Therefore the updated fiber is the singleton $\{c^{(i)}\}$, the containment test succeeds, and
Algorithm~\ref{alg:pointwise-cp} terminates after exactly one augmentation, returning $D=D_{\mathrm{init}}\cup\{\delta_i\}$.
\end{proof}

\paragraph{Step 3: A coupon-collector lower bound for Algorithm~\ref{alg:cumulative}.}
Let $I\subseteq\{2,\dots,d^\star\}$ be the set of rare indices that appear at least once among the $n$ i.i.d.\ samples.
By Lemma~\ref{lem:typei-only-deltai}, Algorithm~\ref{alg:cumulative} learns $\delta_i$ if and only if $i\in I$.
By Lemma~\ref{lem:lb-need-delta}, the final dataset $\Dset_n$ fails on every rare type $i\notin I$.
Therefore
\[
    R(\Dset_n)
    \;\ge\;
    \sum_{i\in\{2,\dots,d^\star\}\setminus I} \PP(c=c^{(i)})
    \;=\;
    \frac{2\varepsilon}{k}\,\bigl|\{2,\dots,d^\star\}\setminus I\bigr|.
\]

Let $N$ be the number of rare samples among $c_1,\dots,c_n$:
\[
    N \;=\; \sum_{t=1}^n \1\{c_t\neq c^{(1)}\}\ \sim\ \mathrm{Binomial}(n,2\varepsilon).
\]
Since $|I|\le N$, the event $N<k/2$ implies $|\{2,\dots,d^\star\}\setminus I|\ge k/2$, and hence $R(\Dset_n)\ge \varepsilon$.
It remains to lower bound $\PP(N<k/2)$.
If $n\le k/(8\varepsilon)$, then $\EE[N]=2\varepsilon n\le k/4$, and Markov's inequality gives
\[
    \PP\!\Bigl(N\ge \frac{k}{2}\Bigr)
    \;\le\;
    \frac{\EE[N]}{k/2}
    \;\le\;
    \frac{k/4}{k/2}
    \;=\;
    \frac{1}{2}.
\]
Hence $\PP(N<k/2)\ge 1/2$, and on this event we have $R(\Dset_n)\ge \varepsilon$.
This concludes the proof of \Cref{thm:lowerbound}.
\end{proof}

\section{Proofs and technical details for Section~\ref{sec:clo_compression}}
\label{app:clo_details}

\subsection{Ellipsoidal lifting}
\label{app:lift}

Throughout this appendix, we assume the shifted ellipsoidal prior
\[
\C:=\{c\in\RR^d:\ (c-c_0)^\top \Sigma^{-1}(c-c_0)\le 1\},
\qquad \Sigma\succ 0,\ c_0\in\RR^d.
\]
For any orthonormal basis $U\in\RR^{d\times t}$ we define the lifting matrix
\[
\mathcal{L}_U:=\Sigma U(U^\top \Sigma U)^{-1},
\]
and the associated canonical lifting map $\operatorname{lift}_U:\RR^t\to\RR^d$ by
\[
\operatorname{lift}_U(s):=c_0+\mathcal{L}_U s.
\]

\begin{lemma}
\label{lem:lifting_main}
For any $s\in\RR^t$, $\operatorname{lift}_U(s)$ is the unique solution to
\[
\min_{c\in\RR^d}\ \frac12(c-c_0)^\top \Sigma^{-1}(c-c_0)
\quad \text{s.t.}\quad U^\top(c-c_0)=s.
\]
In particular, it satisfies $U^\top(\operatorname{lift}_U(s)-c_0)=s$ and
\[
(\operatorname{lift}_U(s)-c_0)^\top \Sigma^{-1}(\operatorname{lift}_U(s)-c_0)
=s^\top(U^\top\Sigma U)^{-1}s.
\]
Consequently, $\operatorname{lift}_U(s)\in\C$ whenever $s\in U^\top(\C-c_0)$.
When $\Sigma=I$, we have $\operatorname{lift}_U(s)=c_0+Us$.
\end{lemma}

\begin{proof}
Let $z:=c-c_0$. The problem becomes
\[
\min_{z\in\RR^d}\ \frac12 z^\top \Sigma^{-1}z
\quad\text{s.t.}\quad U^\top z=s,
\]
which is exactly the centered-ellipsoid case after the change of variables.
The Lagrangian is $\mathcal{L}(z,\lambda)=\frac12 z^\top \Sigma^{-1}z-\lambda^\top(U^\top z-s)$.
Stationarity gives $\Sigma^{-1}z-U\lambda=0$, hence $z=\Sigma U\lambda$.
Imposing the constraint yields $s=U^\top z=U^\top\Sigma U\lambda$, so $\lambda=(U^\top\Sigma U)^{-1}s$ and
$z=\Sigma U(U^\top\Sigma U)^{-1}s=\mathcal{L}_U s$.
The claimed identities follow by direct substitution.
\end{proof}

\subsection{SPO generalization in the decision-sufficient subspace (Proof of Theorem~\ref{thm:misspec_plus_gen})}
\label{app:spo_gen_dec_suff_subspace}
We now prove the Stage~II generalization bound in Theorem~\ref{thm:misspec_plus_gen}.
The proof has three ingredients:
(i) we bound the complexity of the induced decision class $x^\star\circ\mathcal H_{U_\star,\dstar}$ via its Natarajan dimension,
(ii) we plug this bound into the Natarajan-dimension generalization theorem for the SPO loss due to \citet{elbalghiti2023generalization}, and
(iii) we show that under a global sufficient dataset, projecting to $W_\star$ and lifting back is decision-preserving, so restricting to the compressed class incurs no approximation error.

\paragraph{Step 1: Natarajan dimension of the compressed decision class.}
We begin by translating compressed affine predictors into a multiclass linear prediction problem over labels $\X^\angle$.
This is the only place where the intrinsic dimension $\dstar$ enters the analysis.

\begin{lemma}[Natarajan dimension bound]
\label{lem:natarajan_tp}
Let $\X\subseteq\RR^d$ be a bounded polytope with extreme points $\X^\angle$.
Let $\mathcal{H}^{\mathrm{aff}}_{U_\star,\dstar}$ be the class of compressed affine predictors
\[
\mathcal{H}^{\mathrm{aff}}_{U_\star,\dstar}
:=\Bigl\{
f_{B,b}(\xi):=b c_0+\mathcal{L}_{U_\star}B\xi\ :\ B\in\RR^{\dstar\times p}, \ b \in \RR
\Bigr\},
\]
and let $\mathcal{F}^{\mathrm{aff}}_{U_\star,\dstar}:=x^\star\circ \mathcal{H}^{\mathrm{aff}}_{U_\star,\dstar}$ be the induced decision class.
Then the Natarajan dimension satisfies $d_N(\mathcal{F}^{\mathrm{aff}}_{U_\star,\dstar})\le \dstar p+1$.
\end{lemma}
\begin{proof}
For any $x\in\X^\angle$,
\[
f_{B,b}(\xi)^\top x
=b c_0^\top x + \bigl(\mathcal{L}_{U_\star} B \xi\bigr)^\top x
= bc_0^\top x + \left\langle \vec(B),\ \vec\! \left((\mathcal{L}_{U_\star}^\top x)\xi^\top \right)\right\rangle.
\]
Define the feature map
\[
\Psi^{\mathrm{aff}}(\xi,x)
:=\begin{bmatrix}
\vec\!\left((\mathcal{L}_{U_\star}^\top x) \xi^\top\right)\\[1mm]
c_0^\top x
\end{bmatrix}
\in\RR^{\dstar p+1},
\qquad
w_{B,b}
:=\begin{bmatrix}
\vec(B)\\
b
\end{bmatrix}\in\RR^{\dstar p+1}.
\]
Then $\mathcal{F}^{\mathrm{aff}}_{U_\star,\dstar}$ is a subset of the multiclass linear hypothesis class
$$\mathcal{H}_{\Psi^{\mathrm{aff}}}:=\{\xi\mapsto \argmin_{x\in\X^\angle}\langle w,\Psi^{\mathrm{aff}}(\xi,x)\rangle:\ w\in\RR^{\dstar p+1}\},$$
because each $f_{B,b}$ induces the decision rule
$x^\star(f_{B,b}(\xi))=\argmin_{x\in\X^\angle}\langle w_{B,b},\Psi^{\mathrm{aff}}(\xi,x)\rangle$.
Finally, by \citet[Theorem~29.7]{shalev-shwartz2014understanding}, the Natarajan dimension of $\mathcal{H}_{\Psi^{\mathrm{aff}}}$ is at most $\dstar p+1$,
and the same bound holds for its subset $\mathcal{F}^{\mathrm{aff}}_{U_\star,\dstar}$.
\end{proof}

\paragraph{Step 2: Uniform SPO generalization in the compressed class.}
We now combine Lemma~\ref{lem:natarajan_tp} with the SPO generalization bound of \citet{elbalghiti2023generalization}.
This yields a uniform convergence guarantee over $\mathcal H_{U_\star,\dstar}$ with leading complexity term scaling as $\dstar p$.

\begin{lemma}
\label{lem:spo_gen_compressed}
For any $\delta\in(0,1)$, with probability at least $1-\delta$ over an i.i.d.\ sample
$S=\{(\xi_i,c_i)\}_{i=1}^n$, we have
\[
R_{\mathrm{SPO}}(f)
\le
\hat R_{\mathrm{SPO}}(f)
+2\,\omega_{\X}(\C)\,\sqrt{\frac{2(\dstar p+1)\log(n|\X^\angle|^2)}{n}}
+\omega_{\X}(\C)\,\sqrt{\frac{\log(1/\delta)}{2n}},
\]
simultaneously for all $f\in\mathcal{H}_{U_\star,\dstar}$, where
$\omega_{\X}(\C):=\sup_{c\in\C}\left(\max_{x\in\X} c^\top x-\min_{x\in\X} c^\top x\right)$.
\end{lemma}

\begin{proof}
This follows from the Natarajan-dimension generalization bound for SPO loss in \citet{elbalghiti2023generalization}.
The SPO loss is uniformly bounded by $\omega_{\X}(\C)$ when $c\in\C$.
Using Lemma~\ref{lem:natarajan_tp} and the inclusion
$\mathcal{H}_{U_\star,\dstar}\subseteq \mathcal{H}^{\mathrm{aff}}_{U_\star,\dstar}$ (take $b=1$),
we have $d_N(x^\star\circ \mathcal{H}_{U_\star,\dstar})\le d_N(\mathcal{F}^{\mathrm{aff}}_{U_\star,\dstar})\le \dstar p+1$,
which yields the stated bound.
\end{proof}

\paragraph{Step 3: Global sufficiency implies lossless compression.}
To prove the first statement in Theorem~\ref{thm:misspec_plus_gen}, we show that if $\Dset$ is globally sufficient on $\C$,
then compressing any cost vector to $W_\star$ and lifting it back to $\C$ leaves the oracle decision unchanged.

\begin{lemma}
\label{lem:global_sufficiency_lossless}
Assume $\Dset$ is globally sufficient on $\C$, $W_\star:=\spanop(\Dset)$ and let $U_\star\in\RR^{d\times \dstar}$ be an orthonormal basis of $W_\star$.
Recall that $\mathcal{L}_{U_\star}:=\Sigma U_\star(U_\star^\top\Sigma U_\star)^{-1}$ and $\operatorname{lift}_{U_\star}(s):=c_0+\mathcal{L}_{U_\star} s$.

Fix a deterministic tie-breaking rule so that $x^\star(\cdot)$ is single-valued.
Then for any $\hat c\in\C$, letting
\[
\tilde c := \operatorname{lift}_{U_\star}\!\bigl(U_\star^\top(\hat c-c_0)\bigr),
\]
we have $x^\star(\hat c)=x^\star(\tilde c)$.

Consequently, for any predictor $f:\Xi\to\C$, the compressed predictor
$\hat f(\xi):=\operatorname{lift}_{U_\star}\!\bigl(U_\star^\top(f(\xi)-c_0)\bigr)$
induces the same decisions and satisfies
$\ell_{\mathrm{SPO}}(f(\xi),c)=\ell_{\mathrm{SPO}}(\hat f(\xi),c)$ for all $c\in\C$.
\end{lemma}

\begin{proof}
Fix $\hat c\in\C$ and set $s:=U_\star^\top(\hat c-c_0)$ and $\tilde c:=\operatorname{lift}_{U_\star}(s)=c_0+\mathcal{L}_{U_\star} s$.
By the definition of $\mathcal{L}_{U_\star}$,
\[
U_\star^\top(\tilde c-c_0)
=U_\star^\top\mathcal{L}_{U_\star} s
=U_\star^\top\Sigma U_\star\,(U_\star^\top\Sigma U_\star)^{-1}s
=s
=U_\star^\top(\hat c-c_0).
\tag{$\star$}
\]
Now take any $q\in\Dset\subseteq W_\star=\spanop(U_\star)$. Write $q=U_\star a$ for some $a\in\RR^{\dstar}$. Then using $(\star)$,
\[
q^\top\tilde c-q^\top\hat c
= a^\top U_\star^\top(\tilde c-\hat c)
= a^\top\Bigl(U_\star^\top(\tilde c-c_0)-U_\star^\top(\hat c-c_0)\Bigr)=0.
\]
Hence $s(\tilde c;\Dset)=s(\hat c;\Dset)$.
Moreover, since $\hat c\in\C$ we have $s\in U_\star^\top(\C-c_0)$ and thus Lemma~\ref{lem:lifting_main} implies $\tilde c\in\C$.
Global sufficiency of $\Dset$ on $\C$ then yields
$\argmin_{x\in\X}\hat c^\top x=\argmin_{x\in\X}\tilde c^\top x$,
and the deterministic tie-breaking rule gives $x^\star(\hat c)=x^\star(\tilde c)$.

The predictor claim follows by applying the above argument pointwise to $\hat c=f(\xi)\in\C$.
\end{proof}

\smallskip
With Lemmas~\ref{lem:spo_gen_compressed} and \ref{lem:global_sufficiency_lossless} in hand, we can now complete the proof of Theorem~\ref{thm:misspec_plus_gen}.

\begin{proof}[Proof of Theorem~\ref{thm:misspec_plus_gen}]

Part~(1) follows from Lemma~\ref{lem:global_sufficiency_lossless} by taking $f=f_\star\in\mathcal{H}(\C)$:
the compressed predictor $\hat f_\star(\xi)=\operatorname{lift}_{U_\star}(U_\star^\top(f_\star(\xi)-c_0))$ induces the same decisions as $f_\star$,
hence has the same population SPO risk, and is a risk minimizer in $\mathcal{H}_{U_\star,\dstar}(\C)$.

For part~(2), apply Lemma~\ref{lem:spo_gen_compressed}, which gives a uniform generalization bound over the larger class $\mathcal{H}_{U_\star,\dstar}$
(and therefore also over its subset $\mathcal{H}_{U_\star,\dstar}(\C)$).
\end{proof}

\subsection{A concrete bound for ordinary least squares}
\label{app:bound_ols}
Suppose $\|\xi\|_2\le 1$ almost surely and the population design covariance satisfies
$\Sigma_\xi:=\EE[\xi\xi^\top]\succeq \kappa I_p$ for some $\kappa>0$. More generally, if $\|\xi\|_2\le C_{\xi}$ almost surely, the same proof yields the same bound up to an additional multiplicative factor $C_{\xi}$.

Assume the regression noise $\epsilon:=c-\mu(\xi)$ is conditionally mean-zero and $\sigma$-subgaussian
in every direction, i.e., for all $\lambda\in\RR$ and all $u\in\RR^d$ with $\|u\|_2=1$,
\[
\EE\!\left[\exp\!\bigl(\lambda\,u^\top \epsilon\bigr)\mid \xi\right]\ \le\ \exp(\lambda^2\sigma^2/2).
\]
Under the centered linear model \eqref{eq:centered_condmean_model}, let $\hat A_\mu$ be the (multi-response) OLS estimator based on $n_\mu$ i.i.d.\ samples
$\{(\xi_i,c_i)\}_{i=1}^{n_\mu}$ by regressing $y_i:=c_i-c_0$ on $\xi_i$, and define
\[
\hat\mu(\xi):=c_0+\hat A_\mu \xi,
\qquad
\varepsilon_\mu^2:=\EE_{\xi}\bigl[\|\hat\mu(\xi)-\mu(\xi)\|_2^2\bigr]
=\EE_\xi\bigl[\|(\hat A_\mu-A_\mu)\xi\|_2^2\bigr].
\]

\begin{lemma}
\label{lem:ols_bound_eps_mu}
Fix $\delta_\mu\in(0,1)$ and assume
\[
n_\mu\ \ge\ \frac{8}{\kappa}\log\!\frac{2p}{\delta_\mu}.
\]
Then, with probability at least $1-\delta_\mu$ over the regression sample, the OLS estimator is well-defined and
\begin{equation}
\label{eq:ols_bound_A_mu_fro}
\|\hat A_\mu-A_\mu\|_F
\ \le\
C_{\mathrm{reg}}\cdot\frac{\sigma}{\sqrt{\kappa}}
\sqrt{\frac{d\Bigl(p+\log\!\frac{4d}{\delta_\mu}\Bigr)}{n_\mu}},
\qquad\text{where one may take } C_{\mathrm{reg}}=4\sqrt{2}.
\end{equation}
Consequently,
\begin{equation}
\label{eq:ols_bound_mu_sup}
\sup_{\|\xi\|_2\le 1}\|\hat\mu(\xi)-\mu(\xi)\|_2
\ \le\
C_{\mathrm{reg}}\cdot\frac{\sigma}{\sqrt{\kappa}}
\sqrt{\frac{d\Bigl(p+\log\!\frac{4d}{\delta_\mu}\Bigr)}{n_\mu}},
\end{equation}
and, in particular,
\begin{equation}
\label{eq:ols_bound_eps_mu}
\varepsilon_{\mu}
\ \le\
C_{\mathrm{reg}}\cdot\frac{\sigma}{\sqrt{\kappa}}
\sqrt{\frac{d\Bigl(p+\log\!\frac{4d}{\delta_\mu}\Bigr)}{n_\mu}}.
\end{equation}
\end{lemma}

\begin{proof}
Write the regression sample in centered form $y_i=c_i-c_0=A_\mu\xi_i+\epsilon_i$, where
$\EE[\epsilon_i\mid \xi_i]=0$ and $\epsilon_i$ is $\sigma$-subgaussian in every direction.
Let
\[
\hat\Sigma\ :=\ \frac1{n_\mu}\sum_{i=1}^{n_\mu}\xi_i\xi_i^\top
\quad\text{and}\quad
X\in\RR^{n_\mu\times p}\ \text{be the design matrix with rows }\xi_i^\top.
\]
Then $X^\top X = n_\mu \hat\Sigma$.

This proof follows a standard non-asymptotic OLS argument: we first lower bound the minimum eigenvalue of the empirical Gram matrix via a matrix Chernoff bound \citep{tropp2012userfriendly}, and then control the self-normalized noise term $\|(X^\top X)^{-1/2}X^\top \epsilon^{(j)}\|_2$ using sub-Gaussian concentration together with an $\varepsilon$-net (sphere covering) argument \citep{vershynin2018hdp}.

\emph{Step 1: a lower bound on $\lambda_{\min}(\hat\Sigma)$.}
Set $Y_i:=\xi_i\xi_i^\top$, so each $Y_i\succeq 0$ and $\lambda_{\max}(Y_i)=\|\xi_i\|_2^2\le 1$ a.s.
Moreover,
\[
\EE[Y_i]\ =\ \EE[\xi\xi^\top]\ =\ \Sigma_\xi\ \succeq\ \kappa I_p.
\]
Applying the matrix Chernoff bound \citep[Theorem~1.1]{tropp2012userfriendly} with $\delta=1/2$ yields
\[
\PP\!\left(\lambda_{\min}\!\left(\sum_{i=1}^{n_\mu}Y_i\right)\le \frac12\,\lambda_{\min}\!\left(\sum_{i=1}^{n_\mu}\EE Y_i\right)\right)
\ \le\
p\left[\frac{e^{-1/2}}{(1/2)^{1/2}}\right]^{n_\mu\kappa}
\ \le\ p\,e^{-n_\mu\kappa/8}.
\]
Under the assumed sample size condition, $p\,e^{-n_\mu\kappa/8}\le \delta_\mu/2$; hence with probability at least
$1-\delta_\mu/2$,
\[
\lambda_{\min}(X^\top X)=n_\mu \lambda_{\min}(\hat\Sigma)\ \ge\ \frac{n_\mu\kappa}{2}.
\]
In particular, $X^\top X$ is invertible, and OLS is well-defined on this event.

\emph{Step 2: row-wise OLS coefficient error.}
Let $\beta_j^\top$ denote the $j$-th row of $A_\mu$ and $\hat\beta_j^\top$ the $j$-th row of $\hat A_\mu$.
Then the scalar response model for coordinate $j$ is
$y_{i,j}=\beta_j^\top \xi_i+\epsilon_{i,j}$, and the OLS error satisfies
\[
\hat\beta_j-\beta_j\ =\ (X^\top X)^{-1}X^\top \epsilon^{(j)},
\]
where $\epsilon^{(j)}:=(\epsilon_{1,j},\dots,\epsilon_{n_\mu,j})\in\RR^{n_\mu}$.

Define the normalized noise vector
\[
g_j\ :=\ (X^\top X)^{-1/2}X^\top \epsilon^{(j)}\in\RR^p.
\]
Conditioned on $X$, for any $u\in\RR^p$ with $\|u\|_2=1$,
\[
u^\top g_j
= u^\top (X^\top X)^{-1/2}X^\top \epsilon^{(j)}
= a^\top \epsilon^{(j)},
\quad\text{where } a:=X(X^\top X)^{-1/2}u\in\RR^{n_\mu}.
\]
Note $\|a\|_2^2 = u^\top (X^\top X)^{-1/2}X^\top X (X^\top X)^{-1/2}u = \|u\|_2^2=1$.
Since $\{\epsilon_{i,j}\}_{i=1}^{n_\mu}$ are independent and each is $\sigma$-subgaussian, we have for all $\lambda\in\RR$,
\[
\EE\!\left[\exp\!\bigl(\lambda\,u^\top g_j\bigr)\mid X\right]
= \prod_{i=1}^{n_\mu}\EE\!\left[\exp(\lambda a_i \epsilon_{i,j})\mid X\right]
\le \prod_{i=1}^{n_\mu}\exp(\lambda^2\sigma^2 a_i^2/2)
= \exp(\lambda^2\sigma^2/2).
\]
Thus, conditional on $X$, $u^\top g_j$ is $\sigma$-subgaussian for every unit vector $u$.

Let $\mathcal N$ be a $1/2$-net of the Euclidean unit sphere in $\RR^p$ with $|\mathcal N|\le 5^p$ \citep{vershynin2018hdp}. For any fixed $u\in\mathcal N$ and any $t>0$,
subgaussian tails imply
$\PP\bigl(|u^\top g_j|\ge \sigma\sqrt{2t}\mid X\bigr)\le 2e^{-t}$.
Taking a union bound over $\mathcal N$ and choosing
$t:=p\log 5 + \log(2/\delta_j)$ gives
\[
\PP\!\left(\max_{u\in\mathcal N}|u^\top g_j|\ge \sigma\sqrt{2t}\ \bigm|\ X\right)
\le |\mathcal N|\,2e^{-t}\ \le\ \delta_j.
\]
On the complementary event, the standard net argument yields $\|g_j\|_2\le 2\max_{u\in\mathcal N}|u^\top g_j|$,
so with conditional probability at least $1-\delta_j$,
\[
\|g_j\|_2\ \le\ 2\sigma\sqrt{2t}
=2\sigma\sqrt{2\Bigl(p\log 5+\log(2/\delta_j)\Bigr)}
\le 4\sigma\sqrt{p+\log(2/\delta_j)},
\]
where we used $\log 5\le 2$ and $\log(2/\delta_j)\ge 0$.

Set $\delta_j:=\delta_\mu/(2d)$, so $\log(2/\delta_j)=\log(4d/\delta_\mu)$.
Then, with probability at least $1-\delta_\mu/2$ over the noise (and conditional on $X$),
the above bound holds simultaneously for all $j=1,\dots,d$ by a union bound.

\emph{Step 3: combine and translate to prediction error.}
On the intersection of the events from Steps 1 and 2 (which has probability at least $1-\delta_\mu$),
for each $j$,
\[
\|\hat\beta_j-\beta_j\|_2
= \|(X^\top X)^{-1/2} g_j\|_2
\le \frac{\|g_j\|_2}{\sqrt{\lambda_{\min}(X^\top X)}}
\le \frac{4\sigma\sqrt{p+\log(4d/\delta_\mu)}}{\sqrt{n_\mu\kappa/2}}
= \frac{4\sqrt{2}\,\sigma}{\sqrt{\kappa}}
\sqrt{\frac{p+\log(4d/\delta_\mu)}{n_\mu}}.
\]
Therefore,
\[
\|\hat A_\mu-A_\mu\|_F^2=\sum_{j=1}^d \|\hat\beta_j-\beta_j\|_2^2
\le
\frac{32\,\sigma^2}{\kappa}\cdot \frac{d\bigl(p+\log(4d/\delta_\mu)\bigr)}{n_\mu}.
\]
Taking square roots yields \eqref{eq:ols_bound_A_mu_fro}. Moreover, for every $\xi$ with $\|\xi\|_2\le 1$,
\[
\|\hat\mu(\xi)-\mu(\xi)\|_2
=\|(\hat A_\mu-A_\mu)\xi\|_2
\le \|\hat A_\mu-A_\mu\|_F\,\|\xi\|_2
\le \|\hat A_\mu-A_\mu\|_F,
\]
which implies \eqref{eq:ols_bound_mu_sup}. Finally, since $\|\xi\|_2\le 1$ a.s.,
\[
\varepsilon_\mu^2
=\EE_{\xi}\bigl[\|(\hat A_\mu-A_\mu)\xi\|_2^2\bigr]
\le \|\hat A_\mu-A_\mu\|_F^2,
\]
which implies \eqref{eq:ols_bound_eps_mu}.
\end{proof}

\subsection{Stage-I representation error bound (Proof of Theorem~\ref{thm:misspec_plus_gen_condmean})}
\label{app:clo_condmean_stage1}

We first relate the regression error of $\hat\mu$ to the probability that the induced plug-in decision disagrees with the Bayes rule.
Fix $\eta>0$.  On the event that $\mu(\xi)$ lies at distance $>\eta$ from the cone boundary $\mathcal{B}_{\X}$, the optimal extreme point is constant
throughout the ball $\mathbb{B}(\mu(\xi),\eta)$.  Hence, if the regression estimate is $\eta$-accurate and the learned dataset $\hat\Dset$ is pointwise sufficient at
$\hat\mu(\xi)$, then the lifted predictor $\tilde\mu(\xi)$ induces the same unique decision as $\mu(\xi)$.
The lemma below formalizes this decomposition.

\begin{lemma}
\label{lem:regression_to_sufficiency_margin}
Assume $c\in\C$ almost surely, and let $\tilde\mu$ be defined in \eqref{eq:tilde_mu_def}.
Under Assumption~\ref{assump:cone_boundary_margin}, define for $\eta>0$
\[
\tau_\mu(\eta)
\ :=\
\PP_{\xi}\bigl[\|\hat\mu(\xi)-\mu(\xi)\|_2>\eta\bigr].
\]
Then, for any $\eta>0$,
\begin{equation}
\label{eq:reg_to_sufficiency_margin_tail}
\PP_{\xi}\bigl[x^\star(\tilde\mu(\xi))\neq x^\star(\mu(\xi))\bigr]
\ \le\
\PP_{\xi}\bigl[\hat\Dset\ \text{is not pointwise sufficient at }\hat\mu(\xi)\bigr]
+\tau_\mu(\eta)
+C_{\mathrm{marg}}\,\eta^{\alpha}.
\end{equation}
\end{lemma}
\begin{proof}
Define the events
\begin{align*}
A(\xi)&:=\{\hat\Dset\ \text{is pointwise sufficient at }\hat\mu(\xi)\},\\
B(\xi)&:=\{\|\hat\mu(\xi)-\mu(\xi)\|_2\le \eta\},\\
C(\xi)&:=\{\dist(\mu(\xi),\mathcal{B}_{\X})>\eta\},\\
D(\xi)&:=\{\tilde\mu(\xi)\notin \mathcal{B}_{\X}\}.
\end{align*}
Note that $\PP_\xi[D(\xi)^c]=0$ by Assumption~\ref{assump:cone_boundary_margin}.

\smallskip
\noindent\textbf{Step 1: $\tilde\mu(\xi)$ is fiber-equivalent to $\hat\mu(\xi)$ under $\hat\Dset$.}
Since $\hat\mu(\xi)\in\C$ a.s.\ and $\hat U$ spans $\hat\Dset$, we have $\hat U^\top(\hat\mu(\xi)-c_0)\in \hat U^\top(\C-c_0)$ and thus Lemma~\ref{lem:lifting_main} implies $\tilde\mu(\xi)\in\C$.
Moreover, using the definition of $\tilde\mu$ and the lifting operator
$\operatorname{lift}_{\hat U}(s)=c_0+\mathcal{L}_{\hat U}s$ (see \eqref{eq:lift_operator_main}), we have
\[
\hat U^\top(\tilde\mu(\xi)-c_0)
=
\hat U^\top\mathcal{L}_{\hat U}\,\hat U^\top(\hat\mu(\xi)-c_0)
=
\bigl(\hat U^\top \Sigma \hat U\bigr)\bigl(\hat U^\top \Sigma \hat U\bigr)^{-1}\hat U^\top(\hat\mu(\xi)-c_0)
=
\hat U^\top(\hat\mu(\xi)-c_0).
\]
Therefore, for any $q\in\hat\Dset\subseteq \spanop(\hat U)$ we can write $q=\hat U a$ for some $a\in\RR^t$, and thus
\[
q^\top\tilde\mu(\xi)-q^\top\hat\mu(\xi)
=
a^\top \hat U^\top\bigl(\tilde\mu(\xi)-\hat\mu(\xi)\bigr)
=
a^\top\Bigl(\hat U^\top(\tilde\mu(\xi)-c_0)-\hat U^\top(\hat\mu(\xi)-c_0)\Bigr)
=
0.
\]
In particular, $\tilde\mu(\xi)$ and $\hat\mu(\xi)$ lie in the same fiber
$\C\!\left(\hat\Dset, s(\hat\mu(\xi);\hat\Dset)\right)$.

\smallskip
\noindent\textbf{Step 2: On $A(\xi)\cap B(\xi)\cap C(\xi)\cap D(\xi)$, the oracle decisions coincide.}
On $A(\xi)$, pointwise sufficiency at $\hat\mu(\xi)$ means that there exists some decision $x^{\mathrm{ps}}(\xi)\in\X$ such that
\[
x^{\mathrm{ps}}(\xi)\in \X^\star(c')\qquad \forall\,c'\in \C\!\left(\hat\Dset, s(\hat\mu(\xi);\hat\Dset)\right).
\]
Since $\tilde\mu(\xi)$ is in the same fiber, we have $x^{\mathrm{ps}}(\xi)\in \X^\star(\tilde\mu(\xi))$ and also
$x^{\mathrm{ps}}(\xi)\in \X^\star(\hat\mu(\xi))$.

Next, on $B(\xi)\cap C(\xi)$ we have $\hat\mu(\xi)\in\mathbb{B}(\mu(\xi),\eta)$ while $\mathbb{B}(\mu(\xi),\eta)$ is contained in the interior
of a single normal cone.  Hence, the optimal extreme point is unique throughout this ball, and in particular
$\X^\star(\hat\mu(\xi))=\{x^\star(\hat\mu(\xi))\}$ and $x^\star(\hat\mu(\xi))=x^\star(\mu(\xi))$.

Finally, on $D(\xi)$ we have $\tilde\mu(\xi)\notin\mathcal{B}_{\X}$, so $\X^\star(\tilde\mu(\xi))=\{x^\star(\tilde\mu(\xi))\}$ is also a singleton.
Since $x^{\mathrm{ps}}(\xi)$ is optimal for both $\hat\mu(\xi)$ and $\tilde\mu(\xi)$, uniqueness forces
$x^\star(\tilde\mu(\xi))=x^{\mathrm{ps}}(\xi)=x^\star(\hat\mu(\xi))=x^\star(\mu(\xi))$.

\smallskip
\noindent\textbf{Step 3: Conclude by a union bound.}
Thus, on $A(\xi)\cap B(\xi)\cap C(\xi)\cap D(\xi)$ we have $x^\star(\tilde\mu(\xi))=x^\star(\mu(\xi))$, and therefore
\[
\PP_{\xi}\bigl[x^\star(\tilde\mu(\xi))\neq x^\star(\mu(\xi))\bigr]
\le
\PP_{\xi}[A(\xi)^c]+\PP_{\xi}[B(\xi)^c]+\PP_{\xi}[C(\xi)^c]+\PP_{\xi}[D(\xi)^c].
\]
The first term is the pointwise-sufficiency failure probability at $\hat\mu(\xi)$.
The second term is exactly $\tau_\mu(\eta)$.
The third term is controlled by Assumption~\ref{assump:cone_boundary_margin}, giving $\PP[C(\xi)^c]\le C_{\mathrm{marg}}\eta^\alpha$.
Finally, $\PP[D(\xi)^c]=0$ by Assumption~\ref{assump:cone_boundary_margin}, which yields \eqref{eq:reg_to_sufficiency_margin_tail}.
\end{proof}

\smallskip
Combining the tail-form transfer bound \eqref{eq:reg_to_sufficiency_margin_tail}, the certificate guarantee of Theorem~\ref{thm:certificate} for Algorithm~\ref{alg:cumulative}, and the bounded-design OLS control of Lemma~\ref{lem:ols_bound_eps_mu} yields the main finite-sample bound on the representation-induced error of Stage~I.

\begin{proof}[Proof of Theorem~\ref{thm:misspec_plus_gen_condmean}]
On the regression event from Lemma~\ref{lem:ols_bound_eps_mu}, equation~\eqref{eq:ols_bound_mu_sup} implies that
\[
\|\hat\mu(\xi)-\mu(\xi)\|_2\ \le\ r_{\mu,\delta_\mu}
\qquad\text{for every }\xi\text{ with }\|\xi\|_2\le 1.
\]
Since $\|\xi\|_2\le 1$ almost surely, this gives
\[
\tau_\mu\bigl(r_{\mu,\delta_\mu}\bigr)
=
\PP_{\xi}\bigl[\|\hat\mu(\xi)-\mu(\xi)\|_2>r_{\mu,\delta_\mu}\bigr]
=0.
\]
Applying \eqref{eq:reg_to_sufficiency_margin_tail} with $\eta=r_{\mu,\delta_\mu}$ therefore yields
\[
\PP_{\xi}\bigl[x^\star(\tilde\mu(\xi))\neq x^\star(\mu(\xi))\bigr]
\ \le\
\PP_{\xi}\bigl[\hat\Dset\ \text{is not pointwise sufficient at }\hat\mu(\xi)\bigr]
+
C_{\mathrm{marg}}\,r_{\mu,\delta_\mu}^{\alpha}.
\]
On the same regression event, Theorem~\ref{thm:certificate} applied to the pseudo-cost sample $\{\hat c_j\}_{j=1}^{n_{\disc}}$ gives, with probability at least $1-\delta$ over the discovery contexts,
\[
\PP_{\xi}\bigl[\hat\Dset\ \text{is not pointwise sufficient at }\hat\mu(\xi)\bigr]
\ \le\
\frac{4}{n_{\disc}}\left(6|T|+\log(e/\delta)\right).
\]
Combining the two displays and taking a union bound over the regression and discovery samples proves \eqref{eq:repr_term_mu}.

For the SPO misspecification bound, note that for any predictor $f$,
\[
R_{\mathrm{SPO}}(f)-R_{\mathrm{SPO}}(f_\star)
=
\EE_{\xi}\!\left[\mu(\xi)^\top x^\star(f(\xi))-\mu(\xi)^\top x^\star(\mu(\xi))\right]\ \ge\ 0,
\]
and for each $\xi$ the bracketed difference is at most $\omega_{\X}(\C)$ and is zero whenever the two decisions coincide.
Hence $R_{\mathrm{SPO}}(f)-R_{\mathrm{SPO}}(f_\star)\le \omega_{\X}(\C)\,\PP_{\xi}[x^\star(f(\xi))\neq x^\star(\mu(\xi))]$.

Apply this to the specific candidate $f=\tilde\mu$.
By \eqref{eq:tilde_mu_in_H} we have $\tilde\mu\in\mathcal{H}_{\hat U,t}$,
so the optimality of $\hat f_\star$ in $\mathcal{H}_{\hat U,t}$ yields
$R_{\mathrm{SPO}}(\hat f_\star)\le R_{\mathrm{SPO}}(\tilde\mu)$.
Combining these facts with \eqref{eq:repr_term_mu} gives \eqref{eq:stage1_misspec_loss}.
\end{proof}

\end{document}